\documentclass[french]{smfart}
\usepackage[utf8]{inputenc}
 \usepackage[T1]{fontenc} 
\usepackage{mathrsfs}
\usepackage{amssymb, amsmath, mathabx}
\usepackage[francais]{babel}
\usepackage{smfthm}
\usepackage{graphics}
\usepackage{hyperref} 
\usepackage{lmodern}
\usepackage{multicol}
\usepackage[shortlabels]{enumitem}
\usepackage{scalerel}
\usepackage{tikz-cd}
\makeindex
\DeclareUnicodeCharacter{00A0}{\relax}
\renewcommand{\leq}{\leqslant}
\renewcommand{\geq}{\geqslant}
\newcommand{\abs}[1]{\mathopen|#1\mathclose|}

\newcommand{\an}{^{\mathrm{an}}}
\newcommand{\ad}{^{\mathrm{ad}}}

\newcommand{\gma}[1]{(\mathbf G_{\mathrm m,k}\an)^{#1}}
\newcommand{\gmb}[2]{(\mathbf G_{\mathrm m,#1}\an)^{#2}}
\newcommand{\gmc}[1]{\mathbf G_{\mathrm m,k}^{#1}}
\newcommand{\sacc}[1]{\mathscr S^{\mathrm{acc}}(#1)}
\newcommand{\sel}[1]{\mathscr S^{\mathrm{el}}(#1)}
\newcommand{\scl}[1]{\mathscr S^{\mathrm{cl}}(#1)}

\newcommand{\hotimes}{\hat \otimes}
\newcommand{\hr}[1]{\mathscr H(#1)}
\newcommand{\hrt}[1]{\widetilde{\mathscr H(#1)}}

\newcommand{\gad}{^{\Gamma\text{-}\mathrm{ad}}}
\newcommand{\gadc}[1]{\overline{
#1^{\Gamma\text{-}\mathrm{ad}}}}

\newcommand{\perf}[1]{\widehat{#1^{1/p^\infty}}}
\newcommand{\spec}{{\rm Spec}\;}
\newcommand{\grot}[1]{_{\mathrm G}^{#1}}
\renewcommand{\phi}{\varphi}
\renewcommand{\epsilon}{\varepsilon}

\NumberTheoremsAs{subsubsection}
\SwapTheoremNumbers

\SetEnumerateShortLabel{1}{\textnormal{(\arabic{enumi})}}
\SetEnumerateShortLabel{s}{\textnormal{(\arabic{enumi}$^\ast$)}}
\SetEnumerateShortLabel{i}{\textnormal{(\roman{enumi})}}
\SetEnumerateShortLabel{a}{\textnormal{(\alph{enumi})}}
\SetEnumerateShortLabel{A}{\textnormal{(\Alph{enumi})}}
\SetEnumerateShortLabel{2}{\textnormal{(\arabic{enumii})}}
\SetEnumerateShortLabel{j}{\textnormal{(\roman{enumii})}}
\SetEnumerateShortLabel{b}{\textnormal{(\arabic{enumi}\alph{enumii})}}
\SetEnumerateShortLabel{c}{\textnormal{(\alph{enumii})}}
\SetEnumerateShortLabel{B}{\textnormal{(\Alph{enumii})}}

\newcommand{\dom}[1]{\mathrm D(#1)}
\newcommand{\domc}[1]{\mathrm D^{\mathrm c}(#1)}

\newcommand{\domaff}[1]{\mathrm D^{\mathrm{aff}}(#1)}

\newcommand{\A}{\mathbf A}

\newcommand{\N}{\mathbf Z_{\geq 0}}
\renewcommand{\P}{\mathbf P}
\newcommand{\Q}{\mathbf Q}
\newcommand{\R}{\mathbf R}
\newcommand{\Z}{\mathbf Z}
\newcommand{\Tau}{\mathrm T}
\newcommand{\rpos}{\R_{\scriptscriptstyle >0}} 
\newcommand{\rposn}[1]{\rpos^{#1}}
\title{Les squelettes accessibles d'un espace de Berkovich}
\author{Antoine Ducros}
\address{Sorbonne Université, Université Paris-Diderot, CNRS, Institut de Mathématiques de Jussieu-Paris
Rive Gauche, IMJ-PRG, 
F-75005, Paris, France
}
\email{antoine.ducros\at imj-prg.fr}
\urladdr{http://www.imj-prg.fr/$\sim$antoine.ducros}
\author{Amaury Thuillier}
\address{Institut Camille Jordan, Université de Lyon 1, 43 boulevard du 11 novembre 1918
69622 Villeurbanne cedex
France}
\email{thuillier\at math.univ-lyon1.fr}
\urladdr{http://math.univ-lyon1.fr/$\sim$thuillier}
\date{}

\begin{abstract}
Nous définissons une classe de squelettes sur les espaces de Berkovich, dits \textit{accessibles}, qui contient le squelette
standard de $\gma n$ pour tout $n$ et est 
préservée par G-recollement, par image réciproque sous un morphisme de dimension relative nulle, 
et par image directe sous un morphisme topologiquement propre en restriction au squelette concerné. 
\end{abstract}

\begin{altabstract}
We define a class of skeletons on Berkovich analytic spaces, which we call \textit{accessible}, which contains the 
standard skeleton of $\gma n$ for all $n$ and is preserved by  G-glueing, by taking the inverse image along
a morphism of relative dimension zero, and by taking the direct image  along
a morphism whose restriction to the involved skeleton is topologically proper. 
\end{altabstract}
\begin{document}

\maketitle

\tableofcontents

\setcounter{section}{-1}
\section{Introduction}
Parmi les différentes approches aujourd'hui disponibles en géométrie non archimédienne (points de vue de Tate, Raynaud, Berkovich, Huber\ldots), 
celle de Berkovich est la plus adaptée à l'étude des liens profonds entre les mondes ultramétrique et linéaire par morceaux.
(Précisons que nous adoptons
ici le point de vue \textit{multiplicatif} :
un polytope est une partie de $\rposn n$ qui s'identifie \textit{via} 
le logarithme à une union finie d'intersections finies de demi-espaces
définis au moyen de formes affines à parties linéaires rationnelles, et
ces polytopes sont les briques de base permettant de définir 
un espace
linéaire par morceaux abstrait
par un procédé de recollement 
un peu délicat -- voir \ref{pl-abstrait} pour les détails). 

En effet, un espace de Berkovich
$X$ contient pléthore de sous-ensembles sur lesquels la structure analytique
de $X$ induit une structure linéaire par morceaux
naturelle, que l'on 
appelle les \textit{squelettes} de $X$
(définition 
\ref{def-squelette}). 
 Le prototype d'un tel squelette est le sous-ensemble $S_n$ de $\gma n$ constitué des normes de Gau\ss~~
$\sum a_I T^I\mapsto \max \abs{a_I}r^I$ pour $r\in \rposn n$ : sa structure linéaire
par morceaux se déduit de celle de
$\rposn n$ \textit{via} l'homéomorphisme $\abs T\colon S_n\simeq \rposn n$. Précisons que
si $f\colon Y\to X$ est un morphisme d'espaces analytiques et si $\Sigma$ et $\Tau$ sont deux squelettes de 
$X$ et $Y$ respectivement tels que $f(\Tau)\subset \Sigma$ alors $f|_\Tau\colon \Tau \to \Sigma$
est \textit{automatiquement} linéaire par morceaux.

Les squelettes sont apparus dès les travaux fondateurs de Berkovich, pour décrire le type d'homotopie de certains espaces analytiques. 
Plus précisément il a démontré dans \cite{berkovich1999} puis \cite{berkovich2004} que si $X$ est un espace $k$-analytique compact 
et quasi-lisse possédant un modèle formel polystable $\mathfrak X$ 
(par exemple, si $X$ est une courbe projective lisse sur un corps algébriquement clos)
alors $X$ se rétracte par déformation sur l'un de ses squelettes, qui est naturellement homéomorphe au complexe simplicial dual
de la fibre spéciale de $\mathfrak X$. Plus récemment Hrushovski et Loeser ont prouvé dans \cite{hrushovski-l2016}
que si $X$ est une variété algébrique quasi-projective et lisse sur un corps
ultramétrique complet $k$ alors $X\an$ se rétracte par déformation sur l'un 
de ses squelettes compacts. Précisons que leurs méthodes n'ont rien à voir avec celles de Berkovich : ils utilisent de la théorie 
des modèles au sens de la logique sans faire intervenir le moindre modèle entier au sens de la géométrie arithmétique. 
Mentionnons
également que leurs techniques s'appliquent tout aussi bien si $X$ n'est pas supposé lisse, mais que
la rétraction se fait alors
sur un squelette \textit{généralisé}, qui est modelé localement 
sur des adhérences dans
$\R_{\scriptscriptstyle \geq 0}^n$
de polytopes de $\rposn n$. 

Les squelettes ont par ailleurs été étudiés d'un autre point de vue par le premier auteur de ce
texte dans les articles \cite{ducros2003b} et
\cite{ducros2012b}. Il a démontré que 
si $X$ est un espace analytique de dimension $\leq n$ et si $f$ est un  morphisme de $X$
vers $\gma n$, alors $f^{-1}(S_n)$ est un squelette de $X$ (vide dès que $\dim X<n$) ;
et que de plus une union finie de tels
squelettes de $X$ est encore un squelette. En utilisant cette dernière propriété, on prouve
(proposition \ref{gunion-squelettes}) que si $\Sigma$ est une partie de $X$ possédant une famille
($\Sigma_i)$ de sous-ensembles telle que :  
\begin{itemize}[label=$\diamond$] 
\item chaque $\Sigma_i$ est une partie linéaire par morceaux d'un squelette 
$f_i^{-1}(S_{d_i})$ où $f_i$ est un morphisme de $Y_i$ vers $\gma {d_i}$ pour un certain sous-espace analytique fermé
$Y_i$ purement de dimension $d_i$ d'un domaine analytique de $X$ ; 
\item les $\Sigma_i$ G-recouvrent $\Sigma$, ce qui veut dire que tout point $x$ de $\Sigma$ a un voisinage dans $\Sigma$ qui est réunion d'un nombre
fini de $\Sigma_i$ contenant $x$, 
\end{itemize}
alors $\Sigma$ est un squelette. Nous qualifierons de \textit{classique} un tel squelette. En pratique, tous les squelettes
rencontrés jusqu'alors (les buts des rétractions par déformation, ceux qui sont utilisés pour intégrer les formes différentielles
dans le travail en cours \cite{chambertloir-d-2012}, ceux  intervenant dans l'étude du treillis des fonctions tropicales menée à bien 
dans \cite{dhly2024}\ldots) étaient  de ce type, ce qui justifie la terminologie. 

Il résulte essentiellement des définitions et constructions 
si $X$ est un espace analytique
et si $\Sigma$ est une partie de $X$, 
les propriétés suivantes sont satisfaites (proposition \ref{recap-squelettes-classiques}) : 
\begin{enumerate}[a]
\item si $Z$ est un sous-espace analytique fermé d'un domaine analytique de $X$
et si $\Sigma \subset Z$,
alors $\Sigma$ est un squelette
classique de $Z$ si et seulement si
c'est un squelette classique de $X$ ; 
\item si $\Sigma$ est un squelette classique de $X$, les squelettes classiques
de $X$ contenus dans $\Sigma$ sont exactement les parties linéaires par morceaux de $\Sigma$ ; 
\item si $\Sigma$ possède un G-recouvrement par des
squelettes classiques de $X$, c'est un squelette classique de $X$ ; 

\item l'intersection de deux squelettes classiques est un squelette classique ; 
\item si $\Sigma$ est
un squelette classique de $X$ et si $f\colon Y\to X$ est un morphisme de dimension relative nulle en tout point de $f^{-1}(\Sigma)$ alors
$f^{-1}(\Sigma)$ est un squelette classique de $Y$. 
\end{enumerate}
(Remarquons que concernant la propriété (e) on ne peut pas espérer supprimer l'hypothèse de dimension 
relative nulle : si $f$ possédait des fibres de dimension $>0$ au-dessus de $\Sigma$, celles-ci n'auraient aucun
espoir, pour de simples raisons topologiques, de pouvoir être munies d'une structure linéaire par morceaux). 

L'objectif de ce travail était d'étudier la question suivante : \textit{la classe des squelettes 
classiques possède-t-elle également de bonnes propriétés de stabilité par images
\emph{directes} ? } Précisons tout de suite qu'espérer une telle stabilité sans 
hypothèse de compacité semble illusoire : considérer par
exemple
l'image du squelette $\coprod_{r\in 2^{\Q}}\{\eta_r\}$ de
$\coprod_{r\in 2^{\Q}}
\mathbf G_{\mathrm m,k}^{\mathrm{an}}$
sous la flèche canonique $\coprod_{r\in 2^\Q}
\mathbf G_{\mathrm m,k}^{\mathrm{an}}
\to \mathbf G_{\mathrm m,k}^{\mathrm{an}}$.

Nous nous sommes donc intéressés à l'image
d'un squelette classique 
$\Sigma\subset Y$ sous un morphisme $f\colon Y\to X$ dont la restriction à $\Sigma$ est topologiquement propre ; 
le résultat principal de cet article (théorème \ref{image-squelette-glocalelem}) assure que $f(\Sigma)$ est alors
un squelette. Nous ne savons pas s'il est classique, et penchons pour la négative
en général, même si expliciter un contre-exemple semble délicat ; précisons toutefois que si $f(\Sigma)$ est constitué uniquement de points
de rang d'Abhyankar égal à la dimension ambiante, il est classique : c'est l'assertion (4) du théorème \ref{image-squelette-glocalelem}.

Les parties d'un espace de Berkovich $X$ qui sont 
 G-recouvertes par des sous-ensembles de la forme $f(\Sigma)$ avec $f$ et $\Sigma$ comme ci-dessus sont alors des squelettes 
 que nous qualifions d'\emph{accessibles}, et
 la classe des squelettes accessibles satisfait essentiellement par construction les propriétés
 (a) à (e) ci-dessus, ainsi que la stabilité par image directe sous un morphisme dont la restriction
 au squelette concerné est topologiquement propre (théorème \ref{theo-accessibles})

Disons maintenant quelques mots de la preuve
du théorème \ref{image-squelette-glocalelem}. 
En raisonnant G-localement
on voit tout d'abord qu'on peut
supposer que $X$ et $Y$ sont affinoïdes, que $Y$ est purement 
de dimension $n$ pour un certain $n$, et que $\Tau$ est égal à 
$g^{-1}(S_n)$ pour un certain $g\colon Y\to \gma n$. 
Le théorème 4.3 de 
\cite{ducros2026}, qui repose lui-même
sur la déclinaison non archimédienne \cite{ducros2021b}
des techniques d'aplatissement de Raynaud et Gruson, 
permet de se ramener au cas où $f$ est composé d'un morphisme plat
et équidimensionnel, d'une immersion fermée et d'un morphisme quasi-étale, et il n'y a finalement 
plus que les deux cas particuliers suivants à traiter. 

\paragraph*{Premier cas}
C'est celui où $X$ est purement de dimension $m$ pour un certain $m$, et où $f$ est plat, 
ses fibres étant donc purement de dimension $n-m$. On montre alors
que $f(\Tau)$ est un squelette classique de $X$. 

Pour ce faire, considérons un
point $t$
de
$\Tau$ ; posons $s=f(t)$. 
Puisque $\Tau$
est égal à $g^{-1}(S_n)$, le corps résiduel complété $\hr t$ de $t$ 
est une extension d'Abhyankar de rang $n$ de
$k$, et par un simple argument de dimension 
on en déduit que $\hr s$ est d'Abhyankar de rang $m$ sur $k$, et que
$\hr t$ est d'Abhyankar de rang $n-m$ sur $\hr s$. Il existe alors un 
voisinage affinoïde $V$  de $s$ dans $X$ et un voisinage affinoïde $W$ de $t$ dans $f^{-1}(V)$, 
ainsi que $m$ fonctions
inversibles $h_1,\ldots, h_m$ sur $V$
et $n-m$ fonctions inversibles $h_{m+1},\ldots, h_n$ sur $W$ telles que les $h_i(t)$
forment une base d'Abhyankar de $t$ sur $k$. Autrement dit, si $h\colon W\to \gma n$ désigne
le morphisme induit par les $h_i$ le point $t$ appartient à $h^{-1}(S_n)$. L'intersection 
$\Tau\cap h^{-1}(S_n)$ est un squelette classqiue de $W$ qui contient $t$ et dont l'image
par $f$ est contenue dans le squelette classique $(h_1,\ldots, h_m)^{-1}(S_m)$ de $X$. 
Ainsi nous avons exhibé pour tout point $t\in \Tau$ une partie linéaire par morceaux compacte $\Upsilon_t$
de $\Tau$ et un squelette classique compact $\Sigma_t$ de $X$
tel que $f(\Upsilon_t)\subset \Sigma_t$. L'idée est alors la suivante : on souhaiterait recouvrir $\Tau$ par une
famille \textit{finie}
$(\Upsilon_t)_{t\in E}$ ; cela garantirait que $f(\Tau)$ est contenue dans le squelette classique
$\Sigma:=\bigcup_{t\in E}\Sigma_t$ de $X$ ; l'application $f|_\Tau \colon\Tau \to \Sigma$ serait alors linéaire par morceaux, 
et $f(\Tau)$ serait ainsi une partie linéaire par morceaux de $\Sigma$, et partant
un squelette classique de $X$. 
Mais comme $\Upsilon_t$ n'est pas \textit{a priori} un voisinage de $t$
dans $\Tau$ (il pourrait
même très bien être réduit à $\{t\}$), cette approche 
semble vouée à l'échec. Elle fonctionne cependant, à condition de faire varier $t$
non pas simplement sur $\Tau$ mais sur le sous-ensemble correspondant 
$\Tau^{\mathbf R\text{-}\mathrm{ad}}$ de \textit{l'espace de Huber-Kedlaya
$X^{\mathbf R\text{-}\mathrm{ad}}$} associé à $X$, le point crucial
étant que
$\Tau^{\mathbf R\text{-}\mathrm{ad}}$ est compact et que les $\Upsilon_t^{\mathbf R\text{-}\mathrm{ad}}$
sont des ouverts compacts de
$\Tau^{\mathbf R\text{-}\mathrm{ad}}$ (pour la topologie constructible, qui suffit ici), ce qui entraîne la finitude attendue. Cette étape demande
évidemment une étude préalable de l'avatar des squelettes dans le contexte réel-adique, qui fait l'objet des sections 
\ref{friable-general}, \ref{friable-analytique} et \ref{friable-polytope} ; indiquons que récemment, ces «squelettes adiques» ont été considérés, totalement 
indépendamment, par Hübner et Temkin dans \cite{hubner-temkin2024}.

\paragraph*{Second cas}
C'est celui où $f$ est composée d'une immersion fermée $Y\hookrightarrow Z$ et d'un 
morphisme quasi-étale $Z\to X$ (où $Z$ est compact). Un raisonnement G-local permet 
de se ramener au cas où $Z$ est un domaine analytique compact d'un revêtement fini galoisien 
$X'$ de $X$, dont on note $G$ le groupe. Quitte à remplacer $\Tau$ par $\bigcup_{g\in G}g(\Tau)$
(ce qui ne modifie pas son image) on peut 
le supposer $G$-invariant. Son image s'identifie alors topologiquement au quotient $\Tau/G$, 
lequel hérite d'une structure linéaire par morceaux naturelle qui en fait le quotient de
$\Tau$ par $G$ dans la catégorie linéaire par morceaux (théorème \ref{quotient-fini-pl}). 
En utilisant le fait que $\Tau$ est un squelette et que $\Tau\to \Tau/G$ est une immersion par
morceaux, on montre sans grande difficulté (en ayant à nouveau recours aux squelettes réels-adiques
pour pouvoir utiliser des arguments de compacité)
réels-adiques de voir que $\Tau/G$ est un squelette -- mais rien n'indique par contre
qu'il soit classique. 

\bigskip
Nous allons terminer cette introduction par quelques remarques.
\begin{itemize}[label=$\diamond$]
\item Soit $f\colon Y\to X$ un morphisme entre espaces compacts et soit $\Tau$
un squelette fermé de $Y$. Comme expliqué
ci-dessus, nous montrons que sous certaines hypothèses (si $\Tau$ est classique
ou plus généralement accessible) l'image $f(\Tau)$ est un squelette, ce qui entraîne automatiquement
que $f|_\Tau\colon \Tau \to f(\Tau)$  est linéaire par morceaux, et en particulier que la relation d'équivalence sur $\Tau$ induite par
$f$ est linéaire par morceaux. Or ce fait n'est pas du tout évident \textit{a priori}, et pourrait 
tomber en défaut pour des squelettes généraux. 
En effet, supposons pour simplifier
que $Y$ et $X$ sont affinoïdes et irréductibles, que $f$ est dominant et que $\Tau$
est constitué de points Zariski-dense.  Deux
points $t$ et $t'$ de $\Tau$ ont alors même
image sur $X$ si et seulement si pour toute fonction analytique non nulle $g$ sur $X$
on a $\abs{g(t)}=\abs{g(t')}$ ; or pour une telle $g$ fixée, la condition $\abs{g(t)}=\abs{g(t')}$
définit une relation d'équivalence linéaire par morceaux sur $\Tau$, si bien que la relation d'équivalence induite par $f$ sur $\Tau$ est
une intersection \textit{a priori} infinie de relations linéaires par morceaux. Notre résultat peut donc être vu comme un théorème de finitude
pour une classe raisonnable de squelettes.

\item Pour simplifier l'exposition dans cette introduction, nous n'avons pas évoqué la question des \textit{paramètres de définition}
des espaces linéaires par morceaux et des espaces analytiques en jeu, mais nous y prêtons attention dans le texte, ce qui alourdit
un peu la terminologie. En géométrie linéaire
par morceaux générale nous fixons ainsi un couple $c=(K,\Delta)$ où $K$ est un sous-corps de $\R$ et $\Delta$ un sous-$K$-espace vectoriel 
(pour l'exponentiation) de $\rpos$, et nous travaillons avec des espaces $c$-linéaires par morceaux, c'est-à-dire modelés sur des parties de
$\rposn n$ définies par des inégalités larges entre monômes dont les exposants appartiennent à $K$ et les termes constants à $\Delta$. 
En géométrie analytique, nous fixons un corps ultramétrique complet $k$ et un sous-groupe $\Gamma$ de $\rpos$ qui est non trivial si
$\abs{k^\times}=\{1\}$, et nous travaillons uniquement avec des espaces $k$-analytiques $\Gamma$-stricts, c'est-à-dire qui peuvent être définis
au moyen de paramètres réels appartenant à $\Gamma$ (ainsi tout espace $k$-analytique est $\rpos$-strict, et si $\abs{k^\times}\neq \{1\}$, un espace
$k$-analytique est $\{1\}$-strict si et seulement s'il est strict) ; et nous utilisons également des $\Gamma$-variantes des espaces de 
Huber-Kedlaya et des réductions de germes à la Temkin. Enfin, dans le contexte des espaces $k$-analytiques
$\Gamma$-stricts nous posons $c=(\Q, (\abs{k^\times}\cdot \Gamma)^\Q)$ et nous nous intéressons aux $c$-squelettes, qui ont une
structure $c$-linéaire par morceaux naturelle. 

\item Nous avons mentionné plus haut, à propos des rétractions par déformations qu'ont construites Hrushovski et Loeser, les
\textit{squelettes généralisés}, modelés non pas sur des polytopes de $\rposn n$, mais sur les adhérences de tels polytopes dans 
$\R_{\scriptscriptstyle \geq 0}^n$. Nous travaillons actuellement sur une
extension du présent travail au cas de ces squelettes généralisés, 
qui est techniquement beaucoup plus délicat ; disposer d'une classe de squelettes généralisés se comportant bien par G-recollement, images réciproques
et images directes devrait être très utile à terme pour l'étude du type d'homotopie des espaces de Berkovich généraux ; cet article constitue un premier
pas dans cette direction.

\end{itemize}

\subsection*{Remerciements}
Nous faisons part de toute notre reconnaissance
au rapporteur anonyme pour ses remarques et suggestions très pertinentes à propos d'une première
version de  cet article. 

Le premier auteur a bénéficié pendant une partie de ce travail du soutien de l'ANR 
à traver le projet AdAnAr (ANR-24-CE40-6184). 

\subsection*{Rôle éventuel de l'intelligence artificielle}
Nous n'avons eu à un aucun moment recours à l'intelligence artificielle. 

\medskip
\textit{Ce travail a été amorcé lors de discussions entre les deux auteurs, et le second est très reconnaissant au premier de l'avoir mené à son terme.}

\section{Valuations colorées}

\subsection{Conventions générales relatives aux valuations}
Soit $A$ un anneau
(commutatif unitaire).

\subsubsection{}
Suivant en cela la terminologie utilisée par Huber en
théorie des espaces adiques,
nous appellerons valuation sur $A$ une application $\phi \cdot 
\colon
A\to G\cup\{0\}$ 
où $G$ est un groupe abélien ordonné noté multiplicativement et où $0$ est un plus petit
élément absorbant qu'on adjoint à $G$, qui possède les propriétés suivantes : 
\begin{itemize}[label=$\diamond$]
\item $\phi(1)=1$  ; 
\item $\phi (ab)=\phi (a) \phi(b)$ et 
$\phi(a+b)\leq \max(\phi(a), \phi(b))$ 
pour tout $(a,b)\in A^2$.
\end{itemize}

On dit que deux valuations $\phi_1 \colon A\to G_1\cup\{0\}$ et $\phi_2\colon A\to G_2\cup\{0\}$ 
sont équivalentes s'il existe une valuation 
$\phi \colon A\to G\cup\{0\}$ et deux morphismes injectifs
croissants 
$u_1\colon G\hookrightarrow G_1$ 
et $u_2\colon G\hookrightarrow G_2$ 
tels que le diagramme 

\[
\begin{tikzcd}
&&&G_1\cup\{0\}\\
\\
A\ar[rr,"\phi"]\ar[uurrr, "\phi_1"]
\ar[ddrrr,"\phi_2"']&&G\cup\{0\}\ar[uur, "u_1"']\ar[ddr, "u_2"]&\\
\\
&&&G_2\cup\{0\}
\end{tikzcd}
\]
commute. 

\subsubsection{}
Soit $\phi \colon A\to G\cup\{0\}$ une valuation. Son noyau est un idéal 
premier de $A$ appelé le \textit{support} de $\phi$. Si $\xi$ est un point de
$\spec A$, se donner une valuation sur $A$ de support $\xi$ revient à se donner une
valuation sur le corps résiduel $\kappa(\xi)$ (dans un sens, considérer la valuation induite
sur $\kappa(\xi)$ par passage au quotient ; dans l'autre, composer la valuation donnée sur 
$\kappa(\xi)$ avec l'évaluation en $\xi$). Deux valuations sur $A$ sont équivalentes si
et seulement si elles ont même support et si les valuations induites
sur le corps résiduel correspondant sont équivalentes. 

\subsubsection{Notations}
Pour une valuation sur un corps nous utiliserons le plus souvent le symbole $\abs \cdot$. 
Sur un anneau 
quelconque nous préfèrerons en général une notation plus géométrique, 
pour penser aux valuations comme à des points. Ainsi nous pourrons typiquement 
désigner par $x$ une valuation sur $A$ ; 
et nous noterons alors $\kappa(x)$ le corps résiduel du support de $x$ \textit{muni de la
valuation induite par $x$}, valuation que nous noterons $\abs \cdot$. Le morphisme
d'évaluation de $A$ vers son corps résiduel $\kappa(x)$ sera noté $a\mapsto a(x)$ ; avec ces
conventions, on a $x(a)=\abs {a(x)}$ pour tout $a\in A$, et 
nous opterons  évidemment pour la seconde écriture sauf exception. Remarquons que le groupe ordonné $\abs{\kappa(x)^\times}$
et la flèche $A\to \abs{\kappa(x)}, a\mapsto \abs{a(x)}$ sont déterminés canoniquement
(à unique isomorphisme près) par
la classe d'équivalence de $x$. 

\subsubsection{Fonctorialité}
Soit $\phi\colon A\to B$ un morphisme d'anneaux
et soit $y$ une valuation sur $B$. 
La flèche $A\mapsto 
\abs{\kappa(y)}, a\mapsto \abs{\phi(a)(y)}$
définit une valuation sur $A$, dont la classe
d'équivalence ne dépend que de celle de $y$. 

\subsubsection{Reconstitution d'une valuation}\label{valuation-reconstitution}
Une valuation 
$\abs \cdot$ sur un corps 
$K$ se reconstitue à équivalence près à partir de son anneau
$\{z\in K, \abs z\leq 1\}$. Il s'ensuit qu'une valuation $x$ sur $A$ se reconstitue à équivalence près 
à partir de
l'ensemble
$E:=\{(f,g)\in A^2, \abs {f(x)}\leq \abs {g(x)}\}$. En effet,
$E$ permet tout d'abord de retrouver 
le support $\{f, f(x)=0\}$ de $x$, qui se récrit $\{f\in A, (f,0)\in E\}$. La valuation résiduelle sur le corps 
$\kappa(x)$ a alors pour anneau
$\{f(x)/g(x)\}_{(f,g)\in E, g(x)\neq 0}$. 

\subsection{Valuations colorées}
Soit $\Delta$ un groupe abélien ordonné, soit $k$
un corps muni d'une $\Delta$-valuation, c'est-à-dire d'une valuation $\abs \cdot\colon k\to \Delta\cup\{0\}$, 
et soit $A$ une $k$-algèbre.

\subsubsection{Groupes colorés}
Un  groupe abélien ordonné \textit{$\Delta$-coloré}
est un groupe abélien ordonné $G$ muni
d'un morphisme injectif croissant $\Delta\hookrightarrow G$. 
Un groupe abélien ordonné $\Delta$-coloré $G$ 
est dit \textit{bridé}
si l'enveloppe
convexe de l'image de $\Delta\hookrightarrow G$
est égale à $G$ tout entier. 

Un morphisme $G\to G'$ de groupes abéliens ordonnés $\Delta$-colorés est 
est un morphisme de groupes croissant
tel que le diagramme
\[\begin{tikzcd}
G\ar[rr]&&G'\\
&\Delta\ar[ul,hook']\ar[ur,hook]&\end{tikzcd}\]
commute.

\subsubsection{Valuations colorées}
Une \textit{$k$-valuation $\Delta$-colorée} sur $A$ est la donnée d'un groupe abélien ordonné
$\Delta$-coloré $G$ et d'une valuation 
$A\to G\cup\{0\}$ 
telle que le diagramme 
\[\begin{tikzcd}
A\ar[r]&G\cup\{0\}\\
k\ar[u]\ar[r,"\abs \cdot"]&\Delta\cup\{0\}\ar[u, hook]\end{tikzcd}\]
commute.

On dit que deux $k$-valuations
$\Delta$-colorées $\phi_1 \colon A\to G_1\cup\{0\}$ et $\phi_2\colon A\to G_2\cup\{0\}$ 
sont équivalentes
s'il existe une $k$-valuation 
$\Delta$-colorée $\phi \colon A\to G\cup\{0\}$ 
et deux morphismes
$u_1\colon G\hookrightarrow G_1$ 
et $u_2\colon G\hookrightarrow G_2$ 
de groupes abéliens ordonnés $\Delta$-colorés 
tels que
l'on ait $u_1\circ \phi=\phi_1$ et $u_2\circ \phi=\phi_2$. Si $\xi$ est un point de
$\spec A$, se donner une $k$-valuation
$\Delta$-colorée sur $A$
de support $\xi$ revient à se donner une
$k$-valuation $\Delta$-colorée sur le corps résiduel $\kappa(\xi)$, 
et deux $k$-valuations $\Delta$-colorées sur $A$ 
sont équivalentes si
et seulement si elles ont même support et si les $k$-valuations 
$\Delta$-colorées 
sur le corps résiduel correspondant sont équivalentes. 

\subsubsection{}
Soit $x$ une $k$-valuation $\Delta$-colorée sur $A$, à valeurs dans
$G\cup\{0\}$ pour un certain groupe abélien ordonné $\Delta$-coloré $G$. 
Le morphisme structural $\Delta\hookrightarrow G$ permet de voir $\Delta$
comme un sous-groupe de $G$. Le groupe abélien ordonné $\Delta$-coloré 
$\abs{\kappa(x)}^\times \cdot \Delta\subset G$
et la flèche $A\to \abs{\kappa(x)}\cdot \Delta, a\mapsto a(x)$ sont alors déterminés 
à unique isomorphisme près par la classe de $x$. 
Dans ce contexte, la notation $\abs \lambda$ pour $\lambda\in k^\times$
pourrait apparaître
ambiguë, puisque pouvant désigner aussi bien désigner l'image de $\lambda$ dans
$\abs{\kappa(x)^\times}$ sous la valuation structurale de $\kappa(x)$ (le corps $k$
étant identifié naturellement à un sous-corps de $\kappa(x)$) que son image dans $\Delta$
par la valuation structurale de $k$. Mais il n'en est rien : la définition même d'une $k$-valuation $\Delta$-colorée
assure que ces deux images sont égales lorsqu'on les voit comme appartenant à $\abs{\kappa(x)}^\times \cdot \Delta$. 
Si $\Gamma$ est un sous-groupe de $\Delta$ l'engendrant modulo $\abs{k^\times}$ on a
$\abs{\kappa(x)^\times}\cdot \Gamma=\abs{\kappa(x)^\times}\cdot \Delta$. 
Nous dirons qu'une $k$-valuation $\Delta$-colorée $x$ sur $A$
est bridée si le
groupe abélien ordonné $\Delta$-coloré
$\abs{\kappa(x)^\times}\cdot \Delta$ 
est bridé.

\subsubsection{}
Soit $\phi\colon A\to B$ un morphisme de $k$-algèbres
et soit $y$ une $k$-valuation $\Delta$-colorée
sur $B$. 
La flèche $A\mapsto 
\abs{\kappa(y)}\cdot \Delta, a\mapsto \abs{\phi(a)(y)}$
définit une $k$-valuation $\Delta$-colorée
sur $A$, dont la classe
d'équivalence ne dépend que de celle de $y$ ; si $y$ est bridée, 
$x$ l'est aussi.

\subsubsection{Reconstitution d'une valuation colorée}\label{reconstituer-valuation}
Soit $K$ une extension de $k$, et soit $\Gamma$
un sous-groupe de $\Delta$ l'engendrant modulo 
$\abs{k^\times}$. 
Une $k$-valuation 
$\Delta$-colorée sur $K$ se reconstitue à partir des ensembles $\{z\in K, \abs z \leq\gamma\}$ 
pour $\gamma$ parcourant $\Gamma$. Par conséquent, une valuation
$\Delta$-colorée $x$ sur un anneau $A$ se reconstitue à équivalence près 
à partir de
\[\{(f,g,\gamma)\in A^2\times \Gamma, \abs {f(x)}\leq \gamma \abs {g(x)}\},\]
par un raisonnement analogue à celui
tenu au \ref{valuation-reconstitution}. 

\subsection{Raffinements des valuations}
Soit $K$ un corps. 

\subsubsection{}
Soient $\abs \cdot $ et $\abs \cdot '$ deux valuations sur $K$. On dit que 
$\abs \cdot '$ raffine $\abs \cdot$ si l'anneau de $\abs \cdot '$ est contenu dans
celui de $\abs \cdot$. Il revient au même de demander qu'il existe un morphisme croissant 
(nécessairement unique)
$\pi \colon \abs {F^\times}'\to \abs{F^\times}$ tel que $\abs \cdot =\pi \circ \abs\cdot '$. 

\subsubsection{}
Décrivons maintenant la variante colorée de cette notion. Soit $\Delta$ un groupe abélien ordonné,
et soit $k$ un sous-corps de $K$ muni d'une $\Delta$-valuation. Soient $\abs \cdot$ et $\abs \cdot'$
deux $k$-valuations $\Delta$-colorées sur $K$. On dit que $\abs \cdot '$ raffine $\abs \cdot$ si
pour tout 
$\delta \in \Delta$, l'ensemble $\{z\in K, \abs z'\leq \delta\}$ est contenu dans
 $\{z\in K, \abs z\leq \delta\}$ ; ctte dernière condition revient à demander qu'il existe un morphisme 
$\pi \colon \abs {F^\times}'\cdot \Delta\to \abs{F^\times}\cdot \Delta$ de groupes
abéliens ordonnés $\Delta$-colorés (nécessairement unique) 
tel que $\abs \cdot =\pi \circ \abs\cdot '$.

\section{Algèbre graduée et corpoïdes résiduels}\label{def-red-graduees}

\subsection{}
Nous utiliserons dans la suite le formalisme de l'algèbre graduée
introduite par Temkin dans
\cite{temkin2004}, et plus précisément sa variante présentée à la section 1 de  \cite{ducros2021b}.
La différence 
entre notre point de vue et celui de Temkin est que pour éviter de devoir systématiquement préciser que les 
éléments manipulés sont homogènes nous choisissons d'emblée de ne considérer \textit{que} des éléments homogènes. 
Pour ce faire, au lieu de travailler avec les anneaux gradués classiques, nous utiliserons 
ce que nous appelons les \textit{annéloïdes}. 

\subsubsection{}
Si $D$ est un groupe abélien
un \textit{annéloïde $D$-gradué} est un
ensemble $A$ muni d'une décomposition comme \textit{union disjointe}
$A=\coprod_{d\in D}A^d$, 
d'une loi de groupe abélien sur chacun des $A^d$ notée additivement
et d'élément neutre $0^d$, 
et d'une loi de composition interne
$A\times A\to A$ associative, commutative et possédant un élément 
neutre $1\in A^1$, qui
induit pour tout couple $(d,d')$ d'éléments de $D$  une application 
bi-additive de $A^d\times A^{d'}$ vers $A^{dd'}$. Le degré d'un élément 
$a$ de $A$ est l'unique élément de $D$
tel que $a\in A^d$ ; il arrivera qu'on écrive $0$ au lieu de
$0^d$ si le degré en jeu est clair. 
Un morphisme d'annéloïdes
$D$-gradués $f\colon A\to B$
est une application $f$ préservant les degrés, 
commutant aux deux lois et envoyant 
$1$ un sur $1$.

\subsubsection{}
Un annéloïde (tout court) est un couple constitué
d'un groupe abélien $D$ et d'un annéloïde $D$-gradué $A$. 
Pour faire des annéloïdes une catégorie suffisamment souple, 
nous allons procéder en deux temps. Définissons tout d'abord un 
morphisme \textit{stricto sensu} entre deux annéloïdes $(D,A)$ et 
$(E,B)$ comme la donnée d'un morphisme injectif de groupes
$i\colon D\to E$ et d'une application $f$ de $A$ dans $B$ envoyant
$A^d$ sur $B^{i(d)}$ pour tout $d$, commutant aux deux lois en envoyant
$1$ sur $1$. La notion de morphisme \textit{stricto sensu}
fait des annéloïdes une catégorie $\mathsf C$, mais
ce que nous appellerons la catégorie des annéloïdes sera la
localisée de $\mathsf C$ par rapport à la classes des flèches
$(i\colon D\hookrightarrow E, f\colon A\to B)$ possédant la
propriété suivante : $f$ est injective d'image
$\coprod_{d\in D}B^{i(d)}$ et $B^e=\{0^e\}$ pour tout $e\in E
\setminus i(D$). 

Ainsi, si $E$ est un groupe abélien et $D$ un sous-groupe de $E$, 
un annéloïde $E$-gradué $B$ tel que $B^e=\{0^e\}$ pour tout
$e\in E\setminus D$ est isomorphe en tant qu'annéloïde à l'annéloïde
$D$-gradué $\coprod_{d\in D}B^d$, et un annéloïde $D$-gradué $A$
est isomorphe en tant qu'annéloïde à l'annéloïde $E$-gradué
$\coprod_{e\in E}A^e$ où l'on a posé $A^e=\{0^e\}$ si $e\notin D$. 
Dans le premier cas nous dirons simplement que nous voyons $B$
comme un annéloïde $D$-gradué, et dans le second cas que nous voyons $A$
comme un annéloïde $E$-gradué. 

Tout morphisme d'annéloïdes $(A,D)\to (B,E)$ le long d'un morphisme injectif $i\colon D\hookrightarrow E$
peut
par ce biais être vu comme un morphisme d'annéloïdes $E$-gradués : on utilise la réciproque
de l'isomorphisme $i\colon D\to i(D)$ pour voir $A$ comme un annéloïde $i(D)$-gradué, puis on 
le voit comme un annéloïde $E$-gradué en le prolongeant par $0$ comme décrit ci-dessus.

\subsubsection{}
Nombre de notions et résultats de l'algèbre commutative
admettent des avatars gradués naturels ; nous renvoyons
à la section 1 de \cite{ducros2021b}
pour davantage de détails. Nous allons simplement
évoquer ici quelques uns de ces avatars qui nous serons particulièrement utiles. Remarquons
par ailleurs qu'un annéloïde $\{1\}$-gradué
est
simplement un anneau : l'algèbre commutative graduée contient
donc l'algèbre commutative standard. 

Un \textit{corpoïde}
est un annéloïde dans lequel $1\neq 0$ et  dont tout élément non nul est inversible.

\subsubsection{}
Il y a une notion d'algèbre de polynômes dans le contexte des annéloïdes, 
décalquée de la notion classique mais avec une petite complication (pour l'essentiel indolore) : on 
doit prescrire les degrés des indéterminées. Si $(r_i)_{i\in I}$ est une famille d'éléments d'un groupe
abélien $D$
et si $A$ est un annéloïde $D$-gradué,
on dispose ainsi d'un
annéloïde $D$-gradué
$A[r_i\backslash T_i]$ qui 
au sein des $A$-algèbres $D$-graduées, 
représente
le foncteur covariant qui envoie $B$ sur l'ensemble des familles
$(b_i)$ d'éléments de $B$, chaque $b_i$ devant être de degré $r_i$. 

\subsubsection{}\label{degtrans-corpoides}
Une fois ces annéloïdes de polynômes à notre
disposition, on peut développer une théorie
de l'indépendance algébrique et du degré de transcendance dans le cadre
des extensions de corpoïdes $D$-gradués

Plus précisément
soit $D$ un groupe abélien et soit $K\hookrightarrow L$
une extension de
corpoïdes $D$-gradués. Soit $E$ un sous-groupe de $D$. 
Soit $(t_i)$ une base de transcendance de $L^E$ sur $K^E$ et soit $(\lambda_j)$ une famille
d'éléments de $L^\times$ telle que les $\deg(\lambda_i)$ constituent une base
de \[\Q\otimes_\Z (\deg (L^\times)/((E\cap \deg(L^\times)\cdot \deg (K^\times)).\] Un calcul
immédiat montre que $(t_i)$ est encore une base de transcendance de $K\cdot L^E$ sur $K$, et que 
$(\lambda_j)$ est une base de transcendance de $L$ sur $K\cdot L^E$. Par conséquent, 
la concaténation de $(t_i)$ et $(\lambda_j)$ est une base de transcendance de $L$ sur $K$. 
En appliquant ce qui précède avec $E=\{1\}$, on voit que le degré de transcendance de 
$L$ sur $K$ est égal à la somme du degré de transcendance classique de $L^1$ sur $K^1$ et du rang
rationnel de $\deg(L^\times)/\deg(K^\times)$.

\subsection{Valuations et corpoïdes résiduels}\label{hypotheses-generales}
Il y a une notion de valuation sur un corpoïde (et même sur un annéloïde, mais nous n'en aurons pas besoin ici), définie 
comme dans le cas classique : si $G$ est un groupe abélien ordonné, une valuation sur un corpoïde $F$ est une application 
$\abs \cdot\colon F\to G\cup\{0\}$ telle que $\abs a=0$ si et seulement si $a$ est nul, telle que 
$\abs {ab}=\abs a\cdot \abs b$ pour tout $(a,b)\in F^2$, et telle que
$\abs{a+b}\leq \max(\abs  a, \abs b)$ pour tout cpoupel $(a,b)$ d'éléments
de même degré de $F$. On définit également dans ce
cadre de manière naturelle les relations d'équivalence et de raffinement entre valuations, et les pendants colorés de toutes ces
notions. 

Fixons un groupe abélien $D$, un corpoïde $D$-gradué $F$ et une valuation $\abs \cdot$ sur $F$
à valeurs dans $G\cup\{0\}$ pour un certain groupe abélien ordonné $G$. 

\subsubsection{}\label{corpoide-residuel-general}
Pour tout
couple
$(d,g)\in D\times G$
nous poserons
\[\widetilde F^{(d,g)}=\{z\in F^d, \abs z\leq g\}/\{z\in F^d, \abs z<g\}\,;\]
si $x$ est un élément de $F^\times$ on notera $\widetilde x$ son image dans
$\widetilde F^{\deg x, \abs x}$. 
Si $H$ est un sous-groupe de $D\times G$ nous poserons
\[\widetilde F^H=\coprod_{(d,g)\in H}\widetilde F^{(d,g)}.\]
S'il n'y a pas d'ambiguïté, nous écrirons simplement
$\widetilde F$ au lieu de
$\widetilde F^{D\times G}$. 
L'addition et la multiplication de
$F$ induisent une structure de
$(D\times G)$-corpoïde
sur $\widetilde F$ ; nous l'appellerons
le \textit{corpoïde résiduel} de $(F,\abs\cdot)$ 
(ou simplement de $F$ s'il n'y a pas d'ambiguïté) ; remarquons que $\widetilde F$
n'a de termes non nuls qu'en degrés appartenant à $\widetilde F^{\deg(F^\times)\times \abs{F^\times}}$. 

\begin{rema}
La construction ci-dessus sera le plus souvent utilisée lorsque $D=\{1\}$, c'est-à-dire lorsque $F$ est un corps valué. 
Sous cette  hypothèse $\widetilde F$ est simplement $\abs{F^\times}$-gradué, et $\widetilde F^1$ est le corps
résiduel ordinaire de $F$. Par exemple supposons que $F$ est un corps
muni d'une valuation discrète dont on note 
$R$ l'anneau, $\pi$
une uniformisante et $k$ le corps résiduel
$R/\pi R$.
Le $\abs \pi^\Z$-annéloïde 
$\widetilde F$ est alors $\coprod_{n\in \Z} \pi^n R/\pi^{n+1}R$, le sommande indexé
par $n$ étant de degré $\abs \pi^n$. 
En tant qu'algèbre annéloïde, c'est l'algèbre $k[\widetilde \pi,\widetilde \pi^{-1}]$ des polynômes de Laurent en $\widetilde \pi$ (qui est 
de degré $\abs \pi$, les éléments de $k$ étant en degré $1$). 

\end{rema}

\subsubsection{}\label{abhyankar}
Revenons aux hypothèses en notations de \ref{hypotheses-generales}
et \ref{corpoide-residuel-general}. 
Soit $K$ un sous-corpoïde de $F$
et soit $(\alpha_i)_i$ une famille d'éléments non nuls de $\widetilde F$, 
de degrés respectifs $(d_i,g_i)_i$. Pour tout $i$, choisissons
un élément $a_i\in F^{d_i}$ tel que $\abs{a_i}=g_i$ et $\widetilde {a_i}=\alpha_i$. 
Il résulte immédiatement des définitions que les $\alpha_i$ sont algébriquement indépendants
sur
le corpoïde $\widetilde K$ si et seulement si pour tout polynôme
$P=\sum a_I T^I$ appartenant à $K[d_i\backslash T_i]$ on a 
$\abs{P(a_i)_i)}=\max \abs {a_I}g^I$ ; si c'est le cas on voit en particulier que les
$a_i$ sont algébriquement indépendants.

Supposons que $F$ soit de degré de transcendance fini $n$ sur $K$. Par ce qui précède
on voit que le degré de transcendance de $\widetilde F$ sur
le corpoide $\widetilde K$ est
majoré par $n$ : c'est l'\emph{inégalité d'Abhyankar}. On dit que l'extension de corpoïdes valués $K\hookrightarrow F$
est d'Abhyankar si  le degré de transcendance de $\widetilde F$ sur $\widetilde K$
est égal à $n$ ; cela revient à demander qu'il existe des éléments $a_1,\ldots, a_n$ non nuls de $F$,
dont on note $d_i$ les degrés respectifs, et des éléments $g_1,\ldots, g_n$ de $G$ tels que 
pour tout polynôme
$P=\sum a_I T^I$ appartenant à $K[d_i\backslash T_i]$ on ait 
$\abs{P(a_i)_i)}=\max \abs {a_I}g^I$ ; nous dirons qu'une telle famille $(a_i)$ est une
\textit{base d'Abhyankar} de $F$ sur $K$. 

Supposons que $F$ soit une extension d'Abhyankar de $K$, et soit 
$(a_1,\ldots, a_n)$
une base d'Abhyankar de $F$ sur $K$ ; soit $F_0$ le sous-corpoïde $K(a_1,\ldots, a_n)$ de $F$. 
La famille $(a_i)$ est une base de transcendance de $F$ sur $K$, et $F$ est donc
une extension algébrique
de $F_0$, finie si $F$ est de type fini sur $K$. Les $\widetilde a_i$ forment une base de transcendance de 
$\widetilde F$ sur $\widetilde K$, et $\widetilde {F_0}
$ est simplement égal à 
$\widetilde K(\widetilde{a_1},\ldots, \widetilde{a_n})$ par
un calcul immédiat fondé sur la description 
explicite de la restriction de $\abs \cdot$ à $F_0$. 
Si $F$ est de type fini sur $K$, la finitude de $F$ sur $K_0$ entraîne la finitude de $\widetilde
F$ sur $\widetilde F_0$, si bien que $\widetilde F$ est de type fini sur $\widetilde K$. 

\subsection{Raffinements de valuations colorées}\label{raff-val-cor}
Soit $\Delta$ un groupe abélien ordonné, soit $k$ un corps
muni d'une $\Delta$-valuation, et soit $F$ une extension de $k$. 
Soit $\abs \cdot$ une $k$-valuation $\Delta$-colorée sur $F$. 
Le but de ce qui suit est d'étudier les raffinements de $\abs \cdot$ en tant que
$k$-valuation $\Delta$-colorée au moyen du corpoïde résiduel $\widetilde F$, qui est
$\abs{F^\times}$-gradué. 

\subsubsection{}
L'application de $\widetilde F^\times$ vers $\abs F\cdot \Delta$
qui envoie $x$ sur $\deg(x)$ si $x\neq 0$ et sur $0$ sinon est une valuation sur $\widetilde F$,
qu'on se permettra de noter encore $\abs \cdot$, puisqu'elle satisfait par construction 
la formule $\abs {\widetilde a}=\abs a$ pour tout $a\in F^\times$. 
Remarquons que $\abs \alpha\in \Delta$ pour tout $\alpha\in \widetilde k$ : ainsi, 
la restriction de $\abs \cdot $ à $\widetilde k$ apparaît comme une 
$\Delta$-valuation, ce qui fait de $\abs \cdot$ une $\widetilde k$-valuation 
$\Delta$-colorée sur $\widetilde F$. 

\subsubsection{}
Supposons donnée une $k$-valuation $\Delta$-colorée $\abs \cdot'$
sur $F$ qui raffine $\abs \cdot$. 
Soit $\alpha$ un élément non nul de $\widetilde F$ et soit $a$ un élément non nul
de $F$ tel que $\widetilde a=\alpha$. Il est immédiat que $\abs a'$ ne dépend que de $\alpha$
et pas de $a$. Ainsi, $\abs \cdot '$ induit une valuation 
sur 
$\widetilde F$, encore notée $\abs \cdot '$,
à valeurs dans le groupe abélien 
$\Delta$-coloré $\abs{F^\times}'\cdot \Delta$, 
caractérisée par le fait
que $\abs a'=\abs{\widetilde a}'$ pour tout $a\in F^\times$.
Par construction, $\abs \cdot'$ est une $\widetilde k$-valuation
$\Delta$-colorée sur $\widetilde F$ qui raffine $\abs \cdot$.

Réciproquement, supposons donnée une
$\widetilde k$-valuation $\Delta$-colorée
$\abs \cdot '$ sur $\widetilde F$
qui raffine $\abs \cdot$. 
La valuation $\abs \cdot '$ induit alors une valuation
$\Delta$-colorée sur $F$ à valeurs dans $\abs{\widetilde F^\times}'\cdot \Delta$, 
qui raffine $\abs \cdot$ 
et
est définie par la formule
$\abs a'=\abs{\widetilde a}'$ pour tout $a\in F^\times$. 

\subsubsection{}\label{raffinement-corpores}
Les constructions précédentes mettent en bijection l'ensemble des raffinements
(à équivalence près) de la $k$-valuation $\Delta$-colorée $\abs\cdot $
sur $F$ et l'ensemble des raffinements
(à équivalence près) de la $\widetilde k$-valuation $\Delta$-colorée $\abs\cdot $
sur $\widetilde F$. 

Cette bijection préserve les corpoïdes résiduels, dans le sens suivant. 
Soit $\abs \cdot '$ un raffinement de la $k$-valuation $\Delta$-colorée $\abs \cdot$ de $F$ ; notons encore $\abs \cdot'$ la
$\widetilde k$-valuation $\Delta$-colorée correspondante 
correspondante sur $\widetilde F$.  Soit $\dot  F$ le corpoïde résiduel de $(F,\abs \cdot '$) et soit
$\ddot F$ celui de  $(\widetilde F,\abs \cdot '$). Le corpoïde résiduel $\ddot F$ est \textit{a priori}
gradué par $\abs {F^\times}\times \abs{\widetilde F^\times}'$, mais
si $\pi$ désigne le morphisme canonique $\abs {F^\times}'\cdot \Delta\to
\abs{F^\times}\cdot \Delta$, 
le degré d'un élément
de $\ddot F^\times$ appartient toujours au groupe
$\{\pi(g), g\}_{g\in\abs{F^\times}'}$, qui s'identifie à $\abs{F^\times}'$ \textit{via}
la seconde projection. 
On peut donc voir $\ddot F$ comme un corpoïde
$\abs{F^\times}'$-gradué, et l'on dispose alors d'un isomorphisme
de corpoïdes $\abs{F^\times}'$-gradués
entre $\dot F$ et $\ddot F$
tel que le diagramme 
\[\begin{tikzcd}
F^\times\ar[r]\ar[d]&\widetilde F^\times\ar[d]\\
\dot F^\times\ar[r,"\sim"]&\ddot F^\times
\end{tikzcd}
\]
commute. 

\subsubsection{}\label{bijection-raffcor-valtoutcourt}
Soit
$\Gamma$
un sous-groupe de $\Delta$ l'engendrant modulo $\abs{k^\times}$. 

Soit $\abs \cdot '$ une $\widetilde k$-valuation $\Delta$-colorée sur $\widetilde F$
qui raffine $\abs \cdot$. 
On lui associe une valuation (tout court, c'est-à-dire non colorée)
$\langle \cdot \rangle$ sur le corpoïde
$\widetilde F^\Gamma$, triviale sur $\widetilde k^\Gamma$,
qu'on définit par la formule 
$\langle \alpha \rangle=\abs \alpha '/\abs \alpha$ pour tout $\alpha$ non nul
de $\widetilde F^\Gamma$ (rappelons que pour un tel $\alpha$ on a $\abs \alpha=\deg(\alpha)$, si
bien que $\abs \alpha\in \Gamma\subset \Delta$ et que diviser par $\abs \alpha$ dans $\abs{\widetilde F}'\cdot \Delta$
a donc un sens).

Réciproquement, soit $\langle \cdot \rangle$ une valuation sur $\widetilde F^\Gamma$, triviale
sur $\widetilde k^\Gamma$.  On vérifie, en utilisant
le fait que $\Delta=\abs{\widetilde k^\times} \cdot \Gamma$,  qu'il existe 
une unique
$\widetilde k$-valuation $\Delta$-colorée
$\abs \cdot '$ sur $\widetilde F$ tel que pour tout $\gamma \in \Gamma$
et tout $\alpha \in \widetilde F^\times$  on ait 
\[\abs \alpha'\leq \gamma\iff (\abs\alpha <\gamma)
\;\text{ou}\;(\abs\alpha =\gamma\;\text{et}\;\langle \alpha \rangle
\leq 1).\]

Il résulte de la construction que pour tout $\delta\in \Delta$ et tout $\alpha \in\widetilde F$ on a 
l'implication $\abs \alpha '\leq \delta \Rightarrow \abs \alpha \leq \delta$. Par conséquent, 
$\abs \cdot '$ raffine $\abs \cdot$. 

\subsubsection{}\label{bij-raffinement-recap}
On a ainsi mis en bijection l'ensemble des $\widetilde k$-valuations $\Delta$-colorées sur $\widetilde
F$ (à équivalence près) qui raffinent $\abs \cdot$ et l'ensemble des valuations sur 
$\widetilde F^\Gamma$ (à équivalence près) qui sont triviales sur $\widetilde k^\Gamma$. 

Compte-tenu de
\ref{raffinement-corpores}
on obtient une bijection entre l'ensemble des $k$-valuations $\Delta$-colorées
sur $F$ (à équivalence près) qui raffinent $\abs \cdot$ et l'ensemble des valuations sur 
$\widetilde F^\Gamma$ (à équivalence près) qui sont triviales sur $\widetilde k^\Gamma$. 
Si $\abs \cdot '$ est une $k$-valuation $\Delta$-colorée
sur $F$ et si $\langle \cdot \rangle$ désigne la valuation associée sur $\widetilde F^\Gamma$,
le lien entre les deux  est entièrement
décrit par 
l'équivalence \[\abs a'\leq \gamma\iff(\abs a<\gamma)\;\text{ou}\;(\abs a=\gamma\;
\text{et}\;\langle \widetilde a\rangle \leq 1)\]
valable pour tout $a\in F^\times$.

\subsection{Compléments
sur les valuations bridées}\label{subsection-valuations-bridees}
On fixe un corps ultramétrique complet $k$, ce qui veut dire que $k$ est muni d'une valuation
\textit{réelle} $\abs \cdot \colon k\to \R_{\geq 0}$ pour laquelle il est complet. On fixe également un sous-groupe
$\Gamma$ de $\R_{>0}$ tel que $\abs{k^\times}\cdot \Gamma\neq \{1\}$ : autrement dit, $\Gamma$ est non trivial
si $k$ est trivialement valué. Dans ce qui suit, nous considèrerons $\abs \cdot$ comme une $\abs k^\times\cdot \Gamma$-valuation.

\subsubsection{}
Soit $F$ une extension de $k$. 
Soit $\abs \cdot$
une
$k$-valuation $\abs{k^\times}\cdot \Gamma$-colorée
et bridée sur $F$ (comme $F$ est un corps, cette
dernière condition revient à demander 
que pour tout 
$\lambda\in F^\times$
il existe $\gamma
\in \Gamma$ tel que $\abs \lambda
\leq \gamma$).
Soit $\lambda\in \abs {F^\times}$. L'ensemble 
\[\{\alpha\in (\abs{k^\times}\cdot \Gamma)^\Q,\alpha \leq \abs\lambda\}
\subset \rpos\] est non vide et majoré, 
et il admet donc une borne supérieure $\abs \lambda^\#$
dans $\rpos$. Il est immédiat que $\abs\cdot^\#$ (étendue à $F$ en posant
$\abs 0^\#=0$) est une $\rpos$-valuation sur $F$ prolongeant la valuation structurale
de $k$ ; c'est donc une $k$-valuation $\abs{k^\times}\cdot \Gamma$-colorée et bridée sur $F$, 
dont $\abs \cdot$ est un raffinement. 
Notons $\widetilde F$ le
corpoïde résiduel de $(F,\abs \cdot^\#)$. La valuation $\abs \cdot$ induit
d'après
\ref{bij-raffinement-recap}
une valuation $\abs \cdot ^\flat$ sur $\widetilde F^\Gamma$, triviale sur $\widetilde k^\Gamma$.

Réciproquement, \ref{bij-raffinement-recap}
permet d'associer à toute valuation réelle $\abs \cdot^\#$ sur $F$ prolongeant la valuation de $k$
et toute valuation $\abs \cdot^\flat$ sur $\widetilde F^\Gamma$, triviale sur 
$\widetilde F^\Gamma$
une $k$-valuation $\abs{k^\times}\cdot \Gamma$-colorée $\abs \cdot$ sur $F$
raffinant $\abs \cdot^\#$, qui sera bridée par
construction. 

\subsubsection{}\label{raffinement-gamma}
Il découle de ce qui précède qu'il existe une bijection entre l'ensemble des 
$k$-valuations $\abs{k^\times}\cdot \Gamma$-colorées et bridées sur $F$ et celui des couples
$(\abs\cdot^\#, \abs\cdot ^\flat)$ où $\abs \cdot^\#$ est une valuation réelle
sur $F$ prolongeant la valuation de $k$ et où $\abs \cdot^\flat$ est une valuation sur
la partie $\Gamma$-graduée du corpoïde résiduel 
de
$(F,\abs \cdot^\#)$, triviale sur $\widetilde k^\Gamma$. 

Le lien entre $(\abs \cdot^\#, \abs \cdot^\flat)$ et $\abs \cdot$ est entièrement décrit par l'équivalence
\[\abs{\lambda}\leq \gamma\iff (\abs{\lambda}^\#<\gamma)\;\text{ou}\;(\abs{\lambda}^\#=\gamma\;\text{et}\;
\abs{\widetilde \lambda}^\flat\leq 1),\]
valable pour tout $\lambda \in F^\times$. 

\subsubsection{}\label{raffinement-completion}
Soit  $\abs \cdot$
une
$k$-valuation $\abs{k^\times}\cdot \Gamma$-colorée
et bridée sur $F$ et soit $(\abs \cdot^\#, \abs \cdot ^\flat)$ le couple qui
lui correspond \textit{via}
la bijection du \ref{raffinement-gamma}. Soit
$\widetilde F$ le corpoïde résiduel 
de $(F,\abs \cdot ^\#)$. 
Soit $\widehat F$ le complété de $K$ pour la valuation réelle $\abs \cdot^\#$ ; celle-ci se prolonge
canoniquement en une valuation réelle sur $\widehat F$ que l'on note encore $\abs \cdot^\#$, et dont
le corpoïde résiduel est encore $\widetilde F$. Au couple $(\abs \cdot ^\#, \abs \cdot^\flat)$ relatif 
à l'extension $\widehat F$ de $k$ correspond alors en vertu de \ref{raffinement-gamma} une 
$k$-valuation $\abs{k^\times}\cdot\Gamma$-colorée et bridée $\abs \cdot$ sur $\widehat F$, qui prolonge
la valuation $\abs \cdot $ de $F$ -- nous dirons que c'est son prolongement canonique à $\widehat F$. 
Celui-ci se décrit aisément. En effet, 
soit $\lambda$ un élément non nul de $\widehat F$, et soit $\lambda_0$ un élément de $F$ tel que
$\abs{\lambda-\lambda_0}^\#<\abs \lambda^\#$ ; il résulte alors
immédiatement des définitions que 
$\abs \lambda=\abs{\lambda_0}$. Notons qu'en conséquence $\abs{\widehat F^\times}=\abs{F^\times}$.

\subsubsection{Le cas d'une $k$-algèbre}
\label{valbrid-casgen}
Soit $A$ une $k$-algèbre. Une  $k$-valuation $(\abs{k^\times}
\cdot \Gamma)$-colorée $x$ sur $A$ est bridée si
et seulement si la valuation $(\abs{k^\times}
\cdot \Gamma)$-colorée $\abs\cdot $ 
de $\kappa(x)$ est bridée. 
Si c'est le cas, il lui correspond un couple $(\abs \cdot^\#, \abs \cdot^\flat)$ par la
bijection de \ref{raffinement-gamma}. Le complété de $\kappa(x)$ pour la valuation réelle
$\abs \cdot^\#$ sera noté $\hr x$, et le prolongement canonique
de $\abs \cdot$ à $\hr x$ sera encore noté $\abs \cdot$
(\ref{raffinement-completion}). On désignera par $\hrt x$ le 
corpoïde résiduel de $\hrt x$ relatif
à sa valuation réelle $\abs \cdot^\sharp$. 

L'application $a\mapsto \abs{a(x)}^\#$ est une $k$-valuation $(\abs{k^\times}\cdot \Gamma)$-colorée
et bridée sur $A$, que nous noterons $x^\#$. Avec nos conventions on a donc $\abs{a(x)}^\#
=\abs{a(x^\#)}$. Quant à la valuation résiduelle $\abs \cdot ^\flat$, nous la verrons comme un point
de l'espace de Zariski-Riemann $\P_{\hrt {x^\#}^\Gamma/\widetilde k^\Gamma}$ des valuations sur le corpoïde 
$\hrt {x^\#}^\Gamma$ qui sont triviales sur $\widetilde k^\Gamma$, et nous la noterons $x^\flat$. Si $a$ est un élément
de $A$ tel que $\abs{a(x^\#)}\in \Gamma$, la notation $\abs{\widetilde{a(x^\#)}(x^\flat)}$ a donc un sens. 
Comme elle est passablement indigeste, nous l'abrègerons en $\abs{a(x^\flat)}$.

\section{Espaces domaniaux et géométrie linéaire par morceaux}
Le but de cette section est de décrire
la théorie des espaces
linéaires par morceaux abstraits dont nous nous servirons ici.
Nous avons adopté (suivant en cela Berkovich)
le point de vue \textit{multiplicatif}, consistant à voir $\rpos$ comme un espace vectoriel réel, 
sa loi de groupe étant la multiplication, et sa multiplication externe l'exponentiation ; cela
nous permettra d'éviter une profusion de symboles $\exp$ et $\log$ qui alourdiraient les notations
sans aucune utilité mathématique. 

Nous avons par ailleurs choisi de faire une présentation faisceautique de la théorie, 
qui nous 
a
\textit{in fine}
a semblé plus adaptée à la définition et l'étude des squelettes
que la présentation plus traditionnelle par cartes
et atlas, développée par Berkovich dans \cite{berkovich2004} et reprise dans
\cite{ducros2012b} ; cette présentation est en partie inspirée par le formalisme
des «espaces tropicaux» développé par Antoine Chambert-Loir et le premier auteur
dans leur travail en cours \cite{chambertloir-d-2012}. 

\subsection{Les espaces domaniaux}
Avant d'en venir aux espaces linéaires par morceaux proprement dite nous allons 
introduire une notion plus générale, celle d'\textit{espace domanial enrichi}, qui s'avèrera 
très commode pour décrire les relations entre la géométrie de Berkovich et la géométrie
linéaire par morceaux, qui sont au cœur de la
théorie des squelettes.

\subsubsection{Brefs rappels de topologie générale}
Rappelons qu'une application continue $f\colon Y\to X$ est dit \textit{propre}
si et seulement si elle est universellement fermée, c'est-à-dire si et seulement si pour tout espace topologique
$Z$ et toute application continue $Z\to Y$, l'application continue induite $Y\times_X Z\to Z$ est fermée. 
On démontre que $f$ est propre si et seulement si $f$
est fermée à fibres \textit{quasi-compactes} (nous ne faisons aucune hypothèse de séparation \textit{a priori}). 

Si $f$ est propre $f^{-1}(K)$ est quasi-compact pour toute partie quasi-compacte $K$ de $X$. La réciproque est vraie lorsque $X$
est localement compact (rappelons qu'un espace localement compact est \textit{par définition} séparé). 

Si $X$ possède une base de voisinages compacts (sans être nécessairement séparé), l'application $f$ est propre et séparée si et seulement si
$f^{-1}(K)$ est compacte pour toute partie compacte $K$ de $X$.

\subsubsection{}\label{def-gtopo}
Si $E$ est une partie d'un espace
topologique $X$, on dit qu'une famille $(E_i)$ de sous-ensembles de $E$ est un
\textit{G-recouvrement} de $E$ si tout point $x$ de $E$ possède un voisinage \textit{dans $E$}
de la forme $\bigcup_{i\in I} E_i$ où $I$ est un ensemble fini d'indices tels que $x\in \bigcap_{i\in I}E_i$. 
Donnons deux exemples importants de G-recouvrements : 

\begin{itemize}[label=$\diamond$]
\item si $U$ est un ouvert de $X$, tout recouvrement ouvert de $U$ en est un G-recouvrement ; 
\item si $(F_i)$ est une famille localement finie de fermés de $X$ et si l'on pose
$F=\bigcup_i F_i$ alors
$F$ est un fermé de $X$ dont les $F_i$ constituent un G-recouvrement. 
\end{itemize}

\begin{defi}\label{def-espace-domanial}
Un \textit{espace domanial} est un espace topologique $X$
dont tout point possède une base de voisinages compacts
et qui est muni d'un ensemble de parties, les \textit{domaines}
de $X$, vérifiant les propriétés suivantes : 
\begin{enumerate}[a]
\item tout ouvert de $X$ est un domaine ; 
\item l'intersection d'une famille finie de domaines de $X$ est un domaine de $X$ ; 
\item tout domaine de $X$ est G-recouvert par des domaines
compacts ; 
\item toute partie de $X$ qui est G-recouverte par des domaines de 
$X$ est un domaine de $X$.
\end{enumerate}
\end{defi}

\subsubsection{}\label{sorites-domaines}
Soit $X$ un espace domanial. 
Soit 
$V$ un domaine de $X$.
Choisissons un recouvrement $(U_i)$ de $X$
par des ouverts séparés. Chacun des $U_i\cap V$ est un domaine de $X$, donc est G-recouvert 
par des domaines compacts de $X$, qui ont fermés dans l'espace séparé $U_i$. Il s'ensuit que
chacun des $V\cap U_i$ est localement fermé dans $U_i$ puis que $V$ est localement fermé dans $X$. 

Le domaine $V$ hérite d'une structure naturelle d'espace domanial : ses domaines sont simplement les domaines
de $X$ contenus dans $V$. 

Si $x$ est un point de $X$, un \textit{voisinage domanial} de $x$ est simplement
un voisinage de $x$
qui est aussi un domaine de $X$. 
Si $U$ est un voisinage ouvert de $x$, il existe un voisinage ouvert séparé $U'$ de $x$ dans $U$, 
et un G-recouvrement de $U'$ par des domaines compacts ; en particulier, il existe une famille finie $(V_i)$ de domaines compacts
de $X$ contenus dans $U'$ et contenant $x$ dont la réunion $V$ est un voisinage
de $x$ dans $U'$. Comme $U'$ est séparé les $V_i$ sont
fermés dans $U'$ et $V$ est donc un domaine compact de $X$. 
Ainsi tout point de $X$ possède-t-il une base de voisinages domaniaux compacts. 

On fait de la catégorie des domaines de $X$ (le flèches étant les inclusion) un site
en décrétant que les familles couvrantes sont les G-recouvrements ; on l'appelle le site domanial, 
et sa topologie est appelée la topologie domaniale, ou encore la G-topologie
de $X$. 

\subsubsection{}
On fait de la classe des espaces domaniaux une catégorie, 
en définissant un morphisme d'espaces domaniaux 
de $Y$ vers $X$ 
comme
une application continue $f\colon Y\to X$ telle que $f^{-1}(V)$ soit un domaine de $Y$ pour tout domaine $V$ de
$X$. 
Un tel morphisme induit un morphismes entre les sites domaniaux de $Y$ et $X$. 

\subsubsection{}\label{locferm-domanial}
Soit $X$ un espace domanial et soit $Z$
une partie localement fermée de $X$. 
L'ensemble $Z$ hérite d'une structure  domaniale naturelle, dite induite, pour laquelle
une partie $T$ de $Z$ est un domaine si et seulement si $T$ est G-recouverte par des parties de la forme $V\cap Z$
où $V$ est un domaine de $X$. Lorsque $Z$ est un domaine de $X$, cette structure domaniale coïncide avec celle 
de \ref{sorites-domaines}. 
L'inclusion de $Z$ dans $X$ est un morphisme d'espaces domaniaux
et $Z\hookrightarrow X$ est l'objet final de la catégorie
des espaces domaniaux munis d'un morphisme vers $X$
dont l'image ensembliste est contenue dans Z. 

On déduit de cette 
caractérisation fonctorielle que si $T$ est une partie 
localement fermée de $Z$, les structures domaniales
sur $T$ héritées de $X$ et de $Z$ coïncident. 

\begin{defi}
Un espace domanial \textit{enrichi} est un 
espace domanial $X$ muni
d'un faisceau de fonctions continues
à valeurs dans $\rpos$ sur son site domanial. 

Un morphisme d'espaces domaniaux enrichis $f\colon (Y,\Lambda_Y)\to (X,\Lambda_X)$ est 
un morphisme d'espaces domaniaux $f\colon Y\to X$ tel que 
pour tout domaine $U$ de $X$, pour tout domaine $V$ de $f^{-1}(U)$
et toute $\phi\in \Lambda_X(U)$, la fonction $\phi\circ (f_{|V})$ appartienne à
$\Lambda_Y(V)$. 
\end{defi}

\subsubsection{Commentaires et premières propriétés}
\label{locferm-enrichi}
Si $(X,\Lambda)$ est un espace domanial
enrichi, tout
domaine $V$ de $X$ peut être vu naturellement comme
un espace domanial
enrichi, \textit{via} $\Lambda_{|V}$. Plus
généralement, soit $Z$ une partie localement fermée
de $X$ munie de sa structure domaniale induite 
(\ref{locferm-domanial}) et soit 
$i$ l'inclusion de $Z$ dans $X$. 
Pour tout domaine $T$ de $Z$ notons $\Lambda_Z(T)$ l'ensemble
des
applications continues de $T$ vers $\rpos$ qui sont G-localement de la forme
$u_{|V\cap T}$ où $V$ est un domaine de $X$ et $u$
un élément de $\Lambda(V)$ ; il est immédiat que $\Lambda_Z$
est un faisceau sur le site domanial de $Z$, faisant 
de ce dernier un espace domanial enrichi ; nous dirons
que
la structure d'espace domanial enrichi ainsi construite sur $Z$
est héritée de celle de $X$. Par construction, 
$i$ est 
un morphisme d'espaces domaniaux enrichis 
de $(Z,\Lambda_Z)$ vers $(X,\Lambda)$, et $i\colon
Z\hookrightarrow X$ est l'objet final de la catégorie
des espaces domaniaux enrichis
munis d'un morphisme vers $X$
dont l'image ensembliste est contenue dans $Z$. 

On déduit de cette 
caractérisation fonctorielle que si $T$ est une partie 
localement fermée de $Z$, les structures domaniales
enrichies sur $T$ héritées de $X$ et de $Z$ coïncident.

\subsection{Espaces linéaires par morceaux}
\label{pl-abstrait}
Nous allons maintenant donner un premier
example majeur d'espaces domaniaux enrichis. 
On fixe un couple $c=(K,\Delta)$ où $K$ est un
sous-corps de $\R$ et $\Delta$ un sous-$K$-espace
vectoriel de
$\rpos$. 
Ce couple représente les paramètres avec lesquels nos espaces linéaires
par morceaux seront définis : $K$ est l'ensemble
des exposants possibles de nos applications 
monomiales (c'est-à-dire affines, avec notre point de vue multiplicatif) 
et $\Delta$ celui de leurs termes
constants.  Dans les applications que nous avons en vue, $K$ sera toujours égal
à $\Q$.

\subsubsection{Les $c$-polytopes}\label{pl-plonge}
Une application $\phi$
de $\rposn n$
vers $\rposn m$
est dite $c$-affine si chacune
de ses composantes
est de la forme 
\[(x_1,\ldots, x_n)\mapsto a x_1^{e_1}\ldots x_n^{e_n}\]
où les $e_i$ appartiennent à $K$ et $a$ à $\Delta$. 
Un \textit{$c$-polytope} de $\rposn n$ est une partie 
de $\rpos^n$ 
qui est définie par une condition de la forme
\[\bigvee_i \bigwedge_j
\phi_{ij}\leq 1\] où les $\phi_{ij}$ sont des
applications $c$-affines
de $\rposn n$ vers $\rpos$.

\subsubsection{}Une
\textit{$c$-cellule}
de $\rposn n$ est un $c$-polytope de ce dernier
qui est non vide et peut être défini par une conjonction 
finie 
d'inégalités $\bigwedge_i \phi_i\leq 1$, où les $\phi_i$
sont $c$-affines. (On peut démontrer que cela revient
à demander qu'il soit convexe et non vide). 

Le \textit{bord} $\partial C$ d'une $c$-cellule $C$ 
est son bord topologique dans le sous-espace
affine $\langle C\rangle$ ; 
son \textit{intérieur} $\mathring C$ 
est son intérieur topologique
dans $\langle
C\rangle$. 
Si $C$ est décrite 
par une conjonction $\bigwedge_i \phi_i\leq 1$
comme ci-dessus, $\partial C$ est le $c$-polytope
décrit par la condition 
\[\bigvee_{j\in J}\left((\phi_{j}=1)
\wedge  \bigwedge_{i\neq j}\phi_i\leq 1\right),\] où $J$
est l'ensemble des indices $j$ tels que $\phi_j$
ne soit pas identiquement égale à $1$ sur $P$. 

Une \textit{$c$-cellule ouverte} de $\rposn n$ est une
partie de la forme $\mathring C$ où $C$ est
une $c$-cellule, qui est uniquement déterminée
(c'est l'adhérence de $\mathring C$). 

\begin{rema}\label{rem-im-cellule}
Soient $n$ et $m$ deux entiers, soit $\phi$ une application $c$-affine de $\rposn n$ vers $\rposn m$
et soit $C$ une $c$-cellule de $\rposn n$. L'image $\phi(C)$ est une $c$-cellule de $\rposn m$ (par des changements
convenables de coordonnées on se ramène au cas d'une projection, et l'on procède alors par un calcul explicite) ; 
si $\phi$ est injective, elle induit un
isomorphisme affine de $\langle C\rangle $ sur $\langle \phi(C)\rangle$, et en particulier
un homéomorphisme de $C$ sur $\phi(C)$. 
\end{rema}

\subsubsection{Décompositions
$c$-cellulaires des polytopes.}\label{decomp-cell}
Si $P$ est un $c$-polytope quelconque de $\rposn n$, il possède
une \textit{décomposition $c$-cellulaire finie}
$\mathscr C$, c'est-à-dire un ensemble fini de $c$-cellules
contenues dans $P$, qui recouvrent $P$, et telles que : 
\begin{itemize}[label=$\diamond$]
\item $\partial C$ est pour toute $C\in \mathscr C$
une union finie d'éléments
de $\mathscr C$ ; 
\item si $C$ et $D$ appartiennent à $\mathscr C$ alors
$C\cap D$ est une union finie d'éléments de $\mathscr C$ ; 
\item si $C$ et $D$ appartiennent à $\mathscr C$ et si 
$C\subset D$ alors $C=D$ ou $C\subset \partial D$. 
\end{itemize}

Plus précisément, donnons-nous un ensemble fini
$\mathscr P$ de $c$-polytopes de $\rposn n$. 
Choisissons une famille finie $(\psi_i)$ de fonctions $c$-affines sur
$\rposn n$ telle que chacun des $c$-polytopes appartenant
à $\mathscr P$ soit décrit par une conjonction et disjonction d'inégalités
de la forme $\psi_i\leq 1$ pour un certain $i$ et soit $\mathscr C$ la décomposition 
cellulaire de $\rposn n$ constituée  des $c$-cellules
décrites par une conjonction de la forme
\[\bigwedge_i\psi_i\Join_i1\] où
$\Join_i\in\{\leq, \geq, =\}$
(nous dirons que c'est la décomposition cellulaire  induite par  la famille
$(\psi_i)$). 
Chaque polytope de $\mathscr P$
est alors réunion de cellules appartenant à $\mathscr C$.

Si $Q$ est une partie de $\rposn n$ qui s'écrit comme combinaison booléenne
de $c$-polytopes appartenant à $\mathscr P$ il est alors immédiat que l'on peut
écrire  $Q=\coprod_{C\in \mathscr D}\mathring C$ pour un certain sous-ensemble
$\mathscr D$ de $\mathscr C$.

\subsubsection{Espaces $c$-linéaires par morceaux : le cas plongé}
Une partie de $\rposn n$ est dite $c$-linéaire par morceaux
si elle s'écrit comme une réunion localement finie de $c$-polytopes. 
Les parties $c$-linéaires par morceaux de $\rposn n$ satisfont les axiomes 
(a), (b), (c) et (d) de la définition \ref{def-espace-domanial}, ce qui veut dire qu'elles
définissent une structure d'espace domanial sur $\rposn n$, dite $c$-linéaire
par morceaux, dont elles sont 
précisément les domaines. 

Si $P$ est une partie $c$-linéaire par morceaux de $\rposn n$,  nous dirons qu'une application 
$\phi \colon P\to \rpos$ est $c$-linéaire par morceaux s'il existe un G-recouvrement $(P_i)$ de $P$
par des $c$-polytopes tel que $\phi_{|P_i}$ soit $c$-affine pour tout $i$ ; notons que si c'est le cas, $\phi$ 
est continue puisque cette propriété se teste G-localement.

Nous noterons $\Lambda_c(P)$ l'ensemble des applications
$c$-linéaires par morceaux de $P$ vers $\rpos$. 
Si $Q$ est une partie $c$-linéaire par morceaux de $\rposn n$ contenue dans $P$, 
alors $\phi|_{|Q}\in \Lambda_c(Q)$ pour toute $\phi\in \Lambda_c(P)$, et l'appartenance
à $\Lambda_c(P)$ d'une application de $P$ vers $\rpos$ est par définition une propriété G-locale. 
Il s'ensuit que $P\mapsto \Lambda_c(P)$ est un faisceau d'applications continues à valeurs dans $\rpos$ 
pour la G-topologie de $\rposn n$, qui fait dès lors de ce dernier un espace domanial enrichi.

La dimension d'une partie $c$-linéaire par morceaux $P$ de
de 
$\rposn n$ est $\sup \dim \langle C\rangle$, où
$C$ parcourt l'ensemble des $c$-cellules
contenues dans $P$ (pour ces questions
de dimension, on optera pour la convention
$\sup\emptyset=-\infty$). On peut se limiter à calculer le supremum
sur un G-recouvrement donné de $P$ par des $c$-cellules
(par exemple sur une décomposition $c$-cellulaire donnée si $P$ est un polytope).

\begin{defi}[Espaces $c$-linéaires
par morceaux abstraits]
Un \textit{espace $c$-linéaire par morceaux}
est un espace domanial enrichi $X$ possédant un G-recouvrement $(X_i)$ par des domaines tels
que pour tout $i$, il existe un entier $n_i$, 
une partie $c$-linéaire par morceaux $P_i$ de $\rposn{n_i}$
et un isomorphisme $X_i\simeq P_i$ d'espaces $c$-linéaires domaniaux enrichis. 

Un morphisme d'espaces $c$-linéaires par morceaux est un morphisme d'espaces domaniaux
enrichis. 

Si $X$ est un espace $c$-linéaire par morceaux, les
domaines de $X$ seront appelés \textit{sous-espaces $c$-linéaires par morceaux} de 
$X$, son faisceau structural sera noté $\Lambda_c$, et ses sections seront appelées
applications $c$-linéaires par morceau. Cette terminologie et cette notation 
sont compatibles avec celles déjà introduites au \ref{pl-plonge}
dans le cas plongé. 
\end{defi}

\begin{exem}
Soit $n$ un entier et soit $P$ une partie linéaire par morceaux de $\rposn n$. Muni de sa structure
d'espace domanial enrichi définie en \ref{pl-plonge}, $P$ est par construction un espace $c$-linéaire
par morceaux. 

Soit $m$ un entier et soit $Q$ une partie $c$-linéaire par morceaux de $\rposn m$. Un morphisme
d'espaces $c$-linéaires par morceaux de $P$ vers $Q$ est
une application de $P$ vers $Q\subset \rposn m$ dont toutes
les composantes sont $c$-linéaires par morceaux. 
\end{exem}

\subsubsection{Dimension d'un espace
$c$-linéaire par morceaux}

\label{plongeable}
Soit $X$ un espace $c$-linéaire par morceaux. 
Convenons de dire que  $X$ est \textit{plongeable}
s'il est isomorphe à un sous-espace $c$-linéaire 
par morceaux $P$ de $\rposn n$  pour un certain $n$. 
Si c'est le cas $\dim P$ ne dépend
alors
que de $X$ et sera appelé la dimension de $X$. 

En général, on définit la dimension de
$X$ comme
le supremum des dimensions de ses sous-espaces
linéaires par morceaux plongeables ;
ce supremum vaut $-\infty$ si $X$ est vide, 
appartient à $\Z_{\geq 0}\cup\{+\infty\}$
sinon, et il est fini dès
que $X$ est compact et non vide; 
il peut par ailleurs être calculé
par rapport
à une famille donnée de sous-espaces $c$-linéaires
par morceaux plongeables
qui G-recouvre $X$. Cette définition est 
compatible avec la précédente si $X$ est plongeable. 

Si $x$ est un point de $X$ on définit la dimension de $X$
en $x$,  notée $\dim_x X$, comme le minimum des $\dim V$ où 
$V$ est un voisinage $c$-linéaire par morceaux compact de $x$. 
La dimension de $X$ est égale au supremum des $\dim_x X$ 
pour $x$ parcourant $X$. On dira que
$X$ est \textit{purement de dimension $n$}
(pour un certain entier $n$) si $\dim_xX=n$ pour
tout $x\in X$.

\subsubsection{Extension des paramètres de défintion}\label{pl-extension-parametres}
Soit $X$ un espace $c$-linéaire par morceaux.
Dans la notation du faisceau structural $\Lambda_c$ de $X$, le 
$c$ en indice rappelle
quels sont les paramètres avec lesquels sont définis les polytopes (multiplicatifs) sur lesquels 
tout est modelé. Mais il est possible d'élargir l'ensemble des paramètres autorisés. 

Plus précisément, donnons-nous un couple $d=(L,\Upsilon)$ où $L$ est un sous-corps de $\R$
contenant $K$ et $\Upsilon$ un sous-$L$-espace vectoriel de $\rpos$ contenant $\Delta$. 
Pour tout $n$, toute partie $c$-linéaire par morceaux $P$ de $\rposn n$ est \textit{a fortiori} $d$-linéaire par morceaux, 
et $\Lambda_c(P)\subset \Lambda_d(P)$. 

Ces considérations plongées admettent la déclinaison abstraite que voici. Soit $\mathscr E$ l'ensemble des parties $V$ de $X$ possédant la propriété suivante : il existe une partie
$c$-linéaire par morceaux $U$ de $X$ contenant $V$, un entier $n$ et un isomorphisme d'espaces $c$-linéaires par morceaux 
$\iota \colon U\simeq P$ où $P$ est une partie $c$-linéaire par morceaux de $\rposn n$, tels que
$\iota(V)$ soit une partie $d$-linéaire par morceaux
de $\rposn n$. Si $V\in \mathscr E$ et si $\iota$ est comme ci-dessus, 
l'ensemble des applications de $V$ dans $\rpos$ de la forme $f\circ \iota$ ou $f\in \Lambda_d(\iota(V))$ ne dépend que de $V$, 
et sera noté $\Lambda_d(V)$. 
On définit alors comme suit une nouvelle structure domaniale enrichie sur $X$, dont le faisceau 
structural est noté $\Lambda_d$ : 
\begin{itemize}[label=$\diamond$]
\item ses domaines sont les parties de $X$ qui sont $G$-recouvertes par des éléments de $\mathscr E$ ; 
\item si $V$ est un domaine au sens ci-dessus, $\Lambda_d(V)$ est l'ensemble des applications $f$ de $V$
vers $\rpos$ telles que $f_{|W}\in \Lambda_d(W)$ pour toute partie $W$ de $V$ appartenant à $\mathscr E$. 
\end{itemize}
Cette structure domaniale enrichie fait de $X$ un espace $d$-linéaire par morceaux, que l'on dira déduit de l'espace
$c$-linéaire par morceaux initial par extension des paramètres.

\begin{lemm}
\label{c-lin-dom}
Soit $X$un espace $c$-linéaire par morceaux et soit
$Z$ une partie localement fermée de $X$, munie 
de sa structure domaniale enrichie héritée de celle de $X$
(\ref{locferm-enrichi}). Les assertions suivantes
sont équivalentes : 
\begin{enumerate}[i]
\item le sous-ensemble $Z$ 
de $X$ en est une partie $c$-linéaire par
morceaux ; 
\item l'espace domanial 
enrichi $Z$ est un espace $c$-linéaire par
morceaux. 
\end{enumerate}
\end{lemm}

\begin{proof}
L'implication (i)$\Rightarrow$(ii) est claire. 
Supposons que (ii) soit vérifiée. Par définition, il existe
un G-recouvrement de $Z$ de la forme $(V_i\cap Z)_i$, 
où chaque $V_i$ est une partie $c$-linéaire par morceaux de
$X$ et, pour tout $i$, un isomorphisme d'espaces
domaniaux enrichis
$\phi_i\colon (V_i\cap Z)\simeq P_i$ où $P_i$ est
une partie $c$-linéaire
par morceaux de $\rposn{n_i}$ pour un certain 
$n_i$. De plus,
quitte à raffiner le recouvrement, on peut supposer
que pour
tout indice $i$, il
existe un isomorphisme $c$-linéaire par
morceaux
$\psi_i\colon V_i \simeq Q_i$ où $Q_i$ est
une partie $c$-linéaire
par morceaux de $\rposn{m_i}$ pour un certain 
$m_i$, et que l'isomorphisme 
$\phi_i$ est induit par la restrictions 
à $Z\cap V_i$ 
de $n_i$ fonctions $c$-linéaires par morceaux sur $V_i$ ; 
en adjoignant
à ces dernières les $m_i$ fonctions qui définissent
$\psi_i$, 
on peut finalement
supposer que $\phi_i$ est induit par
la restriction de $\psi_i$ à 
$Z\cap V_i$. Dans ce cas
$P_i$ est contenu dans $Q_i$ et 
$V_i\cap Z=\psi_i^{-1}(P_i)$,
si bien que $V_i\cap Z$ est une partie
$c$-linéaire par morceaux de $V_i$, donc de $X$. 
Comme ceci vaut pour tout $i$, le sous-ensemble $Z$
de $X$ en est bien une partie $c$-linéaire par morceaux. 
\end{proof}

\begin{lemm}\label{image-pl-propre}
Soit $f\colon Y\to X$ un morphisme propre (comme application continue)
d'espaces $c$-linéaires par morceaux. 

\begin{enumerate}[1]
\item L'image $f(Y)$ est un sous-espace $c$-linéaire par morceaux (nécessairement 
fermé) de $X$. 
\item Si de plus $f$ est injective, $f$ induit un isomorphisme de $Y$ sur $f(Y)$. 
\end{enumerate}
\end{lemm}

\begin{proof}
Les deux assertions sont G-locales sur $X$, ce qui permet de le supposer compact. 
L'espace $Y$ est alors
quasi-compact. 

Montrons (1). Comme $Y$ est quasi-compact, il est réunion d'un nombre fini de domaines compacts. 
Il suffit alors de montrer que l'image de chacun de ces domaines est une partie
$c$-linéaire par morceaux de $C$, ce qui permet de se ramener au cas où $Y$ est également compact. 
En raisonnant à nouveau G-localement sur $X$ et sur $Y$, on se ramène ensuite au cas où tous
deux sont plongeables et l'on conclut à l'aide de la remarque \ref{rem-im-cellule}. 

Montrons maintenant (2). Au vu de (1) l'image $f(Y)$ est une partie $c$-linéaire par morceaux compacte
de $X$ ; quitte à remplacer $X$ par $f(Y)$ on peut supposer que $f$ est surjective. Tout
G-recouvrement fini $(Y_i)$ de $Y$ par des 
parties $c$-linéaires par morceaux compactes induit alors un recouvrement $(f(Y_i))$ de $X$,
et les $f(Y_i)$ sont $c$-linéaires par morceaux d'après (1), et compacts.
Ceci permet de raisonner G-localement sur $Y$
aussi bien que sur $X$, et partant de supposer que $X$ est un $c$-polytope compact de $\rposn m$ et $Y$
une $c$-cellule compacte de $\rposn n$ (pour $n$ et $m$ convenables) et que $f$ est $c$-affine. 
On conclut alors à l'aide une fois encore de la remarque  \ref{rem-im-cellule}.
\end{proof}

\subsection{Quotient d'un espace
$c$-linéaire par morceaux sous l'action 
d'un groupe fini}

Nous nous proposons maintenant de démontrer que le quotient d'un 
espace $c$-linéaire par morceaux (séparé) $X$ sous l'action d'un groupe fini $G$
a une structure $c$-linéaire par morceaux naturelle, qui en fait un quotient
catégorique de $X$ par $G$ (théorème \ref{quotient-fini-pl}). 
Pour ce faire nous établirons tout d'abord une variante 
nettement plus faible de ce résultat (lemme \ref{quotient-definissable}), qui dit
essentiellement que $X/G$ est définissable dans le langage des $K$-espaces
vectoriels ordonnés à paramètres dans $\Delta$.

\subsubsection{}\label{notation-dx}
Commençons par introduire quelques notations. 
Soit $n$ un entier. Nous noterons $\dom {\rposn n} $ la
sous-algèbre de Boole de $\mathscr P(\rposn n)$ 
engendrée par les $c$-polytopes ; remarquons
que les éléments compacts de $\dom{\rposn n}$
sont précisément les $c$-polytopes compacts 
de $\rposn n$ (on peut le voir
en considérant une décompositon cellulaire ouverte
de $P$, \textit{cf.} \ref{decomp-cell}). 
Si
$P\in \dom {\rposn n}$, nous désignerons
par $\dom P$ la sous-algèbre de Boole
$\mathscr P(P)\cap \dom{\rposn n}$ de
$\mathscr P(P)$. 
Lorsque $P$ est un $c$-polytope compact, 
$\dom P$ est aussi la sous-algèbre de Boole
de $\mathscr P(P)$ engendrée par les 
$c$-polytopes compacts contenus dans $P$. 

Nous emploierons également la notation 
$\dom X$ lorsque $X$ est un espace $c$-linéaire
par morceaux abstrait \textit{compact}, pour
désigner la sous-algèbre de Boole
de $\mathscr P(X)$ engendrée par les sous-espaces
$c$-linéaires par morceaux compacts de $X$.

\begin{lemm}\label{quotient-definissable}
Soit $X$ un espace 
 $c$-linéaire par morceaux compact et soit $G$ un groupe
fini agissant sur $X$ par automorphismes $c$-linéaires par morceaux. 
Il existe une partition $\mathsf P$ de $X$ en éléments de $\dom X$ possédant les propriétés suivantes  : 
\begin{enumerate}[a]
\item pour tout $V\in \mathsf P$ et tout $g\in G$ l'image $g(V)$ appartient
à $\mathsf P$ ; 
\item pour tout $V\in \mathsf P$ et tout $g\in G$ ou bien $g(v)=v$ pour tout $v\in V$ ou bien 
$V\cap g(V)=\emptyset$ ; 
\item pour tout $V\in \mathsf P$ il existe une partie
$c$-linéaire par morceaux compacte $V'$ 
de $X$ contenant $V$ et un isomorphisme
$f_V$ de $V'$ sur un $c$-polytope compact de
$\rposn {n_V}$ (pour un certain $n_V$) qui identifie
$V$ à une cellule ouverte de $\rposn {n_V}$ ; 
\item tout $V\in \mathsf P$ de dimension 
égale à $\dim X$ est un ouvert de $X$. 
\end{enumerate}

\end{lemm}

\begin{proof}
Commençons par choisir un G-recouvrement fini de $X$ par des sous-espaces $c$-linéaires 
par morceaux compacts et plongeables.
En considérant des combinaisons booléennes convenables des éléments
de ce recouvrement, on obtient une première partition $\mathsf Q$ de $X$ en éléments de $\dom X$
qui possède la propriété suivante : pour tout $V\in \mathsf Q$, il existe
une partie $c$-linéaire par morceaux compacte $V'$ de $X$
contenant $V$
et un isomorphisme $c$-linéaire par morceaux $f_V$ entre $V'$ et une partie $c$-linéaire par morceaux compacte
de $\rposn {n_V}$ pour un certain $n_V$ ; l'image $f_V(V)$ appartient alors
à $\dom{g_V(V')}$. De plus, on peut toujours supposer, quitte à agrandir les $n_V$ et à composer les $f_V$ avec des translations convenables
(par des vecteurs dont les composantes appartiennent à $\Delta$) que les $n_V$ sont tous égaux à un même entier $n$, et les $f_V(V)$ 
deux à deux disjoints. La concaténation $f:=\coprod_{V\in \mathsf Q}(f_V)_{|V}$ définit alors une bijection de $X$ sur
un certain $E\in \dom{\rposn n}$. 
De plus il résulte de la construction de $f$ que 
si $Y$ est une partie de 
$X$ alors $Y\in \dom X$ si et seulement si $f(Y)$ appartient à $\dom E$. 

L'action de $G$ sur $X$ en induit une \textit{via} $f$ sur $E$, par bijections définissables dans le langage des
$K$-espaces vectoriels ordonnés, 
à paramètres dans $\Delta$. On munit $\rposn n$ de l'ordre lexicographique. Soit $F$ le sous-ensemble de $E$ formé
des point $x$ qui sont minimaux au sein de leur orbite sous $G$ ; 
le sous-ensemble $F$ de $E$ appartient à $\dom E$. 
Si $m$ est un entier, on note
$F_m$ l'ensemble des points de $F$ dont l'orbite contient exactement $m$ éléments ; 
et
pour toute application 
ensembliste $\phi \colon G\to \{1,\ldots, m\}$, 
on désigne par $F_{m,\phi}$ le sous-ensemble de $F_m$ formé des points $x$ tels que 
$g(x)$ soit
pour tout $g$ le $\phi(g)$-ième point de l'orbite de $X$
(toujours pour l'ordre lexicographique
sur $\rposn n$). 
L'ensemble des $F_{m,\phi}$ non vides constitue une partition finie 
de $F$ en éléments de $\dom E$.
Fixons $m$ et $\phi$ tels que $F_{m,\phi}$
soit non vide. Il existe alors
une partition 
finie $\mathsf P_{m,\phi}$ de 
$F_{m,\phi}$ en $c$-cellules ouvertes telles
que pour toute $C\in \mathsf P_{m,\phi}$
les propriétés suivantes soient satisfaites : 

\begin{enumerate}[a] 
\item pour tout $g\in G$ la restriction de $g$ à $C$
est $c$-affine, et il existe $V\in \mathsf Q$ telle que
$g(C)\subset f_V(V')$ ; 
\item si $\dim C=\dim X$ alors $f^{-1}(C)$ est un ouvert
de $X$. 
\end{enumerate}
En effet, il est facile en vertu de
\ref{decomp-cell}
de construire une première partition $\mathsf R_{m,\phi}$
de $F_{m,\phi}$ 
satisfaisant
(a). 
Considérons une cellule ouverte $C$
de $\mathsf R_{m,\phi}$. 
L'image réciproque $f^{-1}(\overline C)$
est alors une partie $c$-linéaire par morceaux
compacte de $X$, et son  bord topologique
dans $X$ est de ce fait contenu
dans une partie $c$-linéaire par morceaux
de dimension $\dim X-1$
(on le vérifie $G$-localement sur $X$, 
ce qui permet de se ramener au cas où
ce dernier est un $c$-polytope compact de
$\rposn m$ pour un certain $m$, dans lequel c'est clair). Il s'ensuit
(là encore à l'aide de \ref{decomp-cell})
que si $C$
est une cellule 
de dimension $\dim X$
de $\mathsf R_{m,\phi}$ elle possède elle-même
une partition finie $\mathsf S_C$ en cellules ouvertes
telles que pour toute $D\in \mathsf S_C$ de dimension 
$\dim X$ l'image réciproque $f^{-1}(D)$
soit ouverte dans $X$. 
Il en résulte qu'on peut raffiner la partition 
$\mathsf R_{m,\phi}$ pour obtenir une partition 
$\mathsf P_{m,\phi}$ de $F_{m,\phi}$ possédant les propriétés requises.

Soit alors $\mathsf P$ l'ensemble des parties $V$ de $X$
possédant la propriété suivante : il existe $m$ et $\phi$
tels que $F_{m,\phi}$ soit non vide, il existe $A$
appartenant à $\mathsf P_{m,\phi}$ et $g$ appartenant
à $G$ tels que $f(V)=g(A)$. Par construction, $\mathsf P$
est une partition de $X$ répondant aux conditions de
l'énoncé. 
\end{proof}

\begin{theo}\label{quotient-fini-pl}
Soit $X$ un espace $c$-linéaire par morceaux
et soit $G$ un groupe
fini agissant sur $X$ par automorphismes $c$-linéaires par morceaux. On suppose que toute orbite
de $G$ sur $X$ est contenue dans un ouvert
séparé de $X$. 
Soit $p$
l'application quotient
de $X$ vers $X/G$ 
et soit $\mathscr V$ l'ensemble
des parties $c$-linéaires par morceaux compactes
de $X$ telles que $p_{|V}$ soit injective
et telles que $\bigcup_{g\in G}g(V)$ soit séparé. 

\begin{enumerate}[A]
\item Les éléments de $\mathscr V$
constituent un G-recouvrement de $X$. 
\item Soit $\mathscr D$ l'ensemble des parties 
$V$ de $X/G$ telles que $p^{-1}(V)$ soit un domaine de
$X$. Pour tout $V\in \mathscr D$, soit $\Lambda_c(V)$
l'ensemble des applications $u$ de $V$ vers $\rpos$ telles
que $u\circ p\in \Lambda_c(p^{-1}(V))$. La donnée
de $\mathscr D$ et de $V\mapsto \Lambda_c(V)$ fait du
quotient topologique 
$X/G$ un espace domanial enrichi,
qui est 
$c$-linéaire par morceaux et s'identifie \textit{via}
$p$ 
au quotient catégorique de $X$ par $G$
dans la catégorie
des espaces domaniaux enrichis.

\item La structure d'espace domanial enrichi sur 
$X/G$ définie ci-dessus est la seule
telle que $p$ induise un isomorphisme 
de $V$ sur un domaine de $X/G$ pour tout
$V\in \mathscr V$. C'est aussi la seule qui fasse de $p$
un morphisme d'espaces $c$-linéaires par morceaux. 
\end{enumerate}
\end{theo}

\begin{proof}
Nous allons commencer par
(A). Il suffit de démontrer que 
$X$ possède un G-recouvrement par des parties
$c$-linéaires par morceaux $G$-invariantes et satisfaisant
elles-mêmes (A). 
Pour tout $x\in X$
l'orbite $Gx$ est contenue dans un ouvert séparé
$U$, et même dans un ouvert séparé $G$-invariant 
(considérer $\bigcap_{g\in G}g(U)$), ce qui permet 
de se ramener au cas où $X$ est séparé. 
Tout point $x$ de $X$ possède alors
un voisinage
$c$-linéaire par morceaux compact et $G$-invariant : il
suffit en effet de considérer
$\bigcup_{g\in V}g(V)$ où
$V$ est un voisinage $c$-linéaire par
morceaux compact
arbitraire de $X$. Cela
permet de se ramener au cas où
$X$ est compact et non vide ; soit 
$d$ sa dimension. On raisonne par récurrence forte sur $d$, 
et l'on suppose donc (A)
vraie pour tous
les
espaces $c$-linéaires
par morceaux compacts de dimension $<d$. 
Choisissons une partition $\mathsf P$ de $X$
satisfaisant les conclusions du lemme
\ref{quotient-definissable}. 
Notons $\mathsf P_d$ (resp. $\mathsf P_{<d}$)
l'ensemble des éléments $V$ de
$\mathsf P$ qui sont de dimension $d$ (resp. $<d$). 

Posons $Y=\bigcup_{V\in \mathsf P_{<d}}
V$. Le complémentaire de $Y$ dans $X$ est égal à
$\bigcup_{V\in \mathsf P_d}V$, qui est ouvert
d'après nos hypothèses sur $\mathsf P$. 
Par conséquent $Y$ est fermé, et peut donc se
récrire  $Y=\bigcup_{V\in \mathsf P_{<d}}
\overline V$. Or, là encore en vertu de nos hypothèses
sur $\mathsf P$, l'adhérence $\overline V$ est pour
tout $V\in \mathsf P$ une
partie $c$-linéaire par morceaux
compacte de $X$, 
qui est isomorphe à une $c$-cellule fermée
de dimension 
$\dim V$. Il s'ensuit que $Y$ est une partie
$c$-linéaire par morceaux compacte de $X$ ;
par construction, $Y$ est de dimension
$<d$ et est stable sous $G$. D'après l'hypothèse de récurrence, $Y$
possède un recouvrement fini par des éléments
de $\mathscr V$. 
Il suffit dès lors
pour conclure de démontrer que pour tout 
$V\in \mathsf P_d$, l'adhérence $\overline V$ possède
un recouvrement fini par des éléments
de $\mathscr V$. 

Fixons $V\in \mathsf P_d$. Par choix de 
$\mathsf P$, il existe une partie $c$-linéaire
par morceaux compacte $V'$ de $X$ contenant $V$ et un 
isomorphisme $f$ de $V'$ sur un $c$-polytope
compact de $\rposn n$ (pour un certain $n$), tel que
$f(V)$ soit une $c$-cellule ouverte $C$. 
Le bord $\partial V$ de $V$ dans $X$ est alors
égal à $f^{-1}(\partial C)$ et est en conséquence
une partie $c$-linéaire par morceaux compacte de $Y$ ; 
il existe donc un recouvrement fini 
$(W_i)$ de $\partial V$ par des éléments
de $\mathscr V$. Posons $D_i=f(W_i)$ 
pour tout $i$. 

Toujours d'après les propriétés 
satisfaites par la partition
$\mathsf P$, 
si $g$ est un élément de $G$ alors ou bien $g(V)\cap
V=\emptyset$, ou bien $g(v)=v$ pour tout $v\in V$. 
Par conséquent, $p_{|V}$ est injectif. Et puisque 
$\partial V$ est contenu dans $Y$ qui est stable sous $G$, 
on ne peut avoir $p(w)=p(v)$ pour un couple
$(v,w)\in V\times \partial V$. Il suffit donc
pour conclure d'exhiber un recouvrement fini
$(V_j)$ de $V$ par des parties $c$-linéaires
par morceaux compactes telles que 
pour tout $j$ l'on ait $V_j\cap \partial V
\subset W_{i(j)}$ pour un certain entier $i(j)$. 
En utilisant $f$, on voit qu'il suffit d'exhiber
un recouvrement fini de $\overline C$ par des $c$-cellules
compactes $C_j$ telles que pour 
tout $j$ l'on ait $C_j\cap \partial C
\subset D_{i(j)}$ pour un certain entier $i(j)$. 
Pour ce faire, on peut raffiner $(D_j)$ et donc supposer que $(D_j)$ est
une décomposition $c$-cellulaire de $\partial C$, 
puis ne conserver que les $D_j$
de dimension $d-1$ (qui recouvrent 
$\partial C$). 
Pour toute paire $\{a,b\}$ 
d'indices
distincts on choisit une forme $c$-affine $\phi_{ab}$ tel que l'hyperplan d'équation $\phi=1$ sépare 
$D_a$ et $D_b$, et ne contienne aucune de
ces deux $c$-cellules. 
Soit $\mathscr C$ la décomposition $c$-cellulaire de $\rposn n$ 
définie par la famille $(\phi_{ab})_{a,b}$ et soit $\mathscr C'$ l'ensemble
des parties de la forme $P\cap C$ avec $P\in \mathscr C$. Par construction $C=\bigcup_{Q\in \mathscr C'}
Q$, et pour tout $Q\in \mathscr C'$ il existe $j$
tel que $Q\cap \partial C\subset D_j$, ce qui achève la démonstration de (A). 

Montrons maintenant (B). On a vu au début de la preuve
de (A) que tout point de $X$ possède un voisinage
ouvert $G$-invariant, et, au sein de ce voisinage, un 
second voisinage $c$-linéaire par morceaux compact 
lui aussi $G$-invariant. Il s'ensuit
aussitôt que 
$\mathscr D$ fait de l'espace topologique quotient
$X/G$ un espace
domanial, et il est immédiat que 
$V\mapsto \Lambda_v(V)$ est un faisceau
pour la G-topologie sur $X/G$ ; par conséquent, 
$(X/G,\mathscr D, \Lambda_c)$ est un espace domanial
enrichi, et il résulte des définitions que 
$p\colon X\to X/G$ est un morphisme $G$-invariant
d'espaces domaniaux enrichis. Si $\phi$
est un morphisme $G$-invariant de $X$ vers un espace
domanial enrichi $Z$, il se factorise
de manière unique par une application 
continue de $X/G$ vers $Z$, qui est un morphisme
d'espaces domaniaux enrichis par
définition de la structure domaniale enrichie sur 
$X/G$. Par conséquent $p$ fait bien de $X/G$
le quotient catégorique de $X$ par $G$
dans la catégorie des espaces domaniaux enrichis. 

Soit $V$ appartenant à $\mathscr V$. 
L'image réciproque
$p^{-1}(p(V))=\bigcup_{g\in G}g(V)$ est alors
une partie $c$-linéaire par morceaux de
$X$, si bien que $p(V)$ est un domaine 
compact de $X/G$ ; puisque
$p_{|V}$ est injective et que $V$
est compact, $p$ induit un homéomorphisme 
de $V$ sur $p(V)$. 
Soit $E$ une partie de $p(V)$. 
Pour que $E$ soit un domaine de $p(V)$ (c'est-à-dire
de $X/G$), il faut et il suffit que 
$p^{-1}(E)\cap V$ soit un domaine de $V$. Cette 
condition est en effet clairement nécessaire. 
Vérifions qu'elle est suffisante. 
Supposons donc que 
$p^{-1}(E)\cap V$ est un domaine de $V$. 
Il possède alors un G-recouvrement par des domaines
compacts $E_i$ de $V$. Comme $p$ induit un homéomorphisme
de $V$ sur $p(V)$, les $p(E_i)$ forment un G-recouvrement 
de $p(E)$. De plus on a pour tout $i$ 
l'égalité $p^{-1}(p(E_i))=\bigcup_g g(E_i)$, si bien que
$p^{-1}(p(E_i))$ est un domaine compact de $X$. En conséquence
chacun des $p(E_i)$ est un domaine compact de $X/G$, et $p(E)$
est en conséquence lui-même un domaine compact de $X/G$.

Soit $E$ un domaine
de $p(V)$ et soit $u$ une
application de
$E$ vers
$\rpos$. Posons
$F=p^{-1}(E)\cap V$ 
et $v=u\circ p_{|F}$. 
L'application
$u$
appartient à
$\Lambda_c(E)$
si et seulement si $v$
appartient à $\Lambda_c(F)$. 
C'est en effet clairement nécessaire. 
Supposons maintenant que $v$
appartient à $\Lambda_c(F)$
et montrons que $u\in \Lambda_c(E)$. 
On peut raisonner G-localement sur $E$, donc le supposer compact. 
Dans ce cas $p^{-1}(E)$ est G-recouvert
par les domaines $\bigcup_{g\in G}
g(F)$, et pour tout $g$ 
la restriction de $u\circ p$ à $g(F)$ 
et égale à $v\circ g^{-1}$, et appartient donc
à $\Lambda_c(g(F))$. Il s'ensuit
que $u\circ p$ appartient à
$\Lambda_c(p^{-1}(E))$, et donc
que $u\in \Lambda_c(E)$. 

Par conséquent, $p$ induit un isomorphisme
d'espaces domaniaux enrichis de $V$ sur $p(V)$ ; 
en particulier, $p(V)$ est un espace
$c$-linéaire par morceaux. 

Enfin comme tout point de $X$ possède, comme
on l'a vu au début de la preuve de (A), un voisinage
$c$-linénaire par morceaux compact et $G$-invariant
qui est une union finie d'éléments de $\mathscr V$, les
domaines de $X/G$ de la forme $p(V)$
avec $V\in \mathscr V$ constituent un G-recouvrement
de $X/G$. Par conséquent ce dernier est 
$c$-linéaire par morceaux, ce qui
achève de montrer (B) ; et sa structure
domaniale enrichie est entièrement
déterminée par celle des $p(V)$ pour $V\in \mathscr V$,
ce qui montre la première
assertion de (C) ; la seconde en découle au vu de l'assertion  (2) du lemme 
\ref{image-pl-propre}.
\end{proof}

\section{Géométrie analytique : généralités}
Comme en \ref{subsection-valuations-bridees}, nous désignons par $k$
un corps ultramétrique complet et par $\Gamma$ un sous-groupe
de $\rpos$ tel que $\Gamma\cdot \abs{k^\times}\neq\{1\}$ ; nous utiliserons les conventions de
\ref{def-red-graduees} en matière d'algèbre commutative graduée. 

\subsection{Rappels et conventions}
Nous travaillerons avec
la théorie des espaces $k$-analytiques \textit{au sens de Berkovich}, et plus
précisément de l'article \cite{berkovich1993}.

\subsubsection{}
Nous utiliserons librement dans ce contexte la théorie de la dimension 
(voir \cite{ducros2007}), celle des composantes irréductibles (\cite{ducros2009}, section 4), 
tout ce qui a trait aux propriétés usuelles d'algèbre commutative
(\cite{ducros2009}, section 3 et \cite{ducros2018}, chapitre 2), et la théorie de la platitude
(\cite{ducros2018}, notamment le chapitre 4).

Si $X$ est un espace $k$-analytique et si $x$ est un point de
$X$, le corps résiduel complété de $x$ sera noté $\hr x$, et $\hrt x$
désignera son corpoïde résiduel (qui est $\rpos$-gradué).
Pour tout morphisme
$f\colon Y\to X$ entre espaces $k$-analytiques, la fibre de $f$ en $x$ pourra être notée
$f^{-1}(x)$ aussi bien que $Y_x$ ; c'est un espace $\hr x$-analytique. La dimension locale
$\dim_y Y_x$ sera également notée $\dim_y f$ et appelée dimension de $f$ en $y$, ou dimension relative de
$Y$ sur $X$ en $y$. 

Nous noterons $d_k(x)$ le degré de transcendance de $\hrt x $ sur $\widetilde k$, qui
peut également se décrire en termes plus classiques comme la somme du degré de transcendance résiduel (au sens
usuel) de $\hr x$ sur $k$ et du rang rationnel de $\abs{\hr x^\times}/\abs{k^\times}$. L'intérêt de cet invariant
est l'égalité $\dim X=\sup_{x\in X}d_k(x)$. 

\subsubsection{}
Nous nous servirons également de la notion d'espace $k$-analytique
\textit{$\Gamma$-strict} introduite au chapitre 3 de \cite{ducros2018}. Informellement, 
un espace est $\Gamma$-strict s'il admet une description dont tous les paramètres réels
appartiennent à $\Gamma$. Ainsi tout espace $k$-analytique est $\rpos$-strict, et si la valuation
de $k$ n'est pas triviale, un espace $k$-analytique est $\{1\}$-strict si et seulement s'il est strict. 
(La condition que $\abs{k^\times}\cdot \Gamma\neq \{1\}$ sert à garantir que tout point
d'un espace $k$-affinoïde $\Gamma$-strict a une base de voisinages affinoïdes $\Gamma$-stricts). 

\subsubsection{}
Nous utiliserons librement la variante $\Gamma$-graduée de la théorie
des réductions de germes d'espaces analytiques due à Temkin 
lorsque $\Gamma=\rpos$ (
\cite{temkin2004} et \cite{ducros2018}, chapitre 3). 
Indiquons simplement ici que si $(X,x)$ est un germe
d'espace $k$-analytique $\Gamma$-strict et séparé,
sa réduction $\widetilde{(X,x)}^\Gamma$ est un ouvert quasi-compact 
et non vide de $\P_{\hrt x^\Gamma/\widetilde k^\Gamma}$
(rappelons que cette dernière notation désigne l'espace des valuations
sur $\hrt x^\Gamma$ qui sont triviales sur $\widetilde k^\Gamma$).

\subsection{La structure domaniale
sur un espace $\Gamma$-strict}

\subsubsection{}\label{domrat-approximation}
Si $X$ est un espace affinoïde $\Gamma$-strict nous dirons qu'un domaine affinoïde
$V$ de $X$
est \textit{$\Gamma$-rationnel} si $V$ peut être défini par une condition de la forme 
\[\abs{f_1}\leq \lambda_1 \abs g\;\text{et}\ldots\;\text{et}\;\abs{f_n}\leq \lambda_n \abs g\] où $g$ et
les $f_i$ sont des
fonctions analytiques sur $X$ sans zéro commun sur $X$ et où les $\lambda_i$ sont des éléments de $\Gamma$.

Supposons que ce soit le cas. Le fait que les $f_i$
et $g$ ne s'annulent pas simultanément sur $X$
assure que $g|_V$ est inversible. Mais cela entraîne aussi, par
compacité de $X$, que si $\epsilon$ est un élément suffisamment
petit de $\rpos$
l'ensemble des points $x$ de $X$ tels
que $\abs{g(x)}\leq \epsilon $ et  $\abs{f_i(x)}\leq \epsilon $ pour tout $i$
est vide. Il en résulte que si l'on se donne une fonction analytique $g'$ sur $X$ suffisamment proche 
de $g$ et, pour tout $i$, une fonction analytique $f'_i$ sur $X$ suffisamment proche de $f_i$, les
$f'_i$ et $g'$ n'ont pas de zéro commun sur $X$ et le domaine affinoïde $\Gamma$-rationnel
défini par
la condition \[\abs{f'_1}\leq \lambda_1 \abs {g'}\;\text{et}\ldots\;\text{et}\;\abs{f'_n}\leq \lambda_n \abs {g'}\]
coïncide avec $V$. 

\subsubsection{}\label{gerr-grau}
Nous utiliserons à plusieurs reprises la version $\Gamma$-stricte du théorème de Gerritzen-Grauert : 
si $X$ est un espace $k$-affinoïde $\Gamma$-strict, tout domaine affinoïde 
$\Gamma$-strict de $X$ est une union finie de domaines $\Gamma$-rationnels. 
À l'aide de ce résultat on démontre que si $X$ est un espace
$k$-analytique $\Gamma$-strict et compact et si $Y$ est un sous-espace analytique fermé de $X$, 
tout domaine analytique $\Gamma$-strict et compact de $Y$ est de la forme $V\cap Y$ pour un certain
domaine analytique $\Gamma$-strict et compact $V$ de $X$. (Sur tous ces points on pourra  par exemple 
consulter \cite{ducros2021a}, 1.18, lemme 1.19
et remarque 1.20).

\subsubsection{}
Soit $X$ un espace $k$-analytique $\Gamma$-strict. On définit le site $X\grot \Gamma$ comme suit : 
\begin{itemize}[label=$\diamond$]
\item ses objets sont les domaines analytiques $\Gamma$-stricts de $X$, et ses flèches les inclusions ; 
\item ses familles couvrantes sont les G-recouvrements. 
\end{itemize}

Ce site raffine le site topologique de $X$ : tout ouvert de $X$ en est un domaine analytique $\Gamma$-strict, et 
tout recouvrement ouvert d'un ouvert de $X$ en est un G-recouvrement.

 \subsubsection{Invariance radicielle de la $G$-topologie}
\label{radiciel-domanial}
Supposons le corps 
$k$ de caractéristique $p>0$, soit $X$
un espace $k$-analytique $\Gamma$-strict et soit
$F$ une extension complète de
$k$ dans laquelle la fermeture radicielle de $k$
est dense. 
Soit $\pi$ la flèche $X_F\to X$. On sait que
le
morphisme $\pi$ induit un homéomorphisme entre les espaces topologiques
sous-jacents à $X_F$ et $X$, et on a un résultat
analogue relatif aux G-topologies
$\Gamma$-strictes.
Plus précisément, soit $V$
une partie de $X_F$. Les assertions
suivantes sont alors équivalentes : 
\begin{enumerate}[i]
\item $V$ est un domaine analytique $\Gamma$-strict de
$X_F$ ; 
\item $\pi(V)$
est un domaine analytique $\Gamma$-strict
de $X$.
\end{enumerate}
En effet, l'implication (ii)$\Rightarrow$(i)
provient de l'égalité
$V=\pi^{-1}(\pi(V))$, due à l'injectivité de $\pi$. 
Supposons
maintenant
que (i) soit vérifiée, et
montrons (ii). 
L'assertion est
locale sur $X\grot \Gamma$,
ce qui permet de supposer $X$ affinoïde ;
posons $A=\mathscr O_X(X)$. 
Notre assertion est également locale sur $V\grot\Gamma$
puisque l'homéomorphisme $\pi$ préserve les G-recouvrements.
Par conséquent
on peut supposer que $V$ est affinoïde et même,
en vertu du théorème de Gerritzen-Grauert (\ref{gerr-grau})
que c'est un domaine
$\Gamma$-rationnel de $X_F$. Choisissons
une description de $V$ par une conjonction d'inégalités
\[\abs{f_1}\leq \lambda_1 \abs g\;\text{et}\ldots
\;\text{et}\;\abs{f_n}\leq \lambda_n \abs g\] où $g$ et
les $f_i$ sont des
fonctions analytiques sur $X_F$ sans zéro commun sur $X_F$
et où les $\lambda_i$ sont des éléments de $\Gamma$.
Comme on peut perturber légèrement $g$ et les $f_i$ sans
que cela change le domaine $V$ décrit par les inégalités
ci-dessus (\ref{domrat-approximation}), on
peut supposer par notre hypothèse
sur $F$ qu'ils appartiennent tous à
$\mathscr A\hotimes_kk^{1/p^m}$
pour un certain $m$
puis, en les remplaçant tous,
ainsi que les $\lambda_i$, par leur puissance $p^m$-ième,
qu'ils appartiennent à $A$. Mais
$\pi(V)$ est alors le domaine $\Gamma$-rationnel
de $X$ défini par la  «même» conjonction
d'inégalités que $V$.

\subsubsection{Structure
domaniale sur un espace analytique}\label{gamma-an-domanial}
Soit $X$ un espace $k$-analytique $\Gamma$-strict. L'ensemble des domaines analytiques $\Gamma$-stricts de $X$
munit l'espace topologique sous-jacent à $X$ d'une structure domaniale, dite $\Gamma$-strictement analytique
(ou éventuellement 
analytique si $\Gamma=\rpos$). Le site domanial correspondant est celui
que nous avons noté $X\grot \Gamma$. 

On dispose ainsi d'un foncteur de la catégorie des espaces
analytiques $\Gamma$-stricts vers celle des espaces 
domaniaux. 

Si $Y$ est un sous-espace analytique fermé de $X$, la structure domaniale sur $Y$ induite par la structure
domaniale $\Gamma$-strictement analytique de $X$ coïncide 
en vertu de \ref{gerr-grau} avec la structure domaniale $\Gamma$-strictement analytique de $Y$. 

Supposons $k$ de caractéristique 
$p>0$ et soit $F$ une extension complète de $k$ dans laquelle la fermeture radicielle de
$k$ est dense. Il résulte de
\ref{radiciel-domanial}
que $\pi$ induit un isomorphisme d'espaces
domaniaux entre $X_F$ et $X$, lorsque ceux-ci sont munis de 
leurs structures domaniales $\Gamma$-strictement analytiques.

\subsubsection{Structure domaniale enrichie  sur un espace
analytique} \label{gamma-an-enrichi}
Soit $X$ un espace $k$-analytique $\Gamma$-strict. 
On dispose sur ce dernier d'une structure d'espace domanial, décrite 
à l'exemple \ref{gamma-an-domanial}. On en fait un espace domanial enrichi en 
le munissant du faisceau $(\Gamma\cdot \abs{\mathscr O_X^\times})^\Q$
(c'est-à-dire de la faisceautisation
sur le site $X\grot\Gamma$ du préfaisceau
$V\mapsto
(\Gamma\cdot \abs{\mathscr O_X^\times(V)})^\Q$). 
Nous dirons que la structure domaniale enrichie
correspondante est la structure domaniale
enrichie $\Gamma$-strictement $k$-analytique. 

On dispose ainsi d'un foncteur de la catégorie des espaces
analytiques $\Gamma$-stricts vers celle des espaces 
domaniaux enrichis.

Si $Y$ est un sous-espace analytique fermé de $X$, 
sa structure domaniale enrichie
$\Gamma$-strictement $k$-analytique coïncide avec
la structure domaniale enrichie héritée de celle $X$ : 
c'est une conséquence de l'assertion analogue
pour les structures domaniales (non enrichies) et du
fait que si $V$ est un domaine affinoïde $\Gamma$-strict
de $X$, toute fonction analytique $f$
sur $Y\cap V$ se relève en une fonction analytique sur $V$, 
inversible au voisinage de $Y\cap V$ lorsque $f$
est inversible.

\subsubsection{Invariance par extension radicielle}
\label{radiciel-enrichi}
Supposons $k$
de caractéristique $p>0$, 
soit $X$ un espace $k$-analytique $\Gamma$-strict
et soit $F$ une extension complète de $k$ 
dans laquelle la clôture radicielle de $k$
est dense. Soit $\pi$ le morphisme
canonique $X_F\to X$. Nous avons vu 
que $\pi$ induit un isomorphisme entre
les espaces domaniaux $X_F$ et $X$. Mais 
c'est en fait même un isomorphisme d'espaces domaniaux enrichis
(les structures en jeu étant bien entendu les structures
$\Gamma$-strictement analytiques). 
Pour le voir, il suffit en effet de vérifier que si
$V$ est un domaine affinoïde $\Gamma$-strict de $X$ et
si $f$ est une fonction analytique inversible sur 
$V_F$, alors $\abs f=\lambda\circ \pi$ pour une certaine
section $\lambda$
de $(\Gamma\cdot \abs{\mathscr O_X^\times})^\Q$.
Or comme $\abs f$ ne change pas par petite perturbation 
de la fonction 
inversible $f$, on peut supposer que celle-ci appartient à
$\mathscr O_X(V)\hotimes_kk^{1/p^n}$ pour un certain $n$. 
Mais $f^{p^n}$ appartient dans ce cas à $\mathscr O_X(V)$, 
et l'assertion souhaitée est alors évidente.

\subsection{Squelettes d'un espace analytique}
Nous allons utiliser systématiquement la géométrie $c$-linéaire
par morceaux développée ci-dessus à la sous-section 
\ref{pl-abstrait}, mais en fixant désormais la valeur de $c$ qui sera pour toute la suite égale
à $(\Q,(\abs{k^\times}\cdot \Gamma)^\Q))$.

\subsubsection{}\label{sigma-enrichi}
Si $X$ est un espace $k$-analytique $\Gamma$-strict, il possède une structure naturelle 
d'espace domanial enrichi (\ref{gamma-an-enrichi}). 

Si $\Sigma$ est une partie localement fermée de $X$, elle hérite donc d'une structure
domaniale enrichie (\ref{locferm-enrichi}). 
Concrètement, un domaine de $\Sigma$ en est un sous-ensemble G-recouvert par des parties de la forme
$V\cap \Sigma$ où $V$ est un domaine analytique $\Gamma$-strict de $X$, et le faisceau structural de $\Sigma$ est le 
sous-faisceau domanial de $\mathscr C^0(\cdot, \rpos)$ engendré par les $r
\cdot\abs f_{|\Sigma \cap V}^q\colon \Sigma\cap V\to \rpos$
où $V$ est un domaine analytique $\Gamma$-strict de $X$, où
$f$ une fonction analytique inversible sur $V$, où $r$ est un élément de $\Gamma^\Q$ et $q$ un rationnel.

\begin{defi}\label{def-squelette}
Soit $X$ un espace $k$-analytique $\Gamma$-strict
et soit
$\Sigma$ une partie localement fermée de $X$. On dit que $\Sigma$ est un \textit{$c$-squelette}
de $X$ si sa structure d'espace domanial enrichi décrite en \ref{sigma-enrichi} en fait un espace
$c$-linéaire par morceaux. 
\end{defi}

\subsubsection{}
Soit $X$ un espace $k$-analytique $\Gamma$-strict
et soit $\Sigma$ une partie localement fermée
de $X$.

Si $\Sigma$ est un $c$-squelette
de $X$, une partie localement fermée
$\mathrm T$ de $\Sigma$ est un $c$-squelette de $X$ si
et seulement si c'est une partie $c$-linéaire par morceaux
de $\Sigma$ : cela résulte du lemme \ref{c-lin-dom}. 

\paragraph{}\label{squel-imm-ferm}
Si $Y$ est un sous-espace analytique fermé de $X$
contenant $\Sigma$, alors $\Sigma$ est un $c$-squelette
de $X$ si et seulement si c'est un $c$-squelette de $Y$ : 
cela résulte du fait que la formation 
de la structure domaniale enrichie $\Gamma$-strictement
$k$-analytique est compatible aux immersions
fermées, comme expliqué à la fin de l'exemple
\ref{gamma-an-enrichi}. 

\paragraph{}
Supposons que $\Sigma$ soit un $c$-squelette et soit
$x\in \Sigma$. Soit $d$ la dimension 
de $\Sigma$ en $x$. Il existe alors un domaine
analytique $\Gamma$-strict et compact $V$
de $X$
contenant $x$ et une famille finie de fonctions inversibles
$(f_1,\ldots, f_n)$ sur $V$ telles que 
le $n$-uplet
$\abs f:=(\abs{f_1},\ldots, \abs{f_n})$ induise
un isomorphisme $(\Sigma\cap V)\simeq P$ où $P$ 
est un $c$-polytope compact de $\rposn n$ de dimension 
$d$ en $\abs {f(x)}$. Soit $W$ un voisinage analytique
$\Gamma$-strict compact
de $x$ dans $V$. L'image $\abs f(W)$ est un $c$-polytope
compact de $\rposn n$ (\cite{ducros2012b}, théorème 1.2), 
qui contient le voisinage $\abs f(W\cap \Sigma)$ de $\abs {f(x)}$
dans $P$. Il est donc de dimension $\geq d$. Ceci valant pour
tout $W$, il résulte du théorème 3.4 de
\cite{ducros2012b} que $d_k(x)\geq d$. 

\subsubsection{Fonctorialité 
des $c$-squelettes}\label{fonctor-csquel}
Soit $f\colon Y\to X$ un morphisme d'espaces
$k$-analytiques, soit $\Tau$ un $c$-squelette de $Y$
et soit $\Sigma$ un $c$-squelette de $X$ tel que 
$f(\Tau)$
soit contenue dans
$\Sigma$. Par propriété universelle
de la structure domaniale enrichie induite, l'application 
$\Tau\to \Sigma$ induite par $f$ est un morphisme d'espaces
domaniaux enrichis, c'est-à-dire ici une application 
$c$-linéaire par morceaux. 

\subsubsection{}
La définition de $c$-squelette que nous venons de donner est 
étroitement apparentée à celle de \cite{ducros2012b}, sans lui
être équivalente. La différence de fond
est que nous exigeons dans
\cite{ducros2012b} que tout point $x$ d'un $c$-squelette de
$X$ soit d'Abhyankar de rang maximal, c'est-à-dire satisfasse
l'égalité $d_k(x)=\dim_x X$. Nous ne demandons rien de tel ici, 
et ce relâchement est crucial pour
pouvoir espérer des résultats de stabilité
par image directe un tant soit peu généraux 
(penser déjà
au cas d'une immersion fermée,
en l'occurrence traité plus haut en \ref{squel-imm-ferm}). 

Nous allons maintenant donner un critère qui nous sera
très utile
en pratique pour montrer qu'une partie d'un espace
$k$-analytique $\Gamma$-strict en est un $c$-squelette.

\begin{lemm}\label{critere-squelette}
Soit $X$ un espace $k$-analytique $\Gamma$-strict et soit $\Sigma$ une partie localement fermée
de $X$. Notons $\mathscr O_X(\Sigma)$ la colimite des $\mathscr O_X(V)$ où $V$
parcourt l'ensemble des domaines analytiques $\Gamma$-stricts de $X$ contenant $\Sigma$. 
Supposons qu'il existe un sous-ensemble $E$ de $\mathscr O_X(\Sigma)^\times$ et une famille finie
$(f_1,\ldots,f_n)$ d'éléments de $E$ satisfaisant les propriétés suivantes : 

\begin{enumerate}[A]
\item pour toute famille finie $(g_1,\ldots, g_s)$ d'éléments de $E$, 
la famille
\[(\abs{f_1},\ldots, \abs{f_n}, \abs{g_1},\ldots, \abs{g_s})\] induit un homéomorphisme de $\Sigma$
sur une partie $c$-linéaire par morceaux de $\rposn {n+s}$ ; 
\item pour tout $x\in \Sigma$, le sous-groupe de $\hr x^\times$ engendré par les
$h(x)$ pour $h\in E$ est dense. 
\end{enumerate}
Le sous-ensemble $\Sigma$ de $X$ est alors un $c$-squelette. 
\end{lemm}

\begin{proof}
L'hypothèse (A) assure 
en particulier que $\abs f:=(\abs{f_1}, \ldots, \abs{f_n})$ induit un homéomorphisme de $\Sigma$
sur une partie $c$-linéaire par morceaux $P$ de $\rposn m$. On munit $\Sigma$ de sa structure d'espace $c$-linéaire
par morceaux $\mathsf L$ déduite de celle de $P$ \textit{via}
$\abs f$, et nous allons montrer que cette structure coïncide avec sa structure domaniale enrichie $\mathsf D$.

\paragraph{Identité des structures domaniales}
Soit $\mathrm T$ une partie de $\Sigma$.

Supposons que $\mathrm T$ est $c$-linéaire
par morceaux au sens de $\mathsf L$ ; nous allons montrer que c'est un domaine au sens de $\mathsf D$.
Par définition, $\mathrm T$ s'écrit $\abs f_{|\Sigma}^{-1}(Q)$ pour une certaine partie
$Q$ de $P$
qui est $c$-linéaire par morceaux, ce qui veut
dire que $Q$ est une union localement finie de $c$-polytopes ; 
comme une union localement finie de domaines fermés au sens de $\mathsf D$ est un domaine au sens de $\mathsf D$, on 
peut supposer que $Q$ est un $c$-polytope. Mais alors $\mathrm T=\abs{f}_{|\Sigma}^{-1}(P)$ est décrit en tant
que sous-ensemble de $\Sigma$ par une combinaison booléenne finie 
d'inégalités de la forme $a\prod \abs{f_i}_{|\Sigma}^{r_i}\leq 1$, où $a\in (\Gamma \cdot \abs{k^\times})^\Q$ et où 
les $r_i$ sont rationnels. C'est donc l'intersection de $\Sigma$ avec un domaine analytique $\Gamma$-strict de $X$,
et c'est dès lors un domaine au sens de $\mathsf D$. 

Réciproquement, 
supposons que $\mathrm T$ est un domaine au sens de $\mathsf D$, nous allons montrer que
$\mathrm T$ est $c$-linéaire par morceaux au sens de $\mathsf L$.
Cette assertion étant G-locale sur $\mathrm T$, on peut supposer que 
$\mathrm T=V\cap \Sigma$ pour un certain domaine affinoïde $\Gamma$-strict $V$ de 
$X$, et il suffit de s'assurer que pour tout $x\in \mathrm T$ il existe un voisinage ouvert $U$ de $x$ dans $\Sigma$ tel
que
$U\cap \mathrm T$ soit $c$-linéaire par morceaux relativement à $\mathsf L$. Soit donc $x\in \mathrm T$. 
Le point $x$ appartient alors à $V$. Comme ce dernier est affinoïde et $\Gamma$-strict, 
sa réduction à la Temkin est de la forme $\mathbf P_{\hrt x/\widetilde k}\{a_1,\ldots, a_m\}$ ;
en relevant les $a_i$
dans $\hr x$
puis en utilisant (B), on en déduit l'existence 
d'éléments $\alpha_1, \ldots, \alpha_m$ de $\langle E\rangle$, tous définis sur un même domaine analytique $\Gamma$-strict $W$
contenant $\Sigma$, et d'un voisinage ouvert $\Omega$ de $x$ dans $W$ tel que $V\cap \Omega$ soit défini comme 
le lieu de validité sur $\Omega$ d'une conjonction 
d'inégalités de la forme $\abs \alpha_i\leq r_i$, où chaque $r_i$ appartient à $\Gamma$. 
Posons $U=\Omega\cap \Sigma$. C'est un voisinage ouvert de $x$ dans $\Sigma$, 
et $\mathrm T\cap U$ est défini comme sous-ensemble de $U$ par la conjonction des inégalités 
$\abs {\alpha_i}_{|U}\leq r_i$. Il suffit pour conclure de s'assurer que pour tout $i$, la fonction 
$\abs{\alpha_i}_{|\Sigma}$ est $c$-linéaire par morceaux au sens de $\mathsf L$. Comme $\alpha_i$ appartient
à $\langle E\rangle$, il suffit finalement de démontrer que $\abs g_{|\Sigma}$ est $c$-linéaire par morceaux
au sens de $\mathsf L$ pour toute $g \in E$. Mais c'est une conséquence immédiate de l'hypothèse (A). 

\paragraph{Identité des faisceaux structuraux}
Soit $\mathrm T$ un domaine de $\Sigma$ (pour l'une ou l'autre des structures $\mathsf L$ et $\mathsf D$, par ce qui précède 
leurs domaines sont les mêmes). Soit $\lambda$ une application continue de $\mathrm T$ vers $\rpos$. 

Supposons que 
$\lambda$ est une section du faisceau structural relatif à $\mathsf D$. Montrons que $\lambda$ est $c$-linéaire
par morceaux relativement à $\mathsf L$. L'assertion étant G-locale, on peut supposer que $\mathrm T$
est égal $V\cap \Sigma$ pour un certain domaine analytique $\Gamma$-strict $V$ de $X$, et que 
$\lambda=r\cdot\abs g_{|\Sigma}^q$ pour une certaine fonction analytique inversible $g$ sur $V$, un certain
$r\in \Gamma^\Q$ et un certain $q\in \Q$.  En raisonnant localement et en
utilisant l'hypothèse (B) on peut de surcroît supposer que $g$ appartient à $\langle E\rangle$, puis à $E$. 
Mais il résulte alors de (A) que $\abs g_{|\Sigma}$ est $c$-linéaire par morceaux au sens de $\mathsf L$,
et il en va dans ce 
alors de
même de $r\cdot\abs g^q_{|\Sigma}$.

Soit $\mu\colon \mathrm T
\to \rpos$ une section $c$-linéaire par morceaux relativement à $\mathsf L$.
Montrons que $\mu$ est une section du faisceau structural 
relatif à $\mathsf D$. L'assertion étant G-locale sur $\mathrm T$, on peut supposer 
que $\mu$ est de la forme $a\cdot\abs{f_1}_{|\mathrm T}^{r_1}\cdot \ldots \cdot\abs{f_n}_{|\mathrm T}^{t_n}$ où
$a\in (\Gamma\cdot\abs{k^\times})^\Q$ et où les $r_i$
sont rationnels. Et $\mu$ est alors par définition une section du faisceau structural relatif à $\mathsf D$. 
\end{proof}

\section{Exemples de squelettes}

Nous allons maintenant présenter quelques exemples
fondamentaux de $c$-squelettes, qui avaient déjà été étudiés par le premier
auteur dans \cite{ducros2012b}. Le théorème \ref{imrec-squel}  ci-dessous, 
dont la preuve occupe la plus grande partie de cette section, est essentiellement 
le théorème 5.1
de \textit{loc. cit.}. Nous avons choisi d'en redonner une démonstration 
adaptée aux définitions et conventions du présent article, avec de surcroît une simplification : 
en fin de preuve, nous substituons à un calcul un peu ingrat et \textit{ad hoc}
de \cite{ducros2012b} (voir 5.6 et le lemme 5.7 de \textit{loc. cit.}) l'utilisation d'une 
variante graduée d'un «lemme de normalisation de Noether torique» de Maclagan et Sturmfels ; 
cette simplification est elle-même tirée
du travail en révision
\cite{chambertloir-d-2012} 
d'Antoine Chambert-Loir et du premier auteur. 

\subsection{Images réciproques du squelette standard de $\gma n$ : les énoncés}

\begin{enonce}[remark]{Le $c$-squelette standard de $\gma n$}
Soit $n$ un entier. Nous désignerons par $S_{n,k}$
le sous-ensemble de $\gma n$ égal 
à $\{\eta_r\}_{r\in \rposn n}$, où $\eta_r$
désigne la (semi)-norme 
\[\sum a_I T^I\mapsto \max \abs{a_I}r^I\]
(en notation multi-indicielle). 
En tant qu'image de la section continue
$r\mapsto \eta_r$
de $\abs T\colon \gma n \to \rposn n$, le sous-ensemble
$S_n$ de $\gma n$ en est une partie fermée. 
Il est immédiat
que les hypothèses (A) et (B) du lemme
\ref{critere-squelette}
sont satisfaites 
lorsqu'on prend $X=\gma n$, $\Sigma
=S_{n,k}$, $E=k[T]\setminus\{0\}$ et $f_i=T_i$ pour
tout $i$ ; par conséquent, 
$S_{n,k}$ est un $c$-squelette. 

\end{enonce}

\subsubsection{}
Soit $L$
une extension complète de $k$. 
Il est immédiat que
la flèche naturelle
$\gmb Ln\to \gma n$
induit un homéomorphisme $S_{n,L}\simeq
S_{n,k}$. 
Par ailleurs, il résulte 
de la définition qu'un point $x$ de $X$ appartient à 
$S_{n,k}$ si et seulement si les $T_i(x)$ sont tous 
non nuls, et si les $\widetilde{T_i(x)}$ 
forment une famille d'éléments de $\hrt x$
algébriquement indépendants sur $\widetilde k$. 
Il s'ensuit que si $\widetilde L$ est algébrique sur
$\widetilde k$ (c'est par exemple le cas dès que
la fermeture algébrique de $k$ est dense dans $L$)
alors $S_{n,L}$ est l'image réciproque
de $S_{n,k}$ sur $\gmb Ln$.

\begin{theo}
[\cite{ducros2012b}, th. 5.1]\label{imrec-squel}
Soit $n$ un entier, 
soit 
$X$ un espace $k$-analytique de dimension
au plus $n$
et soient 
$\phi_1,\ldots, \phi_m$ des morphismes de
$X$ vers $\gma n$. 
Pour toute extension
complète $L$
de $k$, on pose $\Sigma_L=\bigcup_i \phi_i^{-1}(S_{n,L})$
et l'on note $c_L$ le couple
$(\Q, (\Gamma\cdot \abs{L^\times})^\Q)$ (on a donc
$c_k=c$).

\begin{enumerate}[1]
\item Le fermé $\Sigma_k$ de $X$
en est un $c$-squelette. 
\item On a 
$\dim_x \Sigma_k=n$ en tout point
$x$ appartenant à  $\Sigma_k\setminus \partial X$. 
\item Pour tout $i$, l'application 
du $c$-squelette $\phi_i^{-1}(S_{n,k})$ vers $S_{n,k}$ induite
par $\phi$ est $c$-linéaire par morceaux et à fibres
discrètes. 
\item Pour toute extension complète $L$ de $k$
l'application naturelle de
$\Sigma_L$ vers $\Sigma_k$ est $c_L$-linéaire par morceaux, surjective, propre, ouverte et 
à fibres finies. 
\item Si $X$ est compact il existe une extension finie séparable
$F$ de $k$ telle que
l'application naturelle de $\Sigma_L$
vers $\Sigma_F$ soit un isomorphisme
$c_L$-linéaire par morceaux
pour toute extension complète $L$ de $F$. 
\end{enumerate}

\end{theo}

\subsection{Preuve 
du théorème \ref{imrec-squel} dans
le cas où $m=1$}
On suppose que
$m=1$ et l'on écrit $\phi$ au lieu de $\phi_1$. 
L'étude de ce cas particulier est en fait le cœur
de la preuve, et comprend un nombre important
d'étapes.

\subsubsection{}
Nous allons 
tout d'abord
expliquer comment nous
réduire au cas d'un
corps de base parfait. 
Il n'y a rien à faire si
$k$ est de caractéristique nulle. 
Supposons
maintenant qu'il est
de caractéristique $p>0$
et que le théorème a été démontré
sur un corps
parfait. Il est alors en particulier
vrai pour la perfection complétée
$\perf k$ de $k$.

Le squelette $S_{n,\perf k}$ est l'image réciproque
de $S_{n,k}$ sur $\gmb{\perf k}n$, si bien que 
$\Sigma_{\perf k}$ est l'image réciproque de
$\Sigma_k$ sur $X_{\perf k}$. Il s'ensuit en vertu de
\ref{radiciel-enrichi} que
$S_{n,\perf k}\to S_{n,k}$ et
$\Sigma_{\perf k}\to \Sigma_k$
sont des isomorphismes d'espaces domaniaux enrichis. Puisque
$\Sigma_{\perf k}$ est par hypothèse un espace $c$-linéaire
par morceaux, on en déduit que $\Sigma_k$ est $c$-linéaire
par morceaux : c'est donc un $c$-squelette, 
isomorphe à $\Sigma_{\perf k}$ (nous
ferons dans la suite de la preuve 
référence à ce fait en parlant d'\textit{invariance 
radicielle}). Et si $x$ est un point de $\Sigma_k\setminus
\partial X$, son unique antécédent $x'$ 
sur $X_{\perf k}$ est situé sur $\Sigma_{\perf k}
\setminus \partial X_{\perf k}$, et 
$\Sigma_{\perf k}$ est par hypothèse de dimension $n$
en $x'$ ; par conséquent, $\Sigma_k$ est 
de dimension $n$ en $x$.

On dispose 
d'un diagramme commutatif mettant en jeu quatre espaces
$c$-linéaires par morceaux, 
\[
\begin{tikzcd}
\Sigma_{\perf k}\ar[d]\ar[r]&S_{n,\perf k}\ar[d]\\
\Sigma_k\ar[r]&S_{n,k}
\end{tikzcd},\]
dans lequel les deux flèches verticales sont des isomorphismes
$c$-linéaires par morceaux,
et dont la flèche horizontale
du haut est $c$-linéaire par 
morceaux à fibres discrètes, puisque le théorème vaut
sur $\perf k$ ; il s'ensuit que $\Sigma_k\to S_{n,k}$
est $c$-linéaire par morceaux à fibres discrètes.

Soit $L$ une extension complète de $k$. Le diagramme
commutatif 
\[
\begin{tikzcd}
\Sigma_{\perf L}\ar[d]\ar[r]&\Sigma_{\perf k}\ar[d]\\
\Sigma_L\ar[r]&\Sigma_k
\end{tikzcd}\]
met en jeu quatre espaces $c_L$-linéaires par morceaux
(concernant la colonne de droites,
on a procédé à une extension
des paramètres de définition, 
\textit{cf.} \ref{pl-extension-parametres}),
ses deux flèches verticales
sont des isomorphismes par invariance
radicielle, et sa flèche horizontale du haut est $c_L$-linéaire
par morceaux à fibres finies puisque le théorème vaut
sur $\perf k$. Par conséquent, la flèche horizontale
du bas est elle aussi $c_L$-linéaire par morceaux
à fibres finies. 

Enfin, supposons 
que $X$ est compact. Comme le théorème vaut sur 
$\perf k$, il existe une extension finie séparable 
$E$ de $\perf k$ pour laquelle l'énoncé (5) est satisfait. 
Le lemme de Krasner assure l'existence d'une extension finie
séparable $F$ de $k$ telle que $E\simeq \perf F$. Soit
$L$ une extension complète de $F$. 
Dans la catégorie des espaces $c_L$-linéaires par morceaux
on dispose du diagramme commutatif
\[
\begin{tikzcd}
\Sigma_{\perf L}\ar[d]\ar[r]&\Sigma_E\ar[d]\\
\Sigma_L\ar[r]&\Sigma_F
\end{tikzcd},\]
dans laquelle les flèches verticales sont des isomorphismes
par invariance radicielle, et dans laquelle la flèche horizontale du haut en est un par
choix de $E$. La flèche horizontale du bas est donc
également un isomorphisme. 

La validité du théorème sur le corps $\perf k$ entraîne donc
sa validité sur $k$. Il est par
conséquent licite de supposer
$k$ parfait, ce que nous ferons désormais. 

\subsubsection{}
Les assertions (1), (2), (3)
et (4) étant locales sur $X$, on peut
le supposer compact (et
il n'y a donc plus lieu
de spécifier cette hypothèse dans (5)). 
Soit $x\in X$. 
Il suffit de démontrer
que (1), (3), (4) et (5) valent sur un voisinage
analytique compact et $\Gamma$-strict de $X$, 
et
que $\dim_x \Sigma_k=n$
si $x\in \Sigma_k\cap \partial X$. 
Si $x\notin
\Sigma_k$ on peut supposer quitte à restreindre 
$X$ que $\Sigma_k$ est vide, auquel cas
il n'y a rien à démontrer. On suppose
à partir de maintenant que $x\in \Sigma_k$. 
Puisque $\phi(x)\in S_{n,k}$ on a $d_k(x)\geq d_k(\phi(x))=n$, 
et comme $\dim_x X\leq n$ il vient $d_k(x)=n$ et 
$\dim_x X=n$. 

On peut de plus
supposer que $X$ est affinoïde. C'est en effet
clair pour les assertions (1), (3), (4) et (5) qui sont
locales sur $X\grot \Gamma$. Et en ce qui concerne
la dimension de $\Sigma_k $ en $x$ lorsque
$x\notin \partial X$, on remarque que si $x\notin \partial X$
alors $x$ possède un voisinage affinoïde $\Gamma$-strict 
dans $X$, par lequel on peut remplacer $X$.

\subsubsection{Algébrisation de la situation}
Si $Y\hookrightarrow X$ est une nilimmersion,
le théorème vaut pour $X$ si et seulement s'il 
vaut pour $Y$ (voir les considérations à la fin 
de l'exemple \ref{gamma-an-enrichi} ; notons
que $x\in \partial X$ si et seulement si
$x\in \partial Y$). 
Par conséquent, on peut supposer $X$ réduit.

L'anneau local $\mathscr O_{X,x}$ est artinien 
puisque $d_k(x)=\dim_x X=n$ ; étant réduit, c'est un corps. 
Il est en particulier régulier. L'espace $X$ est donc
régulier en $x$, et est dès lors quasi-lisse en $x$ puisque
$k$ est parfait. Quitte à restreindre $X$, on peut donc supposer
que c'est un domaine affinoïde de $\mathscr X\an$ pour un certain
$k$-schéma
$\mathscr X$ de type fini, lisse, intègre et de dimension $n$
(\cite{ducros2012b}, 0.21), et 
$\partial X$ est alors
le bord topologique 
de $X$ dans $\mathscr X\an$. 

Le morphisme $\phi \colon X\to \gma n$ est donné par $n$ fonctions analytiques
inversibles $f_1,\ldots, f_n$. Si $g_1,\ldots, g_n$ sont des fonctions analytiques inversibles
sur $X$ telles que $\abs{f_i-g_i}<\abs {f_i}$ identiquement sur $X$ pour tout indice $i$, alors
$\widetilde{g_i(y)}=\widetilde{f_i(y)}$ pour tout $y\in X$
et tout $i$, et en un point $y$ de $X$
la famille des $\widetilde{g_i(y)}$ est algébriquement indépendante sur $\widetilde k$
si et seulement
si c'est le cas de celle des $\widetilde{f_i(y)}$. Autrement dit, si $\psi$
désigne le morphisme
de $X$ vers $\gma n$ induit par les
$g_i$ alors $\psi^{-1}(S_{n,k})$ est
égal à $\Sigma_k$,
si bien qu'on peut remplacer
$\phi$ par $\psi$. Par approximation des $f_i$ on peut
donc, quitte à restreindre $X$ autour de $x$, 
supposer que $\phi$ est la
restriction à $X$
de l'analytification d'un morphisme
de $\mathscr X$ vers $\gmc n$, encore noté $\phi$, et dont 
les composantes sont encore notées $f_1,\ldots, f_n$. Remarquons
que $\phi$ est dominant,
puisque $\phi(x)\in S_{n,k}$ par hypothèse. 

Pour toute extension complète
$L$ de $k$ on notera
désormais $\Sigma_L$ l'image réciproque
$\phi^{-1}(S_{n,L})$ sur $\mathscr
X\an_L$ tout entier. Nous allons
montrer que $\Sigma_k$ est un
$c$-squelette de $\mathscr X\an$ purement de
dimension $n$, que $\Sigma_k\to 
S_{n,k}$ est $c$-linéaire par morceaux
à fibres finies,
que $\Sigma_L\to \Sigma_k$ est 
$c_L$-linéaire par morceaux, surjective, propre, ouverte et à fibres finies pour toute
extension complète $L$ de $k$, et qu'il existe une extension 
finie séparable $F$ de $k$ telle que
la flèche 
de changement de base $\Sigma_L\to \Sigma_F$ soit un isomorphisme pour toute
extension complète $L$ de $F$. Cela permettra de conclure, 
par intersection avec $X$ (qui n'interviendra désormais
plus, non plus que $x$). 

Remarquons que $d_k(y)=n$ pour tout $y\in \Sigma_k$, si bien que tout
point de $\Sigma_k$ est situé au-dessus du point générique de $\mathscr X$ : on pourra
donc dans la suite remplacer $\mathscr X$ par n'importe lequel de ses ouverts de Zariski non vides, 
et toute fonction rationnelle non nulle sur $\mathscr X$ appartient ainsi à 
$\mathscr O_{\mathscr X\an}(\Sigma_k)^\times$.

Si $K$ désigne la fermeture algébrique
de $k$ dans $\mathscr O(\mathscr X)$, le morphisme
$\phi$ se factorise par $\mathbf G_{\mathrm m, K}^n$, et 
$\Sigma_k$ peut tout aussi bien se décrire comme 
$\phi^{-1}(S_{n,K})$ sur $\mathscr X\an$. 
De plus, pour toute extension complète $L$ de $k$,
le schéma $\mathscr X_L$ est une union disjointe finie 
de schémas de la forme $\mathscr X\times_KE$ où 
$E$ est une extension complète de $K$. On peut dès lors
pour démontrer le théorème remplacer
$k$ par $K$, c'est-à-dire supposer
$\mathscr X$ géométriquement intègre. 

\subsubsection{Description détaillée de $\Sigma_k$}\label{description-sigma}
Commençons par introduire un peu
de vocabulaire. Soit $(h_1,\ldots, h_\ell)$ une famille de fonctions rationnelles non nulles
sur $\mathscr X$. Nous dirons que $(h_1,\ldots, h_\ell)$
est \textit{$\Sigma_k$-fidèle} 
si la restriction de
$(\abs{h_1},\ldots, \abs{h_\ell})$ à $\Sigma_k$
est injective, et qu'il est \textit{admissible} s'il existe un ouvert de Zariski non vide
$\mathscr U$ de $\mathscr X$ tel que  $(h_1,\ldots, h_\ell)$ induise
une immersion fermée de $\mathscr U$ dans $\gmc{n+\ell}$  ; 
dans ce cas $\abs h:=(\abs{h_1},\ldots, \abs{h_\ell})$ 
est une application propre de $\mathscr U\an$
vers $\rposn m$, et sa restriction au fermé 
$\Sigma_k$ de $\mathscr U\an$ est encore propre.

La famille $(h_1,\ldots, h_\ell)$ peut toujours être complétée en une famille admissible. 
En effet, choisissons un ouvert 
affine non vide $\mathscr V$ 
de $\mathscr X$ tel que 
les $h_i$ soient définies et inversibles
sur $\mathscr V$. Choisissons une famille
$a_1,\ldots, a_r$ de fonctions non nulles engendrant la $k$-algèbre $\mathscr O_{\mathscr X}(\mathscr V)$. 
Si $\mathscr U$ désigne l'ouvert d'inversibilité simultanée des $a_i$
sur $\mathscr V$ alors les
$a _\ell^{\pm1}$ engendrent
$\mathscr O_{\mathscr X}(\mathscr U)$, 
et $(h_1,\ldots, h_\ell, a_1,\ldots, a_r, a_1^{-1},\ldots, a_r^{-1})$ induit 
par conséquent une immersion fermée de $\mathscr U$ dans
$\gmc{\ell+r}$. 

Par ailleurs, il est clair que si $(h_1,\ldots, h_\ell)$ est $\Sigma_k$-fidèle,
alors $(h_1,\ldots, h_\ell,h)$ est encore $\Sigma_k$-fidèle pour toute
fonction rationnelle non nulle $h$ sur $\mathscr X$.

L'ingrédient essentiel de notre preuve est le théorème 
2.8 (1) de \cite{ducros2012b} qui assure l'existence d'une famille finie $g_1,\ldots, 
g_\ell$ de fonctions
rationnelles non nulles sur $\mathscr X$ qui séparent les prolongements
à
$\kappa(\mathscr X)$ de toute valuation de Gau\ss~ 
sur son sous-corps $k(f_1,\ldots, f_n)$ ; par construction, 
$(f_1,\ldots, f_n,g_1,\ldots, g_\ell)$ est $\Sigma_k$-fidèle. 
Par ce qui précède il est loisible, quitte à agrandir la famille des $g_i$, 
de supposer de plus que $(f_1,\ldots, f_n,g_1,\ldots, g_\ell)$ est admissible.

Soit $(h_1,\ldots, h_p)$ une famille
finie de fonctions rationnelles non nulles sur $\mathscr X$. Nous allons montrer que
$(\abs{f_1},\ldots, \abs{f_n},\abs{g_1},\ldots,
\abs{g_\ell}, \abs{h_1},\ldots, \abs{h_p})$
induit un homéomorphisme entre $\Sigma_k$ et un $c$-polytope de $\rposn{n+\ell+p}$
purement de dimension $n$ : 
il s'ensuivra en vertu du lemme \ref{critere-squelette} que $\Sigma_k$ est un squelette,
purement de dimension $n$.
La restriction de $(\abs{f_1},\ldots, \abs{f_n},\abs{g_1},\ldots, \abs{g_\ell}, \abs{h_1},\ldots, \abs{h_p})$ 
à $\Sigma_k$ est injective et propre, puisque c'est déjà le cas  pour
 $(\abs{f_1},\ldots, \abs{f_n},\abs{g_1},\ldots, \abs{g_\ell})$. Il suffit donc de montrer que 
 l'image de $\Sigma_k$ sous
 $(\abs{f_1},\ldots, \abs{f_n},\abs{g_1},\ldots, g_\ell, \abs{h_1},\ldots, \abs{h_p})$
 est un $c$-polytope de $\rposn {n+\ell+p}$. Ce dernier point peut se vérifier
 après agrandissement de la famille des $h_i$, ce qui autorise
 à supposer que 
 $(f_1,\ldots, f_n,g_1,\ldots, g_\ell, h_1,\ldots, h_p)$ est admissible ; pour simplifier les notations, 
on peut alors inclure les $h_i$ dans la famille des $g_i$. 
En restreignant $\mathscr X$
si besoin, on se ramène au cas où
$(f_1,\ldots, f_n,g_1,\ldots, g_\ell)$ induit une immersion fermée de $\mathscr X$
vers
$\gmc{n+\ell}$. Posons $\tau:=(\abs{f_1},\ldots, \abs{f_n},\abs{g_1},\ldots,
\abs{g_\ell})$.
C'est une application propre de $\mathscr X\an$ vers 
$\rposn{n+\ell}$, et 
il reste à montrer 
que $\tau(\Sigma_k)$ est un $c$-polytope
purement de dimension $n$. 

D'après le théorème  1.2 de \cite{ducros2012b}, il existe un $c$-polytope $P$ de dimension $\leq n$
de $\rposn {n+\ell}$ tel que $\tau(\mathscr X\an)=P$. 
Soit $\pi\colon \rposn {n+m}\to \rposn n$ la projection sur les
$n$ premières coordonnées. 
Soit $Q$ l'ensemble des points $z$ de $P$ tels que l'image de tout voisinage $c$-polytopal
de $z$ dans $P$ par $\pi$ soit de dimension $n$. C'est un
$c$-polytope purement de dimension $n$ : si l'on choisit une décomposition cellulaire
$\mathscr C$ de $P$, c'est la réunion des cellules
$C\in \mathscr C$ telles que $\dim C=n$ et $\pi_{|C}$ soit injectif ; 
remarquons que $\pi \colon Q\to \rposn n$ est à fibres finies. 

Nous allons
démontrer que $\tau(\Sigma	_k)=Q$. Pour cela, nous allons utiliser
la caractérisation suivante de $\Sigma_k$, fournie
par le théorème 3.4
de \cite{ducros2012b}, en posant
$\abs f=(f_1,\ldots, f_n)$ : c'est l'ensemble des points $y$ de $\mathscr X\an$
tels que pour tout voisinage analytique $\Gamma$-strict compact $V$
de $y$ dans $\mathscr X\an$, l'image $\abs f(V)$ (qui est un $c$-polytope
de $\rposn n$ d'après le théorème 1.2 de \cite{ducros2012b}) soit de dimension $n$. 

Soit $z\in P\setminus Q$. Par définition, il existe un voisinage $c$-polytopal
compact $R$ de $z$ dans $P$ dont la projection sur $\rposn n$ est de dimension $<n$. 
L'image réciproque $\tau^{-1}(R)$  est un domaine
analytique $\Gamma$-strict compact de $\mathscr X\an$, dont l'image
sous $\abs f$ est un $c$-polytope
de dimension $<n$. Par conséquent, l'intérieur topologique de
$\tau^{-1}(R)$  dans $\mathscr X\an$ ne rencontre pas $\Sigma_k$.
Il s'ensuit que l'intérieur topologique de
$R$ dans $P$ ne rencontre pas $\tau(\Sigma_k)$. En particulier, $z\notin\Sigma_k$
et $\Sigma_k\subset Q$. 

Réciproquement, soit $z\in Q$ et soit $R$ un voisinage $c$-polytopal compact de 
$z$ dans $P$. Par définition de $Q$ , l'image de $R$ dans $\rposn n$ est de dimension $n$. 
Par conséquent, $\abs f(\tau^{-1}(R))$ est de dimension $n$. Ceci implique que le domaine
analytique compact et $\Gamma$-strict $\tau^{-1}(R)$ rencontre $\Sigma_k$, et partant que
$R$ rencontre $\tau(\Sigma_k)$. Ceci valant pour tout $R$, le point $z$ est adhérent à $\tau(\Sigma_k)$. 
Par propreté de $\tau|_{\Sigma_k}$, le point $z$ appartient à $\tau(\Sigma_k)$.

On a donc bien $\tau(\Sigma_k)=Q$, si bien que $\Sigma_k$ est un $c$-squelette
purement de dimension $n$, comme annoncé. 
Dans le diagramme commutatif 
\[\begin{tikzcd}
\Sigma_k\ar[r, "\tau"]\ar[d,"\phi"']&Q\ar[d,"\pi"]\\
S_{n,k}\ar[r,"\sim"]&\rposn n\end{tikzcd}\]
les flèches horizontales sont des isomorphismes $c$-linéaires par morceaux, 
et la projection $\pi$ est à fibres finies. Il s'ensuit que $\Sigma_k\to S_{n,k}$
est $c$-linéaire par morceaux et à fibres finies. 

\subsubsection{Propriétés supplémentaires de la flèche
$\Sigma_k\to S_{n,k}$}\label{sigma-s-fini}
Le morphisme $\phi$ étant dominant, il existe un ouvert non vide
$\mathscr U$ de $\gmc n$ tel que $\phi^{-1}(\mathscr U)\to \mathscr U$ soit fini et plat
de degré $>0$ ; ceci entraîne que $\phi^{-1}(\mathscr U\an)\to \mathscr U\an$ est 
propre, ouvert, surjectif  et à fibres finies ; puisque $S_{n,k}\subset \mathscr U\an$, on en déduit que 
$\Sigma_k\to S_{n,k}$ est propre, ouvert, surjectif (et à fibres finies, ce que nous savions déjà). 

\subsubsection{Effet de l'extension des scalaires}
Soit maintenant $L$ une extension compète de $k$. 
Conservons les notations $f_1,\ldots, f_n, g_1,\ldots, g_\ell$
introduites au paragraphe 
\ref{description-sigma}.
Le raisonnement qu'on y suit montre
qu'il existe une famille finie de fonctions rationnelles
$b_1,\ldots, b_r$ sur $\mathscr X_L$
sur $\mathscr X_L$ telles que
$(\abs{f_1},\ldots, \abs{g_1},
\ldots, \abs{g_\ell}, \abs{b_1},\ldots, \abs {b_r})$
induise un homéomorphisme
entre $\Sigma_L$ et
un $c$-polytope $\Pi$ de $\rposn{n+m+r}$,
homéomorphisme qui définit la structure $c_L$-linéaire
par morceaux de $\Sigma_L$
(si les $g_i$ séparent encore les prolongements
à $\kappa(\mathscr X_L)$ des valuations de Gauß sur $L(f_1,\ldots f_n)$  on peut prendre pour
$(b_j)$ la famille vide ; mais en général, on a besoin de fonctions supplémentaires). On déduit alors du diagramme commutatif
\[\begin{tikzcd}\Sigma_L\ar[r]\ar[d]&\Pi\ar[d]\\
\Sigma_k\ar[r]&Q\end{tikzcd},\]
dont la flèche verticale de droite est induite par la projection sur les $n+m$ premières coordonnées
et dont les flèches horizontales sont des isomorphismes,
que $\Sigma_L\to \Sigma_k$ est $c_L$-linéaire par morceaux. L'application 
$S_{n,L}\to S_{n,k}$ est bijective, et $\Sigma_L\to S_{n,L}$ est à fibres finies ; il s'ensuit que $\Sigma_L\to S_{n,k}$
est à fibres finies, et il en va \textit{a fortiori} de même
de $\Sigma_L\to \Sigma_k$. Soit $y$ un point de $\Sigma_k$ et soit $s$
son image sur $S_{n,k}$ ; remarquons que
$\hr y$ est une extension finie de $\hr s$
(\ref{sigma-s-fini}). Soit $t$ l'unique antécédent de $s$ sur $S_{n,L}$. L'ensemble
des antécédents de $y$ sur $\Sigma_L$ s'identifie alors à $\mathscr M(\hr y\otimes_{\hr s}\hr t)$, 
qui est non vide ; ainsi, $\Sigma_L\to \Sigma_k$ est surjective. Remarquons
par ailleurs que cette application est propre 
car $\Sigma_L$ est fermé dans $\mathscr X_L\an$, lequel est topologiquement propre sur 
$\mathscr X\an$. Il reste à s'assurer qu'elle est ouverte ; ce sera fait à la fin du paragraphe
suivant. 

\subsubsection{Stabilisation après une extension finie et conclusion}
La stabilisation de $\Sigma_k$ après une extension
finie séparable de $k$ provient de 
ce qui précède et de l'existence d'une extension 
finie séparable $F$ de $k$ et
d'une famille finie de fonctions 
rationnelles non nulles sur $\mathscr X_F$
séparant \textit{universellement}
(c'est-à-dire après toute extension valuée
de $F$) 
les prolongements des valuations de Gauß 
de $F(f_1,\ldots, f_n)$ à $\kappa(\mathscr X_F)$
(\cite{ducros2012b}, théorème 2.8 b). Remarquons que $\Sigma_F$ est 
l'image réciproque de $\Sigma_k$ sur $\mathscr X\an_F$, et que le
morphisme fini et plat $\mathscr X\an_F\to \mathscr X\an$ est ouvert ; 
par conséquent, $\Sigma_F\to \Sigma_k$ est ouverte.

Soit $L$ une extension complète quelconque de $k$ et soit $M$ une extension complète composée de $F$ et $L$. 
Dans le diagramme commutatif 
\[\begin{tikzcd}
\Sigma_M\ar[d]\ar[r]&\Sigma_L\ar[d]\\
\Sigma_F\ar[r]&\Sigma_k\end{tikzcd}\]
la flèche verticale de gauche est un homéomorphisme, 
la flèche horizontale du haut est surjective, et la flèche horizontale du bas est ouverte.  
Il s'ensuit que $\Sigma_L\to \Sigma_k$ est ouverte. 

\subsection{Preuve du théorème \ref{imrec-squel} dans le cas général}
On ne suppose plus désormais que $m=1$, et l'on reprend donc les notations
$\phi_1,\ldots, \phi_m$ de l'énoncé.

\subsubsection{}
Il suffit de démontrer que $\Sigma_k$ est un $c$-squelette : les assertions
(2), (3), (4) et (5) découleront alors du fait qu'elles valent dans le cas $m=1$ (qu'on vient de traiter) et que $\Sigma_k$
est réunion finie de ses parties $c$-linéaires par morceaux fermées $\phi_i^{-1}(S_{n,k})$. On peut supposer
(en raisonnant G-localement) que $X$ est affinoïde. 

Il suffit même simplement de montrer que $\Sigma_k$ est contenu
dans un $c$-squelette $\mathrm T$. Supposons en effet que ce soit le cas. 
On sait par le cas $m=1$ déjà traité que chacun des $\phi_j^{-1}(S_{n,k})$ est un $c$-squelette, donc est une partie
$c$-linéaire par morceaux fermée de $\mathrm T$. Leur réunion est alors une partie $c$-linéaire par morceaux de $\mathrm T$, 
donc un $c$-squelette. 

\subsubsection{}
Chacun des $\phi_i$ est donné par $n$ fonctions analytiques inversibles $f_{i1},\ldots, f_{in}$. Renumérotons pour simplifier les $f_{ij}$
avec un seul indice, ce qui fournit une famille $(f_1,\ldots, f_N)$ de fonctions inversibles sur $X$. L'ensemble $\Sigma_k$ est alors
contenu dans l'ensemble $\mathrm T$ des points $x$ de $X$ tels que le degré de transcendance 
de $(\widetilde{f_i(x)})_i$ sur $\widetilde k$ soit égal à $n$. 
Notons que $\mathrm T$
est fermé (c'est la réunion des images réciproques de $S_{n,k}$ sous \textit{tous} les morphismes
de $X$ vers $\gma n$  définis
par les différents $n$-uplets  qu'on peut extraire de $(f_1,\ldots, f_N)$). 
Nous allons montrer par récurrence sur $N$
que $\mathrm T$ est un $c$-squelette. Si $N<n$ alors $\mathrm T$ est vide, et si $N=n$
alors $\mathrm T$ est un $c$-squelette par le cas $n=1$ déjà traité. 

Supposons $N>n$ et le résultat vrai pour $N-1$. 
Soit $x\in \mathrm T$. 
Par définition, 
le degré de transcendance de $(\widetilde{f_i(x)})_i$ sur $\widetilde k$ est égal à $n$.
Posons $a_i=\widetilde{f_i(x)}$ pour tout $i$. 
Une application de l'avatar gradué du 
«lemme de normalisation de Noether torique» de Maclagan et Sturmfels (\cite{MaclaganSturmfels-2015}, 
preuve de la proposition 3.2.7, page 105)
assure alors l'existence d'une matrice $M=(m_{ij})$
de $\mathbf{GL}_N(\mathbf Z)$ telle que si l'on pose
$b_i=\prod_j a_j^{m_{ij}}$ pour tout $i$, la propriété suivante soit satisfaite : il existe $d>0$ et des monômes $c_{d-1},\ldots, c_0$
en $b_1,\ldots, b_{N-1}$, à exposants entiers relatifs, tels que
$b_N^d+c_{d-1}b_N^{d-1}+\ldots, +c_0=0$. 
Posons $g_i=\prod_j f_j^{m_{ij}}$ pour tout $i$. 
Par ce qui précède, il existe des monômes $\mu_{d-1},\ldots, \mu_0$ en 
$g_1,\ldots, g_{N-1}$, à exposants entiers relatifs, tels que
$\abs{g_N^d(x)+\mu_{d-1}(x)g_N^{d-1}(x)+\ldots+\mu_0(x)}<\max_i \abs{\mu_i(x)g_N^i(x)}$ 
(en posant $\mu_d=1$). Le lieu de validité de l'inégalité 
$\abs{g_N^d+\mu_{d-1}g_N^{d-1}+\ldots+\mu_0}<\max_i \abs{\mu_ig_N^i}$ est donc un voisinage 
ouvert $U$ de $x$. Par construction, $\widetilde{g_N(y)}$ est algébrique
sur $\widetilde k( {g_i(y)})_{1\leq i\leq N-1}$ pour tout $y\in U$. Il s'ensuit (la famille des $g_i$ 
se déduisant de celle des $f_i$ par une transformation monomiale inversible) que si $y\in U$
les assertions suivantes sont équivalentes : 

\begin{enumerate}[i]
\item $y\in \mathrm T$ ; 
\item le degré de transcendance de $(\widetilde{g_i(y)})_{1\leq i\leq N}$ sur $\widetilde k$ est égal à $n$ ; 
\item le degré de transcendance de $(\widetilde{g_i(y)})_{1\leq i\leq N-1}$ sur $\widetilde k$ est égal à $n$. 
\end{enumerate}

La propriété (iii) couplée à l'hypothèse de récurrence montre alors que $U\cap T$ est un $c$-squelette. 
Ceci valant quel que soit $x\in \mathrm T$, le fermé $\mathrm T$ de $X$ est bien un 
$c$-squelette. \qed

\section{L'espace friable associé à un espace domanial}\label{friable-general}
Nous nous proposons d'associer à un espace domanial 
un espace topologique 
possédant
une base d'ouverts compacts, que nous appellerons \textit{l'espace friable} associé.

\subsection{Brefs rappels sur les filtres}
Nous allons commencer par quelques rappels sur les filtres. On fixe un ensemble
$X$ et une sous-algèbre de Boole $\mathsf B$ de $\mathscr P(X)$.

\begin{defi}
Un \textit{filtre} d'éléments de $\mathsf B$ est un sous-ensemble $\mathscr F$ de $\mathsf B$ 
tel que : 
\begin{enumerate}[a]
\item $\mathscr F$ est stable par intersections finies ; 
\item $\emptyset\notin \mathscr F$ ; 
\item si $V$ appartient à $\mathscr F$ et si $W$ est un élément de $\mathsf B$ contenant $V$ alors
$W$ appartient à $\mathscr F$.
\end{enumerate}

Un \textit{ultrafiltre} d'éléments de $\mathsf B$ est un filtre d'éléments de $\mathsf B$ qui est maximal pour l'inclusion. 
\end{defi}

\begin{rema}
L'axiome (a) entraîne que tout filtre
d'éléments de $\mathsf B$ contient $X$ (considérer l'intersection vide). Si $X$ est lui-même vide, 
on obtient une contradiction avec l'axiome (b) : si $X$ est vide, il n'existe donc pas de filtres d'éléments de
$\mathsf B$. 
\end{rema}

\begin{rema}\label{fitre-v-pascv}
Si $\mathscr F$ est un filtre d'éléments de $\mathsf B$ et si $V\in \mathscr F$
alors $X\setminus V$ ne peut appartenir à $\mathscr F$, sans quoi 
$\emptyset=V\cap (X\setminus  V)$ appartiendrait à $\mathscr F$. 
\end{rema}

\subsubsection{}
Soit $\mathscr F$
un filtre d'éléments de $\mathsf B$. On déduit du lemme de Zorn que $\mathscr F$
est contenu dans un ultrafiltre d'éléments de $\mathsf B$.

Soit $V$ un élément de $\mathsf B$ n'appartenant pas à $\mathscr F$. Supposons
que $X\setminus V$ n'appartient pas non plus à $\mathscr F$. Pour tout $U\in \mathscr F$ l'intersection
$U\cap V$ est alors non vide : en effet dans le cas contraire $U$ serait contenu dans $X\setminus V$
qui appartiendrait donc à $\mathscr F$, contredisant nos hypothèses. L'ensemble des éléments de $\mathsf B$
contenant une partie de la forme $U\cap V$ avec $U\in \mathscr F$ est alors un filtre d'éléments de $\mathsf B$
contenant strictement $\mathscr F$. 
On en déduit que si $\mathscr F$ est un ultrafiltre alors $V\in \mathscr F$ ou 
$(X\setminus V)\in \mathscr F$ pour tout $V\in \mathsf B$.

Réciproquement, si $V\in \mathscr F$ ou 
$(X\setminus V)\in \mathscr F$ pour tout $V\in \mathsf B$ alors $\mathscr F$ est un ultrafiltre : en effet
soit $\mathscr G$ un filtre d'éléments de $\mathsf B$ contenant $\mathscr F$ et soit $V\in \mathscr G$ ; 
si $V$ n'appartenait pas à $\mathscr F$ alors $X\setminus V$ appartiendrait à $\mathscr F$
et partant à $\mathscr G$, ce qui est absurde puisque $V\in \mathscr G$ (remarque
\ref{fitre-v-pascv}).

En conséquence, $\mathscr F$ est un ultrafiltre si et seulement si $V\in \mathscr F$ ou $(X\setminus V)\in \mathscr F$
pour tout $V\in \mathsf B$. 

Supposons que ce soit le cas et soit $(V_i)_{i\in I}$ une famille finie d'éléments de $\mathsf B$ dont la réunion $V$ appartient
à $\mathscr F$. Il existe alors $i$ tel que $V_i\in \mathscr F$. En effet dans le cas contraire chacun des $X\setminus V_i$
appartiendrait à $\mathscr F$ puisque ce dernier est un ultrafiltre, si bien que leur intersection $X\setminus V$ appartiendrait
également à $\mathscr F$, ce qui est absurde, là encore en vertu de la remarque
\ref{fitre-v-pascv}. 

\begin{exem}
Si $x\in X$, l'ensemble des $V\in \mathsf B$ tels que $x\in V$ est un ultrafiltre d'éléments de $\mathsf B$ ; 
un tel ultrafiltre est dit \textit{principal}. 
\end{exem}

\subsubsection{}
Notons $\widetilde X$ l'ensemble des ultrafiltres d'éléments de $\mathsf B$. On  peut le voir
comme un sous-ensemble de $\{0,1\}^{\mathsf B}$, et on le munit de la topologie induite
par la topologie produit. Compte-tenu des axiomes qui caractérisent les filtres, et du fait qu'un 
filtre $\mathscr F$ d'éléments de $\mathsf B$ est un ultrafiltre si et seulement si
$V\in \mathscr F$ ou $(X \setminus V)\in \mathscr F$ pour tout $V\in \mathsf B$, on voit 
que $\widetilde X$ est fermé dans $\{0,1\}^{\mathsf B}$ ; c'est donc un espace topologique compact
totalement discontinu. Les parties de $\widetilde X$ de la forme $\{\mathscr F\in \widetilde X, V\in \mathscr F\}$ 
où $V$ parcourt $\mathsf B$ sont exactement les ouverts compacts de $\widetilde X$, et forment
une base de sa topologie. 

On dispose d'une application naturelle de $X$ dans $\widetilde X$, qui envoie un point $x$ sur
l'ultrafiltre principal correspondant. Elle est injective si et seulement si $\mathsf B$ sépare les points de $X$. 

\subsubsection{}\label{vtilde-xtilde}
Soit $V$ un élément de $\mathsf B$. L'intersection $\mathscr P(V)\cap \mathsf B$ est une sous-algèbre de 
Boole de $\mathscr P(V)$. Si $\mathscr F$ est un ultrafiltre d'éléments de $\mathsf B$ contenant $V$, l'ensemble
des éléments de $\mathscr P(V)\cap \mathsf B$ appartenant à $\mathscr F$ est un ultrafiltre d'éléments de 
 $\mathscr P(V)\cap \mathsf B$. Si $\mathscr G$ est un ultrafiltre d'éléments de 
 $\mathscr P(V)\cap \mathsf B$, l'ensemble des éléments de $\mathsf B$ contenant un élément de $\mathscr G$ est un ultrafiltre 
 d'éléments de $\mathsf B$. 
 
 Ces constructions établissent un homéomorphisme entre l'ouvert compact de $\widetilde X$ constitué des
 ultrafiltres contenant $V$, et l'espace $\widetilde V$ des ultrafiltres d'élément de $\mathscr P(V)\cap \mathsf B$.  
 
 \subsubsection{}\label{conventions-filtres}
 Le plus souvent, nous préfèrerons penser à 
 un élément de $\widetilde X$ comme à un point auquel \textit{correspond} un ultrafiltre, plutôt que comme
 à un ensembles de parties de $X$, et nous emploierons donc plus volontiers la notation $x$ (ou $y$, ou $z$\ldots)
 que $\mathscr F$ (ou $\mathscr G$, ou $\mathscr H$) pour désigner un élément de $\widetilde X$. 
 
 Si $V$ est un élément de $\mathsf B$, nous identifierons $\widetilde V$ à un ouvert compact de 
 $\widetilde X$ \textit{via} l'homéomorphisme décrit en \ref{vtilde-xtilde} ; la flèche
 $V\mapsto \widetilde V$ établit une bijection de $\mathsf B$
sur l'ensemble des ouverts compacts de $\widetilde X$. 

Avec ces conventions, l'ultrafiltre
 correspondant à un point $x$ de $\widetilde X$ est l'ensemble des éléments $V$ 
 de $\mathsf B$ tels que $x\in \widetilde V$. 
 
 Si $\mathsf B$ sépare les points de $X$ nous utiliserons l'injection naturelle $X\hookrightarrow \widetilde X$
 pour identifier $X$ à un sous-ensemble de $\widetilde X$. On a alors $V=\widetilde V\cap X$ pour tout
 $V\in \mathsf B$ ; en particulier, $X$ est dense dans $\widetilde X$. 
 
 \subsubsection{}
 Soit $Y$ un ensemble et soit $\mathsf C$ une sous-algèbre de Boole de $\mathscr P(Y)$. Soit
 $\widetilde Y$ l'espace des ultrafiltres d'éléments de $\mathsf C$, et soit $f\colon Y\to X$ une application telle
 que $f^{-1}(V)\in \mathsf C$ pour tout $V\in \mathsf B$. L'application $f$ induit une application continue 
$\widetilde f$
de $\widetilde Y$ vers $\widetilde X$ : en termes d'ultrafiltres, elle envoie un ultrafiltre $\mathscr G$ d'éléments
 de $\mathsf C$ sur l'ensemble des $V\in \mathsf B$ tels que $f^{-1}(V)\in \mathscr G$ ; en termes
 plus topologiques, elle envoie un point $y$ de $\widetilde Y$ sur l'unique élément de l'intersection 
 des ouverts compacts $\widetilde V$, où $V$ parcourt l'ensemble des éléments de $\mathsf B$ tels que
 $y\in \widetilde{f^{-1}(V)}$. Autrement dit, $\widetilde f$
 est caractérisée par le fait que 
 $\widetilde f^{-1}(\widetilde V)=\widetilde{f^{-1}(V)}$
 pour tout $V\in \mathsf B$. 
 
 Il découle de la construction de $\widetilde f$ que le diagramme 
 \[
 \begin{tikzcd}
 Y\ar[r,"f"]\ar[d]
 &X\ar[d]\\
 \widetilde Y\ar[r, "\widetilde f"]&\widetilde X
 \end{tikzcd}
 \]
 est commutatif. Si $\mathsf B$ sépare les points de $X$ et si $\mathsf C$ sépare les points
 de $Y$ l'application $\widetilde f$ apparaît donc, modulo les plongements
 naturels de $X$ dans $\widetilde X$ et de $Y$ dans $\widetilde Y$, 
 comme un prolongement de $f$.

\subsection{Le cas d'un espace domanial compact}\label{friable-compact}
On fixe un espace domanial \textit{compact}
$X$.

\subsubsection{}
\label{def-friable}
Nous noterons $\domc X$ l'ensemble des domaines
compacts de $X$, et $\dom X$ la sous-algèbre de Boole de
$\mathscr P(X)$ engendrée
par $\domc X$ ; nous désignerons par 
$\widetilde X$ l'ensemble des ultrafiltres d'éléments 
de $\dom X$. C'est un espace topologique compact et totalement discontinu,
que nous appellerons l'\textit{espace friable} associé à $X$. 

Si $V\in \dom X$ nous noterons $\widetilde V$ l'ensemble des ultrafiltres d'éléments 
de $\dom X$ contenu dans $V$, que nous identifierons à un ouvert compact
de $\widetilde X$ (\ref{conventions-filtres}). 

Les éléments de $\dom X$ séparent les points de $X$ (c'est déjà le cas des 
éléments de $\domc X$), si bien que l'on dispose d'une injection
naturelle $X\hookrightarrow \widetilde X$
d'image dense. Elle n'est pas continue en général : 
on a en effet $\widetilde V\cap X=V$ pour tout $V\in \dom X$, et un tel $V$
peut très bien n'être pas ouvert.


\subsubsection{}
Soit $x\in \widetilde X$ et soit $\mathsf E$ l'ensemble des éléments
$V$
de $\dom X$ tels que $x\in \widetilde V$. L'intersection d'une famille finie d'éléments de
$\mathsf E$ appartient à $\mathsf E$ et est en particulier non vide. Il en résulte 
par compacité que $U:=\bigcap_{V\in \mathsf E} \overline V$ est non vide. 
Le compact $U$ est en fait un singleton. En effet, soient $y$ et $z$ deux points distincts de $X$. 
Soit $V$ un voisinage domanial compact de $y$ dans $X\setminus \{z\}$
 et soit $W$ le complémentaire
de $V$. Puisque $X=V\coprod W$ on a 
$\widetilde X=\widetilde V\coprod \widetilde W$ ; en particulier $x$
appartient à $\widetilde V$ ou à $\widetilde W$, si bien que $V\in \mathsf E$ ou $W\in \mathsf E$ ; 
dans le premier cas $z\notin U$ puisque $z\notin \overline V$, et dans le second  cas $y\notin U$
puisque $y\notin \overline W$ ; ainsi $y$ et $z$ ne peuvent appartenir tous deux à $U$,
qui est dès lors un singleton, comme annoncé. 
Son unique point sera noté $\rho_X(x)$.

\subsubsection{}
Par construction, $\rho_X$ est une rétraction de l'inclusion 
$X\hookrightarrow \widetilde X$. Si $V$
appartient à $\domc X$ 
le diagramme 
\[\begin{tikzcd}
\widetilde V\ar[r,hook]\ar[d,"\rho_V"']&\widetilde X\ar[d,"\rho_X"]\\
V\ar[r,hook]&X
\end{tikzcd}\]
est commutatif ; on pourra donc se permettre d'écrire $\rho$ au lieu de $\rho_X$, 
ce que nous ferons le plus souvent. 

Soit $x$ un point de $\widetilde X$ et soit $V$
un voisinage domanial compact de $\rho(x)$ ; soit $U$ le complémentaire de $V$
dans $X$. Par construction, $\rho(x)$ n'appartient pas
à $\overline U$, ce qui entraîne que $x$ ne peut appartenir à $\widetilde U$ ; par conséquent, 
$x$ appartient à $\widetilde V$. Ainsi, $\widetilde V$ est un voisinage ouvert de $x$, 
dont l'image par $\rho$ est contenue dans $V$ (et même égale à $V$ tout entier). 
Il en résulte que $\rho$ est continue. 

\subsubsection{}
Soit $f\colon Y\to X$ un morphisme d'espaces domaniaux compacts. 
Si $V$ appartient à $\domc X$ alors $f^{-1}(V)$ appartient
à $\domc Y$. Il s'ensuit que $f^{-1}(V)$ appartient à 
$\dom Y$ pour tout $V\in \dom X$, et $f$ se prolonge
par conséquent en 
une application continue $\widetilde f$ de $\widetilde Y$ vers
$\widetilde X$. 
Il résulte des constructions que le diagramme
\[
\begin{tikzcd}
\widetilde Y\ar[r,"\widetilde f"]\ar[d,"\rho"']&\widetilde X \ar[d,"\rho"]\\
Y\ar[r,"f"]&X\end{tikzcd}
\]
commute.

\subsubsection{Cas d'une partie localement fermée}
\label{breve-comp-comp}
Soit $E$ une partie compacte
de $X$. 
Elle hérite d'une structure domaniale naturelle, 
pour laquelle l'inclusion de $E$ dans $X$ est un morphisme
(\ref{locferm-domanial}). On dispose par ce qui précède
d'une application continue naturelle de 
$\widetilde E$ vers $\widetilde X$. Cette application induit un homéomorphisme
entre $\widetilde E$ et une partie compacte de $\widetilde X$. 

En effet, nous allons montrer que $\widetilde E$ s'injecte 
dans $\widetilde X$, ce qui permettra de conclure par compacité. 
Soient donc $x$ et $y$ deux points distincts
de $\widetilde  E$. Quitte à échanger $x$ et $y$, 
on peut supposer qu'il existe $W\in \domc E$ tel que
$x\in \widetilde W$ et $y\notin \widetilde W$. Par définition de la
structure d'espace domanial sur $E$ et par compacité
de $W$, il existe une famille finie 
de domaines compacts $(V_i)$ de $X$ tels 
que $W=\bigcup W_i$ où l'on a posé $W_i=V_i\cap E$. 
Il existe alors $i$ tel que $x\in \widetilde W_i$, 
et $y$ n'appartient pas à $\widetilde W_i\subset \widetilde W$. 
Par  la définition même de l'application de $\widetilde E$
vers $\widetilde  X$ l'ouvert compact
$\widetilde V_i$
de $\widetilde X$ contient l'image de $x$
mais pas celle de $y$, ce qui montre que ces
deux images sont distinctes. 

Nous nous permettrons donc d'identifier $\widetilde E$ 
à une partie compacte de $\widetilde X$. Modulo cet abus on a
$\rho(x)\in E$ pour tout $x\in \widetilde E$ et $\widetilde E\cap X=E$.

\subsection{Espace friable associé à un espace domanial général}
Soit
$X$ un espace domanial ; on ne le suppose plus compact (ni même séparé). 
Nous désignons encore par $\domc X$ l'ensemble des domaines compacts de $X$. 
L'ensemble $\domc X$ est naturellement ordonné par inclusion, et cofiltrant (l'intersection de deux 
éléments de $\domc X$ n'appartient pas forcément à $\domc X$, l'espace $X$ n'étant pas
supposé séparé ; mais elle est G-recouverte par des éléments de $\domc X$).

Si $V$ est un élément de $\domc X$ nous noterons $\widetilde V$ l'espace friable associé, 
et $\rho \colon \widetilde V\to V$ la rétraction continue naturelle de $V\hookrightarrow \widetilde V$
(voir la sous-section \ref{friable-compact} ci-dessus
pour ces différentes notions).

\subsubsection{}
Nous noterons $\breve X$ la colimite
topologique des $\widetilde V$ pour $V$ parcourant $\domc X$. 
C'est un espace topologique qui s'identifie à ce que nous avons déjà noté 
$\widetilde X$ lorsque $X$ est compact, et que nous appellerons encore l'espace friable associé à $X$. 
Pour tout domaine
$V$ de $X$, la flèche naturelle $\breve V\to \breve X$ est une immersion ouverte, et les
$\breve V$ pour $V\in \domc X$ forment une base d'ouverts compacts de $\breve X$.
Si $X$ est séparé $\breve X$ est séparé, et les $\breve V$ en constituent donc une base d'ouverts fermés. 

Notre choix de la notation $\breve X$ plutôt que $\widetilde X$, malgré la coïncidence des deux objets dans le cas compact, 
s'explique ainsi. Si $V$ est un domaine compact de $X$, si $W$ est un domaine compact de $V$ et si l'on pose
$U=V\setminus W$, alors $\widetilde U$ a été défini à la sous-section \ref{friable-compact} : c'est l'espace
des ultrafiltres de parties de $\dom V$ contenues dans $U$, qui s'identifie à un ouvert compact
de 
$\widetilde V$. Quant à $\breve U$, c'est la réunion des $\widetilde Z$ où $Z$ parcourt l'ensemble
des domaines compacts de $U$ ; c'est donc un ouvert de $\widetilde  U$, qui est strict dès que $U$ n'est pas compact.

\subsubsection{}
Lorsque $V$ parcourt $\domc X$, les inclusions naturelles 
$V\hookrightarrow \breve V$ et leurs rétractions continues $\rho \colon \breve V\to V$
se recollent en une injection $X\hookrightarrow \breve X$ et une 
rétraction continue $\rho \colon \breve X\to X$ de cette dernière. 

\subsubsection{Fonctorialité.} Soit $f\colon Y\to X$ un morphisme d'espaces domaniaux. 
Soit $W$ un domaine compact de $Y$ et soit $V$ un domaine compact de $X$ tel
que $f(W)\subset V$. L'application $f$ induit une application $\breve f$ de $\breve W$ vers
$\breve V$. 
Lorsque $V$ et $W$ varient les applications continues ainsi définies se recollent en une application continue
$\breve f$ de $\breve Y$ vers $\breve X$. On vérifie par réduction au cas compact que le diagramme 
\[\begin{tikzcd}
Y\ar[r,"f"]\ar[d,hook]&X\ar[d,hook]\\
\breve Y\ar[r,"\breve f"]\ar[d,"\rho"']&\breve X\ar[d,"\rho"]\\
Y\ar[r,"f"']&X\end{tikzcd}
\] est commutatif.

\subsubsection{Cas d'une partie localement fermée}
\label{breve-locferm}
Soit $E$ une partie localement fermée de $X$. 
Elle hérite d'une structure domaniale naturelle, 
pour laquelle l'inclusion de $E$ dans $X$ est un morphisme
On dispose dès lors d'une application 
naturelle de 
$\breve E$ vers $\breve X$. Cette application induit un homéomorphisme
entre $\breve E$ et une partie localement
fermée de $\breve X$, fermée si $E$ est fermé dans $X$. 

En effet, choisissons un ouvert $U$
de $X$ dans lequel $E$ est fermé. 
On peut  pour montrer l'assertion remplacer
$X$ par $U$, et partant supposer $E$ fermé
dans $X$. Il suffit alors de prouver que
$\breve E\times_{\breve X}\breve V\to \breve V$
induit un homéomorphisme entre 
$\breve E\times_{\breve X}\breve V$
et une partie fermée
de $\breve V$ pour tout domaine compact
$V$ de $X$. Mais pour un tel domaine, 
$\breve E\times_{\breve X}\breve V$ est par
construction l'espace friable
associé à
$E\cap V$ ; on se ramène ainsi au cas
où $X$ est compact, et où $E$ est une partie compacte de
$X$, qui a été traité en \ref{breve-comp-comp}. 

\subsubsection{Raffinement de la structure domaniale}\label{dom-x1-x2}
Supposons données deux structures domaniales
ayant même espace topologique sous-jacent $X$, que nous 
notons $X_1$ et $X_2$. On suppose que $\dom {X_2}\subset \dom{X_1}$, 
ce qui signifie précisément que l'identité de $X$ induit un morphisme
d'espaces domaniaux $X_1\to X_2$. 
On dispose alors en particulier d'une application continue 
$\breve X_1\to \breve X_2$, qui est propre : si $V$ est un domaine compact
de $X_1$, dont on note $V_1$ et $V_2$ les structures domaniales induites par celles de
$X_1$ et $X_2$, l'image réciproque du compact
$\breve V_1$ sur $\breve X_2$ s'identifie au compact $\breve V_2$. 
Le diagramme \[\begin{tikzcd}
\breve V_1\ar[r]&\breve V_2\\
V\ar[u,hook]\ar[ur,hook]&\end{tikzcd}\]
étant commutatif, l'image de la flèche $\breve V_1\to \breve V_2$ contient le
sous-ensemble dense $V$ de $\breve V_2$. Cette image étant compacte, c'est $\breve V_2$
tout entier. Il s'ensuit que $\breve X_1\to \breve X_2$ est surjective. 

\begin{rema}
On peut démontrer (nous ne nous servirons pas de ce fait) que 
$\breve X$ est en bijection avec l'ensemble $\mathsf P$
des points du topos
associé au site domanial de $X$. 
Cette bijection n'est pas un homéomorphisme : la
topologie sur $\breve X$ héritée de son identification avec
$\mathsf P$ est la topologie engendrée par les $\breve V$
où $V$ est un domaine compact de $X$, qui est 
moins fine que celle que nous utilisons ici : 
si $W$ et $V$ sont deux domaines compacts de $X$ avec $W\subset V$ alors $\breve V\setminus \breve W$ sera 
toujours ouvert pour la topologie friable considérée ici, mais pas en général pour la topologie induite par
celle de $\mathsf P$.

\end{rema}

\section{Espaces friables en géométrie analytique}\label{friable-analytique}

Nous nous proposons maintenant d'étudier en détail l'espace $\breve X$ lorsque
$X$ est un espace $k$-analytique $\Gamma$-strict, et de montrer qu'il s'identifie
à l'espace de Huber-Kedlaya (muni de sa topologie constructible).
C'est en fait par la construction de ce dernier que nous allons commencer, 
dans le cas affinoïde.

\subsection{Espace $\Gamma$-adique associé à un espace
$k$-affinoïde $\Gamma$-strict}
Soit $X$ un espace $k$-affinoïde
$\Gamma$-strict ; on le munit de sa structure 
domaniale $\Gamma$-strictement $k$-analytique. 
Le but de ce qui suit
est de définir l'espace $\Gamma$-adique $X\gad$
associé à $X$ en termes purement valuatifs, puis
de montrer qu'il s'identifie à l'espace friable
$\breve X$. 
Nous  allons utiliser librement les définitions et notations de
\ref{subsection-valuations-bridees}.

\subsubsection{Description  ensembliste 
de $X\gad$}
Notons $A$
l'algèbre des fonctions analytiques sur $X$. 
Comme $X$ est $\Gamma$-strict, la semi-norme spectrale $\|\cdot\|_\infty$
de $A$ est 
à valeurs dans
le monoïde
$(\abs{k^\times}\cdot \Gamma)^\Q\cup\{0\}$, et nous 
dirons qu'une $k$-valuation $\abs{k^\times}\cdot \Gamma$-colorée $x$ sur $A$ est
\textit{bornée} si $\abs{a(x)}\leq \|a\|_\infty$
pour tout $a\in A$. Pour que $x$ soit bornée, il suffit en fait 
que $\abs{a(x)}\leq \|a\|_\infty$ pour tout $a\in A$ telle que
$ \|a\|_\infty\in \Gamma$. Supposons en effet que ce soit le cas, et
soit $b\in A$. Si $\|b\|_\infty=0$ alors $b$ est nilpotente si bien
que $b(x)=0$, et $\abs{b(x)}$ est bien majoré par $\|b\|_\infty=0$. Sinon il existe 
un entier $n\geq 0$, un élément $\lambda$ de $k^\times$ et un élément
$r$ de $\Gamma$ tels que $\|b\|_\infty^n=r\abs \lambda$. Si l'on pose
$a=b^n/\lambda$ on a donc $\|a\|_\infty\in \Gamma$. Par hypothèse ceci
entraîne que $\abs{a(x)}\leq \|a|\|_\infty$, d'où il suit
immédiatement que $\abs{b(x)}\leq \|b|\|_\infty$. 

Nous définissons $X\gad$ comme l'ensemble des classes d'équivalence de
$k$-valuations
$\abs{k^\times}\cdot \Gamma$-colorées, bridées et bornées sur $A$.

\subsubsection{Liens avec la réduction de Temkin.}
\label{lien-gad-temkin}
Tout point de $X$ est en particulier 
une $k$-valuation
$\abs{k^\times}\cdot \Gamma$-colorée bornée et bridée
sur $A$, et peut donc être vu comme appartenant à $X\gad$ ; 
nous considèrerons désormais par ce biais
$X$ comme un sous-ensemble
de $X\gad$. Par ailleurs 
si $x$ est point de $X\gad$, il lui correspond un couple
$(x^\#,x^\flat)$ où 
$x^\#$ est une $k$-valuation $\Gamma$-colorée réelle sur $A$
et $x^\flat$ une $\widetilde k^\Gamma$-valuation sur $\hrt x^\Gamma$, et comme le passage de 
$x$ à $x^\#$ préserve les inégalités larges, $x^\#$ est bornée : 
c'est donc un point de $X$. L'application $x\mapsto x^\#$ est une rétraction de
l'inclusion $X$ dans $X\gad$, puisqu'un point $\xi$ de $X\subset X\gad$ 
correspond au couple $(\xi, \theta)$
où $\theta$ est la valuation triviale sur $\hrt \xi^\Gamma$. 

Soit $\xi \in X$, soit $\omega$ une valuation sur $\hrt \xi^\Gamma$ et soit 
$x$ la $k$-valuation
$\abs{k^\times}\cdot \Gamma$-colorée bridée sur $A$ correspondant
au couple $(\xi,\omega)$. La valuation $x$ appartient à $X\gad$ si et seulement 
si $\abs{a(x)}\leq \|a\|_{\infty}$ pour tout élément $a\in A$ dont la semi-norme
spectrale appartient à $\Gamma$. Mais cela revient à demander
que $\abs{\alpha(\omega)}\leq 1$ pour tout $\alpha$ appartenant à l'image
de $\widetilde A$ dans $\hrt \xi^\Gamma$, c'est-à-dire encore
que $\omega\in\widetilde{(X,\xi)}^\Gamma$. La fibre en $\xi$ de l'application 
$X\gad\to X$
est donc ensemblistement la réduction de Temkin $\Gamma$-graduée
du germe $(X,\xi)$.

\subsubsection{Fonctorialité.}
Soit $Y=\mathscr M(B)$ un espace $k$-affinoïde $\Gamma$-strict
et soit $f\colon Y\to X$ un morphisme. Soit
$y$ un point de $Y\gad$. 
L'application $f^*\colon A\to B$
contracte les semi-normes spectrales. Il s'ensuit que
la $k$-valuation $\abs{k^\times}\cdot \Gamma$-colorée et bridée
$a\mapsto \abs{f^*(a)(y)}$ sur $A$
est bornée, et 
appartient
dès lors à $X\gad$. Ainsi $f$ induit-il une application 
$Y\gad\to X\gad$, qu'on notera $f\gad$ ou, le plus souvent,
simplement $f$ s'il n'y a pas de risque de confusion. 

\subsubsection{Domaines affinoïdes.}\label{domaines-aff}
Soit $\domaff X$ l'ensemble
des domaines affinoïdes
$\Gamma$-stricts de $X$ ; c'est une partie
de l'ensemble $\domc X$ de tous les domaines
analytiques compacts et $\Gamma$-stricts de $X$. 
En vertu de \ref{lien-gad-temkin},
pour tout
$V\in \domaff X$
la flèche naturelle de
$V\gad$
vers $X\gad$ induite par le morphisme
$V\hookrightarrow X$ 
identifie $V\gad$ au sous-ensemble
de $X\gad$ formé des points
$x$ tels que $x^\#\in V$ et $x^\flat\in\widetilde{(V,x^\#)}^\Gamma$.
Il en résulte que la
flèche $\domaff X\to \mathscr P(X\gad), 
V\mapsto V\gad$ commute aux unions et
intersections finies, puis que cette flèche s'étend
d'une unique manière en une flèche
$V\mapsto V\gad$ de $\domc X$ vers $\mathscr P(X\gad)$
commutant aux unions et intersections finies. 
Pour tout $V\in \domc X$, le sous-ensemble
$V\gad$ de $X\gad$ est formé des 
des points
$x$ tels que $x^\#\in V$ et $x^\flat\in\widetilde{(V,x^\#)}^\Gamma$. 

Soit
$x\in X\gad$ et soit $V\in \domaff X$
tel que $x\in V\gad$.
Chacune des notations $x^\#, x^\flat, \hr x$ et $\hr {x^\#}$
désigne un objet qui ne dépend pas du fait que $x$
soit vu comme point de $V\gad$ ou de $X\gad$, et nous
pourrons donc les employer sans aucun risque de confusion
(notons par contre que la notation $\kappa(x)$
serait ambiguë ici, le morphisme 
$A\to \mathscr O_X(V)$ n'ayant aucune
raison de préserver les corps résiduels \textit{schématiques}, 
et nous ne l'utiliserons pas). 

\subsubsection{Un exemple explicite.}
\label{ex-gamma-rat}
Soit $V$ un domaine affinoïde
$\Gamma$-rationnel de $X$. 
Choisissons-en une description 
par une conjonction d'inégalités 
\[\abs{f_1}\leq \lambda_1 \abs g\;\text{et}\ldots\;\text{et}\;\abs{f_n}\leq \lambda_n \abs g\] où $g$ et les $f_i$ sont des
fonctions analytiques sur $X$ sans zéro commun et où les $\lambda_i$ sont des éléments de $\Gamma$. 

Soit $x$ un point de $X\gad$. Il appartient à $V\gad$ si et seulement si $x^\# \in V$
et $x^\flat\in\widetilde{(V,x^\#)}^\Gamma$. La première condition revient à demander que 
$\abs{f_i(x^\#)}\leq \lambda_i \abs{g(x^\#)}$ pour tout $i$ (ce qui entraîne l'inversibilité de $g(x^\#)$, 
les $f_i$ et $g$ étant sans zéro commun), et la seconde revient à demander que
$\abs{(f_i/g)(x^\flat)}
\leq 1$ pour tout $i$ tel que $\abs{f_i(x^\#)}=\lambda_i\abs {g(x^\#)}$, c'est-à-dire encore
que $\abs{(f_i/g)(x)}
\leq \gamma$ pour tout $i$ tel que $\abs{f_i(x^\#)}=\lambda_i\abs {g(x^\#)}$. 
Il en résulte que $x\in V\gad$ si et seulement si $\abs{f_i(x)}\leq \lambda_i \abs{g(x)}$ pour tout $i$ ;
autrement dit, on peut décrire $V\gad$ comme sous-ensemble de $X\gad$ par la même conjonction d'inégalités 
que $V$ comme sous-ensemble de $X$.

\subsection{Identification de $X\gad$ à $\widetilde X$}
On désigne toujours par $X$
un espace $k$-affinoïde $\Gamma$-strict.

\subsubsection{}\label{ugad-barre}
Rappelons 
que $\dom X$ désigne la sous-algèbre de Boole de $\mathscr P(X)$ engendrée par $\domc X$. 
L'application $V\mapsto V\gad$ de $\domc X$ dans $\mathscr P(X\gad)$
commutant aux unions et intersections finies, elle s'étend de manière unique en un
morphisme d'algèbres de Boole de
$\dom X$ vers $\mathscr P(X\gad)$ que nous allons noter
$V\mapsto \gadc V$ (cette notation est pour le moment à considérer 
comme d'un seul tenant ; elle est justifiée par la remarque \ref{rem-ugad-barre} plus bas). 
Si $V\in \domc X$ et si
$U$ désigne son complémentaire alors $\gadc V=V\gad$ 
et $\gadc U=X\gad \setminus V\gad$.

\begin{prop}\label{xgad-filtres}
La formule
\[x\mapsto \{V\in \dom X, x\in \gadc V\}\]
établit une bijection de $X\gad$
sur l'ensemble $\widetilde X$ des ultrafiltres d'éléments de $\dom X$. 
\end{prop}

\begin{proof}
Pour tout $x\in X\gad$,
nous noterons
$\Phi(x)$ l'ensemble
des éléments $V$ de $\dom X$
tels que $x\in \gadc V$. Il est immédiat
que $\Phi(x)$ est pour tout $x\in X\gad$ un ultrafiltre d'éléments de
$\dom X$. On a donc bien défini une application $\Phi\colon  X\gad
\to \widetilde X$, et il reste à montrer qu'elle est bijective. 

Soit $\mathscr F$ 
un ultrafiltre d'éléments de $\dom X$.
Nous allons construire un antécédent de $\mathscr F$ par $\Phi$,
puis montrer qu'il est unique.
Dans ce qui suit, le domaine 
analytique $\Gamma$-strict
de $X$ décrit par une 
certaine
combinaison booléenne entre normes
de fonctions analytiques 
sera simplement désigné par la combinaison en question
placée entre accolades. 

Soit 
$A$ l'anneau
des fonctions analytiques sur $X$. Notons $\mathscr E$ 
l'ensemble des éléments $f$ de $A$ 
vérifiant la propriété suivante :
\textit{il existe $r\in \abs{k^\times}
\cdot \Gamma$ tel que 
$\{\abs f\geq r\}$ appartienne à $\mathscr F$} ; nous dirons
alors d'un tel $r$ qu'il est \textit{adapté} à $f$.
Remarquons que
si $r$ est adapté à $f$ il en va de même de tout élément
$s\leq r$ de
$\abs {k^\times}\cdot \Gamma$ 
parce que $\{\abs f\geq s\}\supset \{\abs f\geq r\}$ et que $\mathscr F$ est un filtre. Remarquons aussi que
$\mathscr E$ est stable par produit : si $r$
est adapté à $f$ et $s$ à $g$
alors $rs$ est adapté à $fg$
puisque
\[\{\abs{fg}\geq rs\}\supset \{\abs f \geq r\}
\cap \{\abs g \geq s\}.\]
On définit sur $\mathscr E^2\times \Gamma$
une relation $\mathscr R$ par la formule
suivante : $(f,g,\lambda)\mathscr R (u,v,\mu)$
si et seulement s'il existe un élément 
$r\in \abs {k^\times}\cdot \Gamma$ adapté à la fois à $f,g,
u$ et $v$ tel que
\[\{\abs f\geq r\;\text{et}\;\abs g\geq r\;
\text{et}\;\abs u\geq r\;\text{et}\; \abs v\geq r
\;\text{et} \;\abs{f/g}\lambda=\abs{u/v}\mu\}\]
appartienne à $\mathscr F$. Notons que si c'est le cas,
alors 
\[\{\abs f\geq s\;\text{et}\;\abs g\geq s\;
\text{et}\;\abs u\geq s\;\text{et}\; \abs v\geq s
\;\text{et} \;\abs{f/g}\lambda=\abs{u/v}\mu\}\]
appartient encore à $\mathscr F$ pour tout
autre $s\in \abs {k^\times}\cdot \Gamma$ adapté à la fois à $f,g,
u$ et $v$ : cela provient du fait que 
\[\{\abs f\geq s\;\text{et}\;\abs g\geq s\;
\text{et}\;\abs u\geq s\;\text{et}\; \abs v\geq s
\;\text{et} \;\abs{f/g}\lambda=\abs{u/v}\mu\}\]
contient 
\[\{\abs f\geq r\;\text{et}\;\abs g\geq r\;
\text{et}\;\abs u\geq r\;\text{et}\; \abs v\geq r
\;\text{et} \;\abs{f/g}\lambda=\abs{u/v}\mu\}\]
si $s\leq r$, et que 
\[\{\abs f\geq s\;\text{et}\;\abs g\geq s\;
\text{et}\;\abs u\geq s\;\text{et}\; \abs v\geq s
\;\text{et} \;\abs{f/g}\lambda=\abs{u/v}\mu\}\]
est l'intersection de
\[\{\abs f\geq r\;\text{et}\;\abs g\geq r\;
\text{et}\;\abs u\geq r\;\text{et}\; \abs v\geq r
\;\text{et} \;\abs{f/g}\lambda=\abs{u/v}\mu\}\]
et de l'élément
\[\{\abs f\geq s\;\text{et}\;\abs g\geq s\;
\text{et}\;\abs u\geq s\;\text{et}\; \abs v\geq s\}\]
de $\mathscr F$ sinon. 

Il s'ensuit immédiatement que $\mathscr R$ est une relation
d'équivalence, et que la formule 
\[((f,g,\lambda),(u,v,\mu))\mapsto
(fu,gv,\lambda \mu)\] passe au quotient et fait de
$G:=\mathscr E^2\times \Gamma/\mathscr R$
un groupe abélien, dont le neutre
est $\overline{(1,1,1)}$, et dans lequel
l'inverse de $\overline{(f,g,\lambda)}$ est
$\overline{(g,f,\lambda^{-1})}$. 
Le groupe $\abs{k^\times}\cdot\Gamma$ se plonge naturellement
dans $G$ par la formule
$(r,\lambda)\mapsto \overline{(a,1,\lambda)}$ 
où $a$ est n'importe quel élément de $k^\times$ tel que
$\abs a=r$ (on vérifie aussitôt que le terme de droite ne dépend pas du choix de $a$). 
Soit $\leq$ la relation sur $\mathscr E^2\times
\Gamma$ définie par la condition que 
$(f,g,\lambda)\leq (u,v,\mu)$ s'il existe
$r\in \abs{k^\times}\cdot \Gamma$ 
adapté à $f,g, u$ et $v$ et tel que
\[\{\abs f\geq r\;\text{et}\;\abs g\geq r\;
\text{et}\;\abs u\geq r\;\text{et}\; \abs v\geq r
\;\text{et} \;\abs{f/g}\lambda\leq \abs{u/v}\mu\}\]
appartienne à $\mathscr F$ (par un raisonnement analogue à celui tenu ci-dessus, une fois
que c'est vrai pour un certain $r$, cela reste vrai pour tout
$s$ adapté à $f,g,u$ et $v$). 
Cette relation passe au quotient
par $\mathscr R$, et on vérifie qu'elle fait de $G$
un groupe abélien ordonné ; 
le plongement
$\abs{k^\times}\cdot\Gamma\hookrightarrow G$ est croissant et fait
donc de $G$ un groupe $\abs{k^\times}\cdot \Gamma$-coloré, qui est bridé par 
construction : si $f$ et $g$ appartiennent à $\mathscr E$ alors
$f$ n'est pas nilpotent et $\overline{(f,g,\lambda)}$ est majoré par
$||f||_\infty\lambda/r$ où $r$ est n'importe quel élément
de $\abs{k^\times}\cdot \Gamma$ adapté à $g$. 

L'application de $A$ dans $G\cup\{0\}$ qui envoie $f$ sur $0$ si $f\notin \mathscr E$ et sur
$\overline{(f,1,1)}$ sinon est alors une $k$-valuation $\abs{k^\times}\cdot \Gamma$-colorée bridée
et bornée sur $A$, et elle définit donc un point $x$ de $X\gad$. Par construction, si $V$ est un domaine
affinoïde $\Gamma$-rationnel de $X$ alors $x\in V\gad$ si et seulement si $V\in \mathscr F$, et $x$ est le
seul point de $X\gad$ à posséder cette propriété. Comme $\mathscr F$ est
un ultrafiltre et comme les éléments de $\domc X$
sont exactement les unions finies de domaines affinoïdes $\Gamma$-rationnels, il s'ensuit que pour tout 
$V\in \domc X$ le point $x$ appartient à $V\gad$ si et seulement si $V\in \mathscr F$, et que
$x$ est le seul
point de $X\gad$ à posséder cette propriété ; il s'ensuit pour des raisons formelles que
pour tout 
$V\in \dom X$ le point $x$ appartient à $\gadc V$ si et seulement si $V\in \mathscr F$, et que
$x$ est le seul
point de $X\gad$ à posséder cette propriété ; autrement dit, $\Phi^{-1}(\mathscr F)=\{x\}$. 

\end{proof}

\subsubsection{Description de $X\gad$ comme espace topologique}
La bijection construite ci-dessus entre $X\gad$ 
et $\widetilde X$ permet de munir $X\gad$ d'une topologie en faisant un 
espace compact totalement discontinu, 
dont les ouverts fermés sont 
les éléments de l'algèbre de Boole engendrée par les $V\gad$ pour $V\in \domc X$.

Si $\Gamma=\{1\}$ (resp. $\Gamma=\mathbb R_{>0}$) l'espace topologique $X\gad$ s'identifie à l'espace adique
(resp. l'espace de valuations réifiées) associées à $X$ par Huber \cite{huber1993} (resp. Kedlaya, \cite{kedlaya2015})
\textit{muni de sa topologie constructible}. (Pour alléger l'écriture nous écrirons
$X\ad$ au lieu de $X^{\{1\}\text{-}\mathrm{ad}}$). 

\subsection{Définition de $X\gad$ dans le cas général}
On désigne maintenant par $X$ un espace $k$-analytique $\Gamma$-strict
quelconque.

\subsubsection{}
On définit l'espace
topologique $X\gad$ comme la colimite
topologique de $V\gad$
pour $V$ parcourant l'ensemble des domaines affinoïdes 
$\Gamma$-stricts de $X$. 

Remarquons qu'il n'y a pas de conflits de notation : 
si $X$ est affinoïde et si $V$ est un domaine
analytique compact et $\Gamma$-strict de $X$, 
les espaces $X\gad$ et $V\gad$ au sens que nous venons
de définir sont les mêmes que ceux 
déjà introduits plus haut. 

En vertu de la proposition \ref{xgad-filtres}
et par la définition même de la topologie sur $V\gad$ lorsque
$V$ est un domaine affinoïde $\Gamma$-strict de $X$, l'espace $X\gad$
s'identifie à la colimite topologique des $\widetilde V$ lorsque $V$ parcourt 
l'ensemble des domaines affinoïdes $\Gamma$-stricts de $X$, qui n'est autre
que la colimite topologique des $\widetilde V$ lorsque $V$ parcourt $\domc X$
(chaque élément de $\domc X$ étant une union finie d domaines affinoïdes
$\Gamma$-stricts). Autrement dit, $X\gad$ s'identifie naturellement
à $\breve X$.

\subsubsection{}
Il résulte de \ref{lien-gad-temkin}
que
$X\gad$ s'identifie naturellement comme
ensemble à 
$\{(\xi,\omega)\}_{\xi \in X,\omega
\in \widetilde{(X,\xi)}^\Gamma}$ ; si
$x\in X\gad$ le couple $(\xi,\omega)$ correspondant
sera noté $(x^\#,x^\flat)$ comme dans le cas
affinoïde.

\subsubsection{}
Si $\Delta$ est un sous-groupe
de $\rpos$ contenant $\Gamma$ l'espace $X$ est également $\Delta$-strict, 
et l'on dispose d'une application continue propre et surjective $X^{\Delta
\text{-}\mathrm{ad}}\to X\gad$ (\ref{dom-x1-x2}). 
À titre d'exemple, décrivons brièvement
les fibres de $X^{\rpos
\text{-}\mathrm{ad}}\to X\ad$ lorsque $X$ est une $k$-courbe
stricte. Ce sont alors toutes des singletons, sauf 
éventuellement au-dessus des points
de type $3$ de $X$ : la fibre en un tel $x$ contient en effet $x$ lui-même 
mais également un point
par branche issue de $x$ ; et il y a deux branches issues de $x$ si ce dernier
est intérieur, il y en a une si $x$ appartient à $\partial X$ sans être un point isolé, 
et aucune si $x$ est isolé.

\begin{rema}\label{rem-ugad-barre}
Supposons $X$ affinoïde, et soit $V$ un domaine affinoïde et $\Gamma$-strict de $X$. 
Soit $U$ le complémentaire de $V$ dans $X$. L'ensemble $\overline{U\gad}$ défini en \ref{ugad-barre}
est alors précisément l'adhérence de $U\gad$ dans $X\gad$. En effet comme $\overline{U\gad}$ est compact, il suffit de vérifier
que $U\gad$ est dense dans $\overline{U\gad}$. Soit donc $x\in \overline{U\gad}$ et soit $W$ un élément de $\dom X$
tel que $x\in \overline{W\gad}$. On a alors $x\in \overline{U\gad}\cap \overline{W\gad}$. La partie $W\cap U$ est donc un élément non 
vide de $\dom X$, qui contient par conséquent un domaine affinoïde $\Gamma$-strict compact non vide $Y$. Mais alors 
$Y\gad$ est une partie non vide de $U\gad$ qui est contenue dans $\overline{W\gad}$, ce qui permet de conclure. 
\end{rema}

\section{Espaces friables en géométrie linéaire par morceaux}\label{friable-polytope}
Soit $\Delta$ un sous-groupe divisible non trivial de $\rpos$ et soit $c$ le couple
$(\Q,\Delta)$ (en pratique nous appliquerons ceci avec $\Delta=(\abs{k^\times}\cdot \Gamma)^\Q$). 
Si $X$ est un espace $c$-linéaire
par morceaux, nous nous proposons de donner une description «valuative» de $\breve X$, dans un sens qui sera précisé, 
analogue à celle exhibée dans le cas des espaces analytiques. 

\subsection{L'espace des types}
Nous allons tout d'abord présenter 
(dans le contexte $c$-linéaire par morceaux) une construction classique en théorie des modèles, 
l'\textit{espace des types}, qui redonnera l'espace friable dans le cas des $c$-polytopes
compacts.

\subsubsection{}
Soit $n$ un entier et soit 
$\mathsf F$ l'ensemble des formules du premier ordre
dans le langage des groupes abéliens ordonnés à paramètres
dans $\Delta$, en $n$ variables libres $x_1,\ldots, x_n$. 
Concrètement, une formule $\Phi$ appartenant à $\mathsf F$
est construite
à partir des quantificateurs et connecteurs logiques standard ainsi que
des symboles
$1,\cdot, (\cdot)^{-1}$ et $\leq, <,\geq, > $ relatifs à la structure de groupe abélien ordonné, 
et met en jeu d'une part un certain nombre de variables quantifiées, 
d'autre part certaines variables non quantifiées parmi $x_1,\ldots, x_n$, ainsi
que des paramètres appartenant à $\Delta$. Par exemple si $n=3$
et si $r$ est un élément fixé de $\Delta$,
\[\forall x\,\exists y \exists z\, (x_1x_2^3\leq xyz)\;\text{ou}\;(yz^{-3}>r\;\text{et}\;x_3y^{-1}=z^2x)\]
est une formule appartenant à $\mathsf F$. 

Si $H$ est un groupe abélien divisible ordonné $\Delta$-coloré et si $(h_1,\ldots, h_n)$ 
sont des éléments de $H$ on obtient pour toute $\Phi\in \mathsf F$, 
en substituant $h_i$ à $x_i$ quel que soit $i$, une formule
$\Phi(h_1,\ldots, h_n)$ sans variables libres et à paramètres dans $H$, dont la valeur de vérité dans le
groupe abélien ordonné $H$ est bien définie. 

\subsubsection{}\label{el-quant}
La théorie des groupes abéliens divisibles ordonnés non triviaux 
admet l'élimination des quantificateurs dans le langage des groupes abéliens ordonnés. 
Il en résulte que pour toute $\Phi\in \mathsf F$ il existe 
$\Psi\in \mathsf F$ \textit{sans quantificateurs}, telle que
pour tout
groupe abélien ordonné divisible $\Delta$-coloré $H$ et tout $(h_1,\ldots, h_n)\in H^n$
on ait l'équivalence 
$\Phi(h_1,\ldots, h_n)\iff \Psi(h_1,\ldots, h_n)$ dans le groupe $H$.

Il s'ensuit que si
$H'$ est un groupe abélien ordonné $\Delta$-coloré divisible contenant $H$ alors 
pour toute formule
$\Phi\in \mathsf F$, l'énoncé $\Phi(h_1,\ldots, h_n)$
vaut dans $H$ si et seulement s'il  vaut dans $H'$ (car c'est évident si $\Phi$ est sans quantificateurs). 

\subsubsection{}
On note $\mathbb G_n$ l'ensemble des «classes d'équivalence» de 
$n$-uplets $(h_1,\ldots, h_n)$ où les $h_i$
appartiennent à un groupe abélien ordonné
divisible $\Delta$-coloré $H$
(qui n'est pas fixé) pour la relation d'équivalence définie par la 
condition suivante : 
\[(h_1,\ldots, h_n)\sim (h'_1,\ldots, h'_n)\] si et seulement si pour toute 
formule $\Phi\in \mathsf F$ on a 
\[\Phi(h_1,\ldots, h_n)\iff \Phi(h'_1,\ldots, h'_n),\] et il suffit 
de le vérifier lorsque $\Phi$ est de la forme $\phi \leq r$ où
$\phi$ est une fonction $c$-affine à exposants entiers et $r$ un élément de $\Delta$. 

\begin{rema}
Bien que sa définition mette en jeu tous les groupes abéliens divisibles ordonnés
$\Delta$-colorés, $\mathbb G_n$ est bien un ensemble : on peut en effet l'identifier par construction 
à un sous-ensemble de $\mathsf F^{\{0,1\}}$. 

Si $\Phi$ est un élément de $\mathsf F$, sa validité en un élément donné $g$ de $\mathbb G_n$ est bien définie : elle se teste
sur n'importe quel $n$-uplet représentant $g$. 
\end{rema}

\begin{rema}\label{rem-f-fq}
Pour définir $\mathbb G_m$ on aurait pu tout aussi bien partir de l'ensemble $\mathsf F^\Q$ des formules en $n$
variables libres $x_1,\ldots, x_n$ à paramètres dans $\Delta$  dans le langage des groupes abéliens \textit{divisibles}
ordonnés, qui est le langage des groupes abéliens ordonnés auquel on adjoint les symboles d'élévation à une puissance
rationnelle. Cela provient du fait que pour toute $\Phi\in \mathsf F^\Q$ il existe $\Psi\in \mathsf F$ telle que pour
tout groupe abélien divisible ordonné $\Delta$-coloré $H$ et tout $(h_1,\ldots, h_n)\in H^n$
on ait $\Phi(h_1,\ldots, h_n)\iff \Psi(h_1,\ldots, h_n)$. 

\end{rema}

\subsubsection{}
Soit $g\in \mathbb G_n$. Choisissons un groupe
abélien ordonné
divisible $\Delta$-coloré $H$ et un
représentant $(h_1,\ldots, h_n)\in H^n$ 
de $g$. Le groupe $H$ n'est pas canonique, mais par
définition de la relation d'équivalence utilisée pour
définir $\mathbb G_n$, 
le sous-groupe de $H$ engendré par $\Delta$ et les $h_i^\Q$ est bien défini 
à unique isomorphisme près comme groupe abélien ordonné
divisible $\Delta$-coloré muni d'un $n$-uplet
$(h_1,\ldots, h_n)$ ; nous noterons ce groupe $\Theta(g)$ et $(x_1(g),\ldots, x_n(g))$
le $n$-uplet $(h_1,\ldots, h_n)$. 

Si $\phi=a\prod x_i^{e_i}$ est une application $c$-affine (avec $a\in \Delta$ et les $e_i$ dans $\Q$)
nous noterons $\phi(g)$ l'élément $a\prod x_i(g)^{e_i}$ de $\Theta(g)$.

\subsubsection{}
Soit $\dom{\rposn n}$ l'ensemble des parties de $\rposn n$ définissables comme le lieu de validité
d'une formule $\Phi$ appartenant à $\mathsf F^\Q$ (ou à $\mathsf F$ si l'on préfère, 
voir la remarque \ref{rem-f-fq} ci-dessus). 
C'est aussi la sous-algèbre de Boole
de $\mathscr P(\rposn n)$ engendrée par les $c$-polytopes (cette notation est donc compatible avec celle introduite en 
\ref{notation-dx}). Si $P\in  \dom{\rposn n}$ nous noterons $\dom P$ la sous-algèbre de Boole de $\mathscr P(P)$ formée
des parties qui appartiennent à $\dom{\rposn n}$ ; lorsque $P$ est un $c$-polytope compact, cette notation est compatible
avec celle introduite en \ref{def-friable}.

Soit $P\in \dom{\rposn n}$ et soit $\Phi$ une formule appartenant à $\mathsf F^\Q$ définissant $P$. 
Si $\Psi$ est une autre formule appartenant à $\mathsf F^\Q$ et définissant $P$, la formule
\[\forall( x_1,\ldots, x_n )\,\Phi\iff \Psi\] ne comporte plus de variables libres et vaut sur $\rposn n$,  et partant sur
tout groupe abélien ordonné divisible et $\Delta$-coloré (d'après \ref{el-quant}). Le lieu de validité de $\Phi$
sur $\mathbb G_n$ ne dépend donc que de $P$ et non du choix de $\Phi$, et sera noté $\dot P$. 
Si $P=\rposn n$ alors $\dot P=\mathbb G_n$. 

On remarque que $P$ se plonge naturellement dans $\dot P$ (l'injectivité
provient du fait que si $p$ et $q$ sont deux points distincts de $P$ on peut trouver une
application $c$-affine $\psi$ telle que $\psi(p)<1$ et $\psi(q)\geq 1$).

\subsubsection{}\label{application-types}
Soit $P$ appartenant à 
$\dom{\rposn n}$ et soit $\phi$ une application $c$-définissable de $P$ vers $\rpos$, dans le langage
des groupes abéliens ordonnés et à paramètres dans $\Delta$. Cela signifie qu'il existe une famille finie $(P_i)$ d'éléments
de $\dom P$ tels que $P=\bigcup P_i$ et, pour tout $i$ une application $c$-affine $\phi_i$ de $\rposn n$ dans $\rpos$, 
telles que $\phi(x)=\phi_i(x)$ pour tout $i$ et tout $x\in P_i$. 

Soit $g$ un point de $\dot P$. Il existe alors $i$ tel que $g\in \dot P_i$, et l'élément 
$\phi_i(g)$ de $\Theta(g)$ ne dépend pas du choix de $i$ ; on le note $\phi(g)$.

\subsection{Types et ultrafiltres, valuations $c$-polytopales}
Nous nous proposons maintenant d'interpréter l'ensemble des types $\dot P$ comme 
un ensemble d'ultrafiltres, qui coïncidera avec l'espace friable associé à $P$ lorsque ce dernier est compact ; 
nous en donnerons également dans ce cas une description «valuative».

\subsubsection{}
Si $P$ est un élément de
$\dom{\rposn n}$
nous noterons 
$\widetilde P$ l'ensemble des ultrafiltres d'éléments
de $\dom P$. Si $P$ est un $c$-polytope compact, la notation $\widetilde P$ 
est compatible avec celle introduite
en \ref{def-friable}
dans le cadre des espaces domaniaux compacts généraux
(le $c$-polytope $P$ étant vu comme un espace domanial \textit{via} sa structure linéaire par morceaux).

\begin{prop}\label{filtres-polytopaux}
Soit $P$ appartenant à $\dom{\rposn n}$. La formule
\[g\mapsto \{Q\in \dom P, g\in \dot Q\}\]
établit une bijection entre $\dot P$ et $\widetilde P$.
\end{prop}

\begin{proof}
Pour tout $g\in G$, notons 
$\Phi(g)$ l'ensemble des éléments
$Q$ de $\dom P$ tels
que $g\in \dot Q$. 
Il est immédiat que $\Phi(g)$ est un ultrafiltre d'éléments
de $\dom P$. 
Réciproquement soit $\mathscr F$ un tel ultrafiltre. Nous allons montrer qu'il existe 
un unique $g\in \dot P$ tel que $\mathscr F=\Phi(g)$.

Soit $A$ le groupe des applications $c$-affines sur $\rposn n$, et soit $B$ 
l'ensemble des éléments 
$\phi$ de $A$ tels que $\{p\in P, \phi(p)=1\}$ appartienne à $\mathscr F$. C'est un sous-groupe de $A$
(pour la stabilité par produit noter que  $\{p\in P, \phi(p)\psi(p)=1\}$ contient l'intersection de $\{p\in P, \phi(p)=1\}$
et $\{p\in P, \psi(p)=1\}$). Soit $H$ le quotient $A/B$. C'est un groupe abélien divisible dans lequel $\Delta$ se plonge naturellement. 
Définissons sur $A$ la relation $\leq $ en décrétant que  $\phi\leq \psi$ si $\{p\in P, \phi(p)\leq \psi(p)\}$ appartient à 
$\mathscr F$. Cette relation passe au quotient par $B$ et fait ainsi de $H$ un groupe abélien divisible $\Delta$-coloré ordonné
L'image $g$ de $(x_1,\ldots,x_n)$ dans $H^n$
est alors un élément de $\mathbb G_n$, qui appartient à $\dot P$ : en effet si $P$ est défini par la formule
\[\bigvee_i \bigwedge_j \phi_{ij}\Join_{ij} 1\] où les $\phi_{ij}$ sont des applications
$c$-affines et
les $\Join_{ij}$ des symboles appartenant
à $\{\leq, >\}$, le fait que $\mathscr F$ soit 
un ultrafiltre assure
qu'il existe $i$ tel que l'élément de $\dom P$ défini par la formule $\bigwedge_j \phi_{ij}\Join_{ij}1$ appartienne à
$\mathscr F$, et cette formule est alors satisfaite en $g$, qui appartient dès lors à $\dot P$.

Vérifions que $\Phi(g)=\mathscr F$. Soit $Q\in \dom P$. Décrivons $Q$ par une formule
\[\bigvee_a \bigwedge_b \psi_{ab}\Join^{ab} 1\] 
où les $\psi_{ab}$ sont des applications
$c$-affines et
les $\Join^{ab}$ des symboles appartenant
à $\{\leq, >\}$. Le point $g$ appartient à $\dot Q$ si et seulement il existe $a$
tel que  $\bigwedge_b \psi_{ab}(g)\Join^{ab} 1$. Mais par définition du groupe ordonné
$\Delta$-coloré $H$, cette dernière condition est satisfaite si et seulement si le sous-ensemble 
de $P$ défini par la conjonction $\bigwedge_b \psi_{ab}\Join^{ab} 1$ appartient à $\mathscr F$. 
Puisque $\mathscr F$ est un ultrafiltre, il s'ensuit que $g\in \dot Q$ si et seulement si 
$Q\in \mathscr F$, ce qui montre que $\Phi(g)=\mathscr F$.  Et il découle de la construction même de
$g$ que c'est le seul antécédent possible de $\mathscr F$ par $\Phi$. 
\end{proof}

\begin{rema}\label{breve-polytope}
Soit $P$ un élément de $\dom{\rposn n}$, que l'on suppose localement
fermé dans $\rposn n$. C'est alors un espace $c$-linéaire par morceaux, 
qui possède donc un espace friable $\breve P$. Celui-ci est par définition 
la colimite des espaces $\breve Q$ où $Q$
parcourt l'ensemble des $c$-polytopes compacts
contenus dans $P$. Pour un tel $Q$ l'espace $\breve Q$ coïncide par définition 
avec $\widetilde Q$, qui s'identifie à $\dot Q$ par la proposition \ref{filtres-polytopaux}
ci-dessus. Il en résulte que $\breve P$ s'identifie lui-même au sous-ensemble 
de $\dot P$ égal à la réunion des $\dot Q$ où $Q$ parcourt l'ensemble
des $c$-polytopes compacts de $\rposn n$ contenus dans $P$. 

Supposons maintenant que $P=\rposn n$, auquel cas $\dot P=\mathbb G_n$. 
On peut alors identifier $\breve P$ au sous-ensemble $\bigcup \dot Q$
de $\mathbb G_n$ où $Q$ parcourt l'ensemble des $c$-polytopes
compacts de $\rposn n$. Par cofinalité, on peut en fait se contenter de faire parcourir
à $Q$ l'ensemble des pavés $[R^{-1},R]^n$ pour $R\in \Delta\cap(1,+\infty)$. 
En conséquence, $\breve P$ s'identifie à l'ensemble des $g\in \mathbb G_n$ tel qu'il existe $R>1$ dans $\Delta$
pour lequel on a $R^{-1}\leq x_i(g)\leq R$ quel que soit $i$, 
c'est-à-dire encore à l'ensemble des 
$g\in \mathbb G_n$ tel que le groupe abélien ordonné $\Delta$-coloré $\Theta(g)$ soit bridé. 
\end{rema}

\subsubsection{}\label{valuation-polytopale-plongee}
Soit $P$ un $c$-polytope compact de $\rposn n$. Une \textit{valuation $c$-polytopale} $g$
sur $P$ est une application $\Lambda_c(P)\to H$, où $H$ est un groupe abélien divisible ordonné $\Delta$-coloré, que l'on note
$\psi \mapsto \psi(g)$ et qui possède la propriété suivante : \textit{si $\psi_1,\ldots, \psi_m$ sont des éléments de $\Lambda_c(P)$, 
toute formule du langage des groupes ordonnés $G$-colorés satisfaite par $(\psi_1(p),\ldots, \psi_m(p))$ pour
tout $p\in P$ est encore satisfaite par $(\psi_1(g),\ldots, \psi_m(g))$.} On dit que deux telles valuations
$c$-polytopales $g$ et $g'$, 
à valeurs respectivement dans $H$ et $H'$,  sont équivalentes s'il existe un groupe abélien divisible $\Delta$-coloré $H''$, deux 
morphismes injectifs croissants de groupe $\Delta$-colorés $H''\hookrightarrow H$ et $H''\hookrightarrow H'$, et une factorisation 
de $g$ et $g'$ par une même valuation $c$-polytopale à valeurs dans $H''$. 

Si $g$ est une valuation $c$-polytopale sur $P$ à valeurs dans un certain  groupe abélien divisible ordonné $\Delta$-coloré, 
l'ensemble des $\psi(g)$ pour $\psi \in\Lambda_c(P)$ est un sous-groupe divisible de $H$ contenant $\Delta$, qui est simplement
le sous-groupe engendré par $\Delta$ et les $x_i(g)^\Q$. Il ne dépend à unique isomorphisme près que de la classe d'équivalence de $g$, 
et comme nous ne considèrerons les valuations $c$-polytopales qu'à équivalence près, nous nous permettrons de le noter $\Theta(g)$. 

Le paragraphe \ref{application-types} permet d'associer à tout élément $g$ de $\dot P$
une application $\psi \mapsto \psi(g)$ définie sur une certain ensemble d'applications contenant $\Lambda_c(P)$, et sa restriction
à $\Lambda_c(P)$ est une valuation $c$-polytopale. 
Réciproquement si $g$ est une valuation polytopale sur $P$, l'élément $(x_i(g),\ldots, x_n(g))$ de $\Theta(g)$ définit un point 
de $\dot P$, et ces deux constructions mettent en bijection $\dot P$ avec l'ensemble des classes d'équivalence
de valuations $c$-polytopales
sur $P$ ; nous identifierons désormais librement ces deux ensembles, et les notations $\Theta(g)$ de ce paragraphe et du
paragraphe \ref{application-types} ci-dessus sont bien compatibles modulo cette identification. L'avantage de la notion de valuation
polytopale est qu'elle ne fait référence qu'à la structure $c$-linéaire par morceaux de $P$, et pas à sa description 
comme partie de $\rposn n$ (elle ne fait jouer aucun rôle particulier aux fonctions coordonnées $x_i$) ; elle garde un sens
sur un espace $c$-linéaire par morceaux quelconque. 

\subsubsection{}
Soit $P$ un $c$-polytope compact. 
En combinant \ref{valuation-polytopale-plongee}
et la proposition \ref{filtres-polytopaux} on obtient une bijection entre 
l'ensemble des valuations $c$-polytopales sur $P$ et $\widetilde P$, 
qui lui-même est égal à $\breve P$ puisque $P$ est un espace $c$-linéaire par morceaux compact.

\begin{rema}\label{val-polytopale-generale}
Soit $P$ un espace $c$-linéaire par morceaux. 
Si $P$ est compact et  \textit{plongeable} au sens de \ref{plongeable}, 
il résulte de ce qui précède que $\breve P$
s'identifie à l'ensemble des valuations $c$-polytopales sur $P$, dont la définition ne fait
référence qu'à la structure abstraite d'espace $c$-linéaire par morceaux de $P$
et ne nécessite pas de choix de plongement. 

Ne faisons plus d'hypothèse sur $P$. Soit $g$ un point de $\breve P$. 
Par définition, $g$ appartient à $\breve Q$ pour un certain sous-espace
$c$-linéaire par morceaux compact $Q$ de $P$, que l'on peut supposer
plongeable. On peut alors interpréter $g$ comme une 
valuation $c$-polytopale sur $Q$. Il résulte des constructions (et d'une réduction
immédiate au cas plongé) que 
$\Theta(g)$ ne dépend pas du choix de $Q$. Si $\phi$ est une application $c$-linéaire par morceaux
définie sur une partie $c$-linéaire par morceaux $R$ de $P$ telle
que $g\in \breve R$ on peut choisir $Q$ comme ci-dessus en demandant de plus
que $Q\subset R$, ce qui permet de restreindre $\phi$ à $Q$ et partant de donner un 
sens à l'élément $\phi(g)$ de $\Theta(g)$ ; cet élément ne dépend pas du choix de $Q$
(là encore, on le voit par réduction au cas plongé). On a donc défini une valuation $c$-polytopale
$\phi \mapsto \phi(g)$ sur $\mathrm{colim}\,\Lambda_c(R)$ , la colimite étant prise sur l'ensemble 
des parties $c$-linéaires par morceaux $R$ de $P$ telles que $g\in \breve R$. 
\end{rema}

\subsection{Squelettes adiques}

On fixe un espace $k$-analytique $\Gamma$-strict $X$. On désigne à nouveau par
$c$ le couple $(\Q, (\Gamma\cdot \abs{k^\times})^\Q$. 

\subsubsection{}
Soit $\Sigma$ un $c$-squelette de $X$. Par définition, sa structure $c$-linéaire par morceaux
est simplement la structure domaniale induite par la structure domaniale $\Gamma$-strictement
analytique de $X$. Nous noterons $\Sigma\gad$ l'espace friable associé à l'espace
domanial $\Sigma$. L'inclusion $\Sigma\hookrightarrow X$ permet d'identifier
$\Sigma\gad$ à une partie localement fermée de $X\gad$, fermée si 
$\Sigma$ est fermée ; si $x\in \Sigma\gad$ alors
$x^\#\in \Sigma$, et $\Sigma\gad \cap X=\Sigma$ (voir \ref{breve-locferm}).

\subsubsection{}
Soit $x$ un point de $\Sigma\gad$. Il résulte de notre étude générale de l'espace friable associé à un espace
$c$-linéaire par morceaux (voir la sous-section
\ref{friable-polytope} et plus particulièrement \ref{valuation-polytopale-plongee}
et la remarque \ref{val-polytopale-generale}) qu'au point $x$ est associé
un groupe abélien divisible ordonné, $(\abs{k^\times}\cdot \Gamma)^\Q$-coloré et bridé $\theta(x)$ et une
une valuation $c$-polytopale $\lambda\mapsto \lambda(x)$, à valeurs dans $\theta(x)$ et
dont le domaine de définition 
est $\mathrm{colim}\,\Lambda_c(Q)$, la colimite étant prise sur l'ensemble (filtrant)
des parties $c$-linéaires par morceaux $Q$ de $\Sigma$ telles que $x\in Q\gad$. 

Soit $V$ un domaine analytique $\Gamma$-strict de $X$ tel que $x\in V\gad$ 
et soit $f$ une fonction analytique inversible sur  $V$. L'intersection $V\cap \Sigma$ est une partie
$c$-linéaire par morceaux de $\Sigma$, et la restriction 
de $\abs f$ à  $V\cap \Sigma $
est $c$-linéaire par morceaux. 
Le point $x$ appartient à $(V\cap \Sigma)\gad$, si bien que $\abs f(x)$ est un élément bien défini du groupe
$\theta(x)$. 

Il découle alors de nos constructions que
\[\theta(x)=(\abs{\hr x^\times}\cdot \Gamma)^\Q\;\text{et}\;
\abs f(x)=\abs{f(x)}.\]

\begin{exem}[Le cas de $S_n\gad$]\label{exemple-snad}
Soit $n$ un entier. L'espace $c$-linéaire par morceaux $S_n$ est isomorphe à 
$\rposn n$ \textit{via} $(\abs {T_1},\ldots, \abs{T_n})$ ; modulo cette identification, 
on peut parler d'application
$c$-affine sur $S_n$ : c'est une application
de $S_n$ dans $\rpos$ de la forme
$a\prod \abs{T_i}^{e_i}$
où $a\in (\abs{k^\times}\cdot \Gamma)^\Q$ et
où les $r_i$ appartiennent à $\Q$. 

Il résulte alors du 
lemme \ref{filtres-polytopaux}
et de la remarque \ref{breve-polytope}
que 
la flèche $x\mapsto (\abs{T_i(x)})_i$ établit une bijection entre 
$S_n\gad$ et l'ensemble des $n$-uplets $(r_1,\ldots, r_n)$ d'éléments d'un groupe abélien divisible
$(\abs{k^\times}\cdot \Gamma)$-coloré et bridé, ces $n$-uplets étant considérés modulo la relation d'équivalence
qui identifie deux $n$-uplets satisfaisant les mêmes inéquations monomiales à exposants rationnels et termes
constants dans $\abs{k^\times}\cdot \Gamma$. Soit 
$r=(r_1,\ldots, r_n)$ un tel $n$-uplet ; 
nous noterons $\eta_r$ le point correspondant de $S_n\gad\subset (\mathbf G_{\mathrm m,k}\an)^{n,
\Gamma\text{-}\mathrm{ad}}$.
Lorsqu'on considère $(r_1,\ldots, r_n)$ à 
équivalence près, le groupe abélien ordonné $(\abs{k^\times}\cdot \Gamma)$-coloré et bridé
dans lequel vivent les $r_i$ est mal défini, mais son sous-groupe
engendré par $(\abs{k^\times}\cdot \Gamma)$ et les $r_i$ est canonique, et nous le noterons
$\Theta(r)$. 

Nous allons
maintenant décrire concrètement $\eta_r$,
le point clef étant que 
toute application $c$-linéaire par morceaux de $S_n$ dans $\rpos$ peut être évaluée en $\eta_r$
à partir des formules $\abs{T_i(\eta_r)}=r_i$, la valeur obtenue étant un élément
de $\Theta(r)$.
Soit $V$ un domaine analytique compact et $\Gamma$-strict 
de $\gma n$ 
et soit $f$ une fonction analytique inversible sur $V$. Posons $P=V\cap S_n$. Le compact $P$ est un une partie 
$c$-linéaire par morceaux 
de $S_n$ et $|f|_{|P}$ est $c$-linéaire par morceaux. En tant que $c$-polytope de
$S_n\simeq \rposn n$, le compact $P$ peut être décrit par une combinaison booléenne finie d'inégalités
de la forme $\phi\leq 1$ où $\phi$ est $c$-affine, et le point $\eta_r$ appartient alors
à $V\gad$ si et seulement si cette combinaison booléenne d'inégalités est satisfaite en $\eta_r$. 
Si c'est le cas $\abs f_{|P}(\eta_r)$ a un sens et $\abs{f(\eta_r)}$ est précisément 
égal à l'élément  $\abs f_{|P}(\eta_r)$ de $\Theta(r)$. 

En particulier, supposons que $f$ soit un polynôme de Laurent $\sum a_IT$ en $T_1,\ldots, 
T_n$. On a alors
$\abs f_{|S_n}=\max \abs{a_I}\cdot \abs{T^I}$, si bien que $\abs {f(\eta_r)}=\max a_I r^I$. 
En fait cette formule \textit{caractérise} le point $\eta_r$, et on peut même 
se contenter d'en supposer la validité
pour les polynômes usuels (en degrés positifs). En effet, pour tout $i$, notons
$r^\#$ la borne supérieure dans $\rpos$ de $\{s\in (\abs{k^\times}\cdot \Gamma)^\Q), s\leq r_i\}$
(qui existe puisque le groupe $\abs{k^\times}\cdot \Gamma$-coloré $\Theta(r)$ est bridé). 
Soit $x$ un point de $(\gma n)\gad$ tel que l'on ait 
$\abs{\sum a_I T^I(x)}=\max \abs{a_I}\cdot r^I$ 
pour tout polynôme $\sum a_I T^I\in k[T]$. On a alors par construction 
$\abs{\sum a_I T^I(x^\#)}=\max \abs{a_I}\cdot (r^\#)^I$ 
pour tout polynôme $\sum a_I T^I$ ; autrement dit, $x^\#$ est nécessairement égal
au point $\eta_{r^\#}$ de $S_n$. Et comme l'image de $k[T]$ dans $\hr{\eta_r^\#}$ engendre un sous-corps dense, 
l'application $\sum a_IT^I\mapsto \abs{\sum a_IT^I(x)}$ détermine uniquement $x$
(\ref{raffinement-completion}), qui coïncide donc
avec $\eta_r$.
\end{exem}

\begin{prop}\label{squelad-abhyankar}
Soit $X$ un espace $k$-analytique
$\Gamma$-strict de dimension $\leq n$ et soient
$(f_1,\ldots, f_n)$ des fonctions analytiques inversibles
sur $X$. Soit $\phi$ le morphisme de $X$
vers $\gma n$ induit par les $f_i$ et soit 
$\Sigma$ le $c$-squelette $\phi^{-1}(S_n)$. 
Soit $x\in X\gad$. 
Les assertions suivantes sont équivalentes : 

\begin{enumerate}[i]
\item Le point $x$ appartient à $\Sigma\gad$. 
\item La valuation $x$ est d'Abhyankar et $f_1(x),\ldots,
f_n(x)$ est une base d'Abhyankar de $\hr x$ sur $k$. 
\end{enumerate}
\end{prop}

\begin{rema}\label{remarque-sigma-ad}
Au vu de l'exemple 
\ref{exemple-snad}, l'assertion (ii) revient à demander que $x$
appartienne à $\phi^{-1}(S_n\gad)$. La proposition peut donc se reformuler en 
disant simplement que $\Sigma\gad=\phi^{-1}(S_n\gad)$. 
(Par abus nous noterons encore $\phi$ l'application de $X\gad$ vers 
$(\mathbf G_{\mathrm m,k}\an)^{n,
\Gamma\text{-}\mathrm{ad}}$
induite par $\phi$).

\end{rema}

\begin{proof}
Supposons que (i) soit satisfaite. Soit $P=\sum a_IT^I$ un polynôme
à coefficients dans $k$, où
$T=(T_1,\ldots, T_n)$. Soit $V_P$ le domaine analytique
$\Gamma$-strict de $X$ défini par la condition $\abs{P(f)}=\max_I \abs{a_I}\cdot \abs f^I$ (où $f=(f_1,\ldots, f_n)$). 
On a $\Sigma\subset V_P$ par définition de $\Sigma$, si bien que $\Sigma\gad\subset V_P\gad$. Puisque (i)
est satisfaite par hypothèse, il s'ensuit que 
$x\in V_P\gad$, ce qui veut dire que $\abs{P(f)(x)}=\max_I\abs{a_I}\cdot \abs{f(x)}^I$. Ceci valant pour tout
$P$, l'assertion (ii) est satisfaite.

Supposons réciproquement que (ii) soit satisfaite. 
Pour montrer (i) on peut remplacer $X$ par n'importe quel domaine
affinoïde $\Gamma$-strict $V$ de $X$ tel que $x\in V\gad$, ce qui permet de
supposer que $X$ est affinoïde. L'assertion à prouver étant insensible aux phénomènes
de nilpotence on peut supposer $X$ réduit. 
Comme $x\in \Sigma\gad$, le point
$x^\#$ appartient à $\Sigma$, si bien que $\phi(x^\#)$
appartient à $S_n$. Ceci entraîne que $\phi$ est plat en $x^\#$, 
et on sait  par ailleurs que $\phi^{-1}(\phi(x^\#))$ est fini. 
Quitte
à remplacer $X$ par un voisinage affinoïde $\Gamma$-strict $V$
de $x^\#$ dans $X$ (notons qu'on a
automatiquement $x\in V\gad$) on peut donc supposer que $\phi$
est plat et que $\phi^{-1}(\phi(x^\#))=\{x^\#\}$
(ensemblistement). 
Comme $\phi$ est plat et $X$ compact, 
$\phi(X)$ est un domaine analytique compact et $\Gamma$-strict 
$U$ de $\gma n$. Soit $\Upsilon$ l'intersection $S_n\cap U$. 
Le morphisme $\phi$ induit une surjection $c$-linéaire
par morceaux de $\Sigma$ sur $\Upsilon$, à fibres finies. 
Il existe alors une famille finie $(\Sigma_j)$ de parties $c$-linéaires 
par morceaux compactes de $\Sigma$ qui recouvrent ce dernier et sont telles que
$\phi_{|\Sigma_i}$ soit injective pour tout $i$.  Quitte à raffiner ce recouvrement 
on peut supposer que chacun des $\Sigma_j$ est de la forme $V_j\cap \Sigma$
où $V_j$ est un domaine affinoïde $\Gamma$-rationnel de $X$. Pour tout $i$, posons
$\Upsilon_j=\phi(\Sigma_j)$. Les $\Upsilon_j$ sont des parties $c$-linéaires
par morceaux compactes de $\Upsilon$ et $\Upsilon=\bigcup_j \Upsilon_j$. 

L'hypothèse (ii) implique que $\phi(x)$ est d'Abhyankar et que $(\abs{T_i(\phi(x))})_i$ 
en est une base d'Abhyankar. En vertu de l'exemple \ref{exemple-snad}, cela revient à dire
que $\phi(x)$
appartient à $\Upsilon\gad$. Comme $\Upsilon=\bigcup \Upsilon_j$, l'ensemble $J$ des indices $j$ tels que
$\phi(x)\in \Upsilon_j\gad$ est non vide. Soit $j\in J$. Le morphisme $\phi$ induit un isomorphisme
$c$-linéaire par morceaux $\Sigma_j\simeq \Upsilon_j$. Il existe donc un unique antécédent de $\phi(x)$
sur $\Sigma_j\gad$, que nous noterons $x_j$. Remarquons que pour un 
tel $j$ on a $\phi(x_j^\#)=x^\#$, ce qui entraîne que
$x_j^\#=x^\#$ d'après nos hypothèses. 
Nous allons maintenant démontrer que $x$ est égal à l'un des $x_j$, 
ce qui permettra de conclure. 

Soit $f$ une fonction analytique inversible définie sur un 
voisinage affinoïde $\Gamma$-strict $\Omega$
de $x^\#$ dans $X$. 
Comme $\abs{\hr {\phi(x)}^\times}^\Q=\abs{\hr x^\times}^\Q$
il existe
$a\in k^\times$, des entiers relatifs $e_1,\ldots, e_n$ et un entier $N>0$ tel que 
$\abs{f^N(x)}=\abs a\cdot \left|\prod T_i(\phi(x))^{e_i}\right|=\abs a\cdot  \left|\prod f_i(x)^{e_i}\right|$. Soit $V$ le domaine
affinoïde $\Gamma$-strict de $\Omega$ défini par l'égalité $\abs {f^N}=\abs a\cdot
\prod \abs{f_i^{e_i}}$. Par construction, 
$V\gad$ contient $x$. Par platitude de $\phi$, l'image $W:=\phi(V)$ est un domaine analytique compact et $\Gamma$-strict
de $U$, qui contient $\phi(x)$. Le point $\phi(x)$ appartient donc à $(\Upsilon\cap W)\gad=\Upsilon\gad \cap W\gad$. 
Or $\Upsilon\cap W$ est égal à $\phi(\Sigma\cap V)=\bigcup_j \phi(\Sigma_j\cap V)$, chacun des $\phi(\Sigma_j\cap V)$ est un 
compact $c$-linéaire par morceaux $\Upsilon _j$ de $\Upsilon_j$, et $\phi$ induit
pour tout $j$ un isomorphisme
$\Sigma_j\cap V\simeq \Upsilon_j$.  Le point $\phi(x)$ appartient alors à $\Upsilon_j\gad$ pour un certain $j$, et est donc l'image
d'un unique point $x'$ de $(\Sigma_j\cap V)\gad \subset \Sigma_j\gad$ ; il s'ensuit que $j\in J$ et que $x'=x_j$. On a
alors
\[\abs{f^N(x)}=\abs a\cdot \left|\prod T_i(\phi(x))^{e_i}\right|=\abs a\cdot \left|\prod f_i(x_j)^{e_i}\right|=\abs{f^N(x_j)},\]
la dernière égalité résultant du fait que $x_j\in V\gad$, et de la définition de $V$. On en déduit
que $\abs{f(x)}=\abs{f(x_j)}$. On a donc démontré que pour toute fonction analytique $f$
définie et inversible au voisinage de $x^\#$ sur $X$,
il existe 
$j\in J$ tel que $\abs{f(x)}=\abs {f(x_j)}$. 

Nous allons maintenant conclure. Puisque $x_j^\#=x^\#$ pour tout $j\in J$, les valuations
résiduelles $x^\flat$ et $x_j^\flat$ pour $j\in J$ sont toutes des valuations sur le corpoïde résiduel
$\hrt{x^\#}^\Gamma$, qui induisent la même valuation $\phi(x)^\flat$ sur $\hrt{\phi(x^\#)}^\Gamma$. Or puisque
$\hrt{x^\#}^\Gamma$ est une extension finie de $\hrt{\phi(x^\#)}^\Gamma$, les extensions de
$\phi(x)^\flat$ à $\hrt{x^\#}^\Gamma$ sont deux à deux non comparables
pour la relation de raffinement 
(chacune d'elle a un groupe de valeurs commensurable à celui
de $\phi(x^\flat)$) , et elles sont donc indépendantes par la variante 
graduée d'un résultat classique de théorie des valuations. Comme il résulte de ce qui précède que pour tout $\alpha$
non nul dans $\hrt{x^\#}^\times$ il existe $j\in J$ tel que $\abs{\alpha(x^\flat)}=\abs{\alpha(x_j)}^\flat$, 
on en déduit qu'il existe $j\in J$ tel que $x^\flat=x_j^\flat$, ce qui signifie que $x=x_j$. 
\end{proof}

\subsubsection{}
Soient $X, f_1,\ldots, f_n, \phi$ et $\Sigma$ comme dans la 
proposition \ref{squelad-abhyankar} ci-dessus. Soit
$\xi$ un point de $\Sigma$. L'ensemble des points
$x$ de $X\gad$ tels que $x^\#=\xi$ s'identifie
à $\widetilde{(X,\xi)}^\Gamma$ ; l'ensemble des points
$x$ de $\Sigma\gad$ tels que $x^\sharp=\xi$ apparaît ainsi comme un sous-ensemble
de $\widetilde{(X,\xi)}^\Gamma$ que nous nous proposons de décrire. 

Pour ce faire, faisons une première remarque. Soit $(g_1,\ldots, g_n)$ une famille de fonctions
analytiques inversibles sur $X$ obtenues en appliquant (multiplicativement) une matrice 
de $\mathrm M_n(\Z)$ de déterminant non nul à $(f_1,\ldots,f_n)$,
puis en multipliant chacune des fonctions obtenues par un élément
de $k^\times$. 
Soit $\psi$ le morphisme
$X\to \gma n$ induit par les $g_i$. Pour tout point $y$ de $X\gad$, la famille $(\widetilde{g_i(y)})$
est algébriquement indépendante sur $\widetilde k$ si et seulement c'est le cas de la famille 
$(\widetilde{f_i(y)})$. Par conséquent $\Sigma=\psi^{-1}(S_n)$ et $\Sigma\gad=\psi^{-1}(S_n\gad)$
(remarque \ref{remarque-sigma-ad}). 

Choisissons alors $(g_1,\ldots, g_n)$ comme ci-dessus de sorte que la propriété suivante
soit satisfaite : il existe $r$ tel que $\abs{g_i(\xi)}$ appartienne à $\Gamma$
pour tout $i\leq r$, et tel que les $\abs{g_i(\xi)}$ pour $i>r$ soient $\Q$-linéairement
indépendants modulo $\abs{\hr \xi^\times}^\Q\cap (\Gamma\cdot \abs{k^\times})^\Q$. 
Sous cete hypothèse
les $\widetilde {g_i(\xi)}$ pour $1\leq i\leq r$ forment alors
une base de transcendance de $\hrt \xi^\Gamma$ sur $\widetilde k$. 

Soit $x$ un point  de $X\gad$ tels que $x^\#=\xi$. 
Compte-tenu de la définition explicite
du point $x^\flat$ de $\widetilde{(X,\xi)^\Gamma}$ qui correspond à $x$, 
on voit que les assertions suivantes sont équivalentes : 
\begin{enumerate}[i]
\item $\abs{\sum a_Ig^I(x)}=\max \abs{a_I}\cdot \abs{g(x)}^I$
pour tout polynôme $\sum a_I T^I\in k[T_1,\ldots, T_n]$  ; 
\item $\abs{\sum a_Ig^I(x^\flat)}=\max \abs{a_I}\cdot \abs{g(x^\flat)}^I$
pour tout polynôme $\sum a_I T^I\in k[T_1,\ldots, T_r]$.
\end{enumerate}

Il résulte de ce qui précède que
l'ensemble des points $x$ de $\Sigma\gad$ tels que $x^\#=\xi$ s'identifie, en tant que sous-ensemble
de $\widetilde{(X,\xi)^\Gamma}$, à celui formé par les valuations de $\hrt \xi^\Gamma$ dont les
$\widetilde {g_i(\xi)}$ pour $1\leq i\leq r$ forment une base d'Abhyankar
sur $\widetilde k^\Gamma$.

\section{Images directes des squelettes}

\subsection{Squelettes élémentaires et classiques}
Soit $X$ un espace $k$-analytique $\Gamma$-strict. 

\subsubsection{}
Nous appellerons \emph{sous-espace} de $X$ tout 
espace $k$-analytique $Y$ muni d'un monomorphisme
$Y\hookrightarrow X$
induisant un homéomorphisme sur son image.

Précisons que les monomorphismes dans la catégorie $k$-analytique
sont bien compris. On déduit en effet de la version de Temkin du théorème de
Gerritzen-Grauert (\cite{temkin2005}, théorème 1.1)
que pour qu'une flèche $Y\to X$ soit un monomorphisme, il faut et il suffit qu'elle 
soit ensemblistement injective et que $Y$ admette un G-recouvrement affinoïde $(Y_i)$ tel que
$Y_i\to X$ se factorise pour tout $i$ par une immersion fermée $Y_i\hookrightarrow X_i$ où $X_i$ est un domaine
affinoïde de $X$. Si c'est le cas et si $Y$ est $\Gamma$-strict, on peut demander que les $X_i$ et les $Y_i$ soient
$\Gamma$-stricts : on le déduit de la variante $\Gamma$-stricte du théorème de Temkin, qui se démontre comme la version originelle
en remplaçant les réductions graduées générales par les réductions graduées $\Gamma$-strictes. 

Un monomorphisme induit donc une bijection continue de sa source sur son image. Insistons sur le fait que dans notre définition de
sous-espace, nous demandons de surcroît que cette bijection continue soit un homéomorphisme, ce qui est toujours le
cas si $Y$ est compact et $X$ topologiquement séparé, mais n'est pas automatique en général 
comme en atteste l'exemple suivant : on suppose que $k$ n'est pas trivialement valué, on choisit un élément $\varpi$ dans $k$
tel que $0<\abs \varpi <1$, et l'on note $i_n\colon \mathscr M(k)\to \A^1$ l'immersion
fermée donnée par le point rigide $T=\pi^n$ si 
$n>0$, et $T=0$ si $n=0$ ; alors $\coprod i_n\colon \coprod_n \mathscr M(k)\to \A^1$ est un monomorphisme
de source non compacte et d'image compacte. 

\subsubsection{}
Si $V$ est un domaine analytique de $X$ et si $Y\hookrightarrow V$ est une immersion fermée, la 
composée $Y\hookrightarrow V\hookrightarrow X$ fait de $Y$ un sous-espace de $X$.

Réciproquement supposons que $X$ est topologiquement séparé et que $Y\hookrightarrow X$ est un sous-espace
compact de $X$. Il existe alors par compacité de $Y$ une famille finie $(Y_i,X_i)$ où $Y_i$ et $X_i$
sont pour tout $i$ des domaines affinoïdes de $Y$ et $X$ respectivement, qu'on choisit $\Gamma$-stricts
si $Y$ et $X$ le sont, tels que $Y_i\hookrightarrow X$ se factorise
par une immersion fermée $Y_i\hookrightarrow X_i$.
Choisissons pour tout $i$ un domaine analytique compact $X'_i$ de $X_i$, 
que l'on prend $\Gamma$-strict si $Y$ et $X$ le sont, 
tel que $X'_i\cap Y=Y_i$ ; alors $V=\bigcup X'_i$ est un domaine analytique compact de $X$, 
qui est $\Gamma$-strict si $Y$ et $X$ le sont, et est tel que $Y\hookrightarrow X$ se factorise
par une immersion fermée $Y\hookrightarrow V$. 

%
%
%

\subsubsection{}
Nous noterons $\sel X$ l'ensemble des parties $\Sigma$
de $X$ possédant la propriété suivante : il existe un sous-espace 
$\Gamma$-strict $Y$ de $X$ purement de dimension $d$
et un morphisme
$f\colon Y\to \gma d$ tel que $\Sigma$ soit une partie
$c$-linéaire par morceaux du $c$-squelette $f^{-1}(S_d)$
(que ce dernier soit un $c$-squelette résulte du théorème \ref{imrec-squel}). 

\subsubsection{}
Soit $\Sigma\in \mathscr T(X)$ et soient $d, Y$ et $f$ comme ci-dessus. Par construction, 
$\Sigma$ est un $c$-squelette et $d_k(x)=d$ pour tout $x\in \Sigma$. 
Notons qu'on peut remplacer $(Y,f)$ par $(Y',f|_{Y'})$ pour tout 
sous-espace $\Gamma$-strict
$Y'$ de $Y$ purement de dimension 
$d$ et contenant $\Sigma$. 
Cela permet par exemple si besoin de supposer $Y$ normal : on commence par remplacer $Y$
par $Y_{\mathrm{red}}$, puis l'on observe que tout point de $y
$ tel que $d_k(y)=d$ (et en particulier tout point de $\Sigma$)
est alors contenu dans le lieu normal de $Y$ ; on remplace alors
$Y$ par son lieu normal. 
Cela permet aussi, si $\Sigma$ est compact et $Y$
topologiquement séparé, de supposer $Y$ compact, en le
remplaçant 
par n'importe lequel de ses domaines analytiques compacts et $\Gamma$-stricts
contenant $\Sigma$.

\begin{defi}
Nous dirons qu'un $c$-squelette de $X$ est \textit{élémentaire}
s'il appartient à $\sel X$. 
\end{defi}

\begin{lemm}\label{squel-elem-gloc}
Soit $d$ un entier et soit $X$ un espace $k$-analytique de dimension
au plus $d$. 
Soit $(X_i)$ une famille localement finie de domaines analytiques $\Gamma$-stricts fermés de $X$, et pour tout $i$
soit $f_i$ un morphisme de $X_i$ vers $\gma {d}$ et soit $\Sigma_i$ une partie $c$-linéaire
par morceaux fermée du $c$-squelette $f_i^{-1}(S_d)$. La réunion des $\Sigma_i$ est un $c$-squelette de $X$.
\end{lemm}

\begin{proof}
Introduisons quelques notations. On pose $\Sigma=\bigcup_i
\Sigma_i$ et pour tout $i$ l'on désigne par 
$(f_{i1},\ldots, f_{id})$ le $d$-uplet de fonctions inversibles sur $X_i$
qui définit $f_i$.

L'assertion à prouver est G-locale, ce qui permet de supposer $X$ affinoïde. 
La famille $(X_i)$ est alors une
famille finie de domaines analytiques $\Gamma$-stricts compacts de $X$. Soit $x\in X$ et soit $I$ l'ensemble
des indices
$i$ tels que $x\in X_i$. Pour tout $i\in I$
et tout $j$ entre $1$ et $d$, il existe une fonction analytique $g_{ij}$ définie et inversible au voisinage de
$x$ telle que $|g_{ij}(x)-f_{ij}(x)|<|f_{ij}(x)|$. Soit $V$ un voisinage affinoïde et $\Gamma$-strict de $x$ tel que $V\cap X_j=\emptyset$
pour tout $j\notin I$, et  tel que les $g_{ij}$ soient toutes définies et inversibles sur $V$ et satisfassent les inégalités
$|g_{ij}-f_{ij}|<|f_{ij}|$ sur $V\cap X_i$ pour tout $i\in I$ et tout $j$. Posons $g_i=(g_{i1},\ldots, g_{id})$ pour tout $i\in I$. 
On a alors $((g_i)_{|V\cap X_i})^{-1}(S_d)=((f_i)_{|V\cap X_i})^{-1}(S_d)$ pour tout $i\in I$, si bien que $\Sigma_i\cap V$ est contenu pour
tout $i\in I$ dans le $c$-squelette $g_i^{-1}(S_d)$, 
et \textit{a fortiori}
dans le $c$-squelette $\Tau:=\bigcup_i
g_i^{-1}(S_d)$. Il en va de même de $\Sigma_ i\cap V$ pour $i\notin I$ car alors
$\Sigma_i \cap V=\emptyset$. Par conséquent $\Sigma\cap V$ est la réunion finie des parties $c$-linéaires
par morceaux compactes $\Sigma_i\cap V$ de $\Tau$ ; c'est donc une partie $c$-linéaire par morceaux
compacte de $\Tau$, et en particulier un $c$-squelette. 
Le point $x$ ayant été arbitrairement choisi, $\Sigma$ est un $c$-squelette. 
\end{proof}

\begin{prop}\label{gunion-squelettes}
Soit $\Sigma$ une partie de $X$ possédant un $G$-recouvrement par des $c$-squelettes élémentaires. 
La partie $\Sigma$ est alors un $c$-squelette. 
\end{prop}

\begin{proof}
L'assertion est locale sur $\Sigma$, ce qui permet de supposer que $X$ est compact et qu'il existe une
famille finie $(\Sigma_i)_{i\in I}$ de $c$-squelettes élémentaires compacts de $X$ tels que 
$\Sigma=\bigcup_i \Sigma_i$. Pour tout $i$ on peut choisir un  domaine analytique compact et $\Gamma$-strict
$V_i$
de $X$, un sous-espace analytique fermé $Y_i$ de $V_i$
purement de dimension $d_i$ pour un certain $d_i$, 
et un morphisme $f_i\colon Y_i\to \gma {d_i}$ tel que $\Sigma_i$ soit une partie $c$-linéaire par morceaux
de $f_i^{-1}(S_{d_i})$. 

Comme chaque $\Sigma_i$ est un $c$-squelette compact et comme les $\Sigma_i$ recouvrent $\Sigma$ il suffit,
pour prouver que $\Sigma$ est un $c$-squelette, de démontrer que pour tout $i$, le compact $\Sigma_i$
est un domaine de $\Sigma$ pour la structure domaniale naturelle de ce dernier, c'est-à-dire qu'il est égal à la trace
sur $\Sigma$ d'un domaine analytique compact et $\Gamma$-strict de $X$. Fixons donc $i$. Pour tout $j$, nous noterons
$(F_{j\ell})_{\ell\in \Lambda_j}$ la famille finie des composantes irréductibles de $Y_j\cap V_i$, et 
pour tout $(j,\ell)$ nous poserons $\Sigma_{j\ell}=\Sigma_j\cap F_{j\ell}$. 

Soit $\ell\in \Lambda_i$
et soit $(j,m)$ un couple d'indices. 

Si $d_j\neq d_i$ alors $\Sigma_{i\ell}\cap \Sigma_{jm}=\emptyset$, 
puisque $d_k(x)=d_i$ (resp. $d_j$) pour tout $x\in \Sigma_i$ (resp. $\Sigma_j$).

Supposons que $d_j=d_i$ et que $F_{jm}$ n'est pas une composante irréductible de $V_j\cap F_{i\ell}$. L'intersection 
 $F_{i\ell}\cap F_{jm}$ est alors un fermé de Zariski de $V_j\cap V_i$ de dimension $<d_i$, et il ne contient donc aucun point de
 $\Sigma_i$ ni de $\Sigma_j$ ; par conséquent, $\Sigma_{i\ell}\cap \Sigma_{jm}=\emptyset$. 
 
Les $\Sigma_{i\ell}$ pour $\ell$ variable sont deux à deux disjoints. Il s'ensuit en vertu de ce qui précède qu'il existe
une famille $(W_\ell)_{\ell\in \Lambda_i}$ de domaines analytiques compacts et $\Gamma$-stricts
deux à deux disjoints de $V_i$ tels que : 
\begin{itemize}[label=$\diamond$]
\item $W_\ell$ contient $\Sigma_{i\ell}$ pour tout $\ell$ ; 
\item pour tout $\ell$ et tout $(j,m)$ tel que $d_j\neq d_i$ ou tel que $F_{jm}$ ne
soit pas une composante irréductible de $V_j\cap F_{i\ell}$, on a $W_\ell\cap \Sigma_{jm}=\emptyset$.
\end{itemize}

Pour tout $\ell\in \Lambda_i$, soit
$\mathscr J(\ell)$ l'ensemble des couples
$(j,m)$ tels que
$d_j=d_i$ et que $F_{jm}$ soit une composante irréductible de $V_j\cap F_{i\ell}$.
Pour tout $\ell$ l'intersection de $\Sigma$ et $W_\ell$ est par construction la réunion de 
$\Sigma_{i\ell}$ et de parties de la forme $\Sigma_{jm}\cap W_\ell$, où $(j,m)
\in \mathscr J(\ell)$. 

Soit $\ell\in \Lambda_i$ et soit $(j,m)\in\mathscr J(\ell)$. 
Le $c$-squelette $\Sigma_{jm}$ est contenu dans
la composante irréductible $F_{jm}$
du
domaine analytique $\Gamma$-strict
$V_j\cap F_{i\ell}$ de $F_{i\ell}$. Puisqu'il est constitué
de points $x$ tels que $d_k(x)=d_i$, ce $c$-squelette ne rencontre
aucune autre composante irréductible de $V_j\cap F_{i\ell}$ ; il
est donc contenu dans un domaine analytique compact
et $\Gamma$-strict $U_{jm}$ de $V_j\cap F_{i\ell}$ qui
lui-même n'en rencontre que la composante 
$F_{jm}$. Le $c$-squelette $\Sigma_{jm}$ est une partie
$c$-linéaire par morceaux de $((f_j)_{|U_{jm}})^{-1}
(S_{d_i})$, et $\Sigma_{jm}\cap W_\ell$ est dès lors
une partie $c$-linéaire par morceaux de 
$((f_j)_{|U_{jm}\cap W_\ell})^{-1}
(S_{d_i})$.

Ceci valant pour tout $(j,m)\in \mathscr J(\ell)$, il résulte du lemme
\ref{squel-elem-gloc}
que $\Sigma\cap W_\ell$ est un $c$-squelette de 
$F_{i\ell}$. Comme $\Sigma_{i\ell}$ en est une partie
$c$-linéaire par morceaux compacte, il existe un domaine
analytique $\Gamma$-strict et compact
$W'_\ell$ de $W_\ell$ tel que $\Sigma_{i\ell}=\Sigma
\cap W'_\ell$.

On a dès lors par construction
$\Sigma_i=\left(\coprod_\ell W'_\ell\right)
\cap \Sigma$, ce qui achève la démonstration. 
\end{proof}

\begin{defi}
Soit $X$ un espace $k$-analytique $\Gamma$-strict. Nous dirons qu'un
$c$-squelette $\Sigma$ de $X$ est
\textit{classique} si $\Sigma$ possède un G-recouvrement
par des $c$-squelettes élémentaires. 
Nous noterons $\scl X$ l'ensemble des $c$-squelettes classiques
$X$.\end{defi}

\begin{prop}\label{recap-squelettes-classiques}
Soit
$X$ un espace $k$-analytique $\Gamma$-strict.

\begin{enumerate}[1]
\item Toute partie de $X$ qui est G-recouverte par des éléments de
$\scl X$ appartient à $\scl X$. 
\item Si $Z$ est un
sous-espace $\Gamma$-strict de $X$, une partie
de $Z$ appartient à $\scl Z$ si et seulement si elle appartient à 
$\scl Z$. 
\item Soit $\Sigma$ un élément de 
$\sel X$,  resp. $\scl X$, et soit $\Tau$ un sous-ensemble de $\Sigma$.
Alors $\Tau$
appartient à $\sel X$, resp. $\scl X$, si et seulement
si c'est une partie $c$-linéaire par morceaux de
$\Sigma$. 
\item L'intersection d'une famille finie d'éléments de 
$\scl X$ appartient à $\scl X$. 

\item Soit $\Sigma$ appartenant à $\scl X$. Il existe un G-recouvrement $(\Sigma_i)$ de
$\Sigma$ (qu'on peut prendre égal à $\{\Sigma\}$ si $\Sigma$ est élémentaire) 
et pour tout $i$ un entier $d_i$ et un sous-espace $Z_i$ de $X$ purement de dimension $d_i$ tel que $\Sigma_i\subset Z_i$ et $d_k(x)=d_i$ 
pour tout $x\in \Sigma_i$. 

\item Pour tout $\Sigma\in \scl X$ l'application $x\mapsto d_k(x)$ de $\Sigma$ dans $\N$ est localement constante. 

\item Soit $f\colon Y\to X$ un morphisme d'espaces $k$-analytiques
$\Gamma$-stricts. Pour tout élément $\Sigma$ de $\sel X$,
resp. $\scl X$, 
tel que $f^{-1}(x)$ soit vide ou de dimension nulle quel que soit
$x\in \Sigma$,
l'image réciproque $f^{-1}(\Sigma)$ appartient à $\sel Y$, resp. $\scl Y$.

\end{enumerate} 

\end{prop}

\begin{proof}
On montre chacune des assertions séparément.

\paragraph{Preuve de (1)}
Toute partie de $X$ qui est G-recouverte par des éléments de $\scl X$ est G-recouverte par des éléments
de $\sel X$ ; c'est alors un $c$-squelette
par la proposition \ref{gunion-squelettes}, et ce $c$-squelette est classique par définition.

\paragraph{Preuve de (2)}
Soit $Z$ un sous-espace de $X$, et soit
$\Sigma$ une partie de $Z$. Supposons que $\Sigma$ est un $c$-squelette classique de $Z$, 
resp. $X$, et montrons que c'est un $c$-squelette classique de $X$, resp. $Z$. 
On peut raisonner G-localement sur $\Sigma$ ce qui permet de supposer
qu'il existe un sous-espace $\Gamma$-strict 
$T$ de $Z$, resp. $X$, purement de dimension $d$ pour un certain $d$, 
et un morphisme $f\colon T\to \gma d$ tel que $\Sigma$ soit une partie 
$c$-linéaire par morceaux de $f^{-1}(S_d)$. 

Dans le premier cas, $T$ est aussi un sous-espace $\Gamma$-strict de $X$, si bien que $\Sigma$ 
appartient à $\sel X$. 

Dans le second cas, $\Sigma$ est contenu dans le sous-espace analytique fermé $Z\cap T$ de $Z$ ; comme $Z\cap T$
est de dimension $\leq d$ et que $d_k(x)=d$ pour tout $x\in \Sigma$, on a $\Sigma\subset T'$ où $T'$ désigne la réunion des
composantes irréductibles de $Z\cap T$ de dimension $d$ (munie disons de sa structure réduite); 
le couple $(T',f|_{T'})$ atteste alors que $\Sigma\in \sel X$.

\paragraph{Preuve de (3)}
Si 
$\Tau$ est un $c$-squelette (sans épithète particulière) contenu dans $\Sigma$ 
c'est automatiquement une partie $c$-linéaire par morceaux de $\Sigma$. 
Réciproquement, supposons que $\Tau$ soit une partie $c$-linéaire par morceaux de $\Sigma$. 
Si $\Sigma$ est élémentaire, il résulte immédiatement
de la définition que $\Tau$ est élémentaire. 
Il s'ensuit par raisonnement G-local que si $\Sigma$ est classique, 
$\Tau$ est classique.

\paragraph{Preuve de (4)}
Soit $(\Sigma_i)$ une famille finie de $c$-squelettes classiques. Pour tout $i$, le $c$-squelette $\Sigma_i$ est 
localement fermé dans $X$, donc fermé dans un certain ouvert $U_i$ de $X$. 
Soit $U$ l'intersection des $U_i$. Chacun des $\Sigma_i\cap U$ est une partie $c$-linéaire par morceaux de $\Sigma_i$,
donc une partie $c$-squelettique classique de $X$ d'après (3). Les $\Sigma_i\cap U$ étant fermés dans $U$, ils constituent un 
G-recouvrement de leur réunion $\Sigma$. Celle-ci est donc un $c$-squelette classique de $X$ d'après (1). 
L'intersection $\bigcap \Sigma_i=\bigcap (\Sigma_i\cap U)$ est alors une partie $c$-linéaire par morceaux de
$\Sigma$, et partant une partie $c$-squelettique classique de $X$ au vu de (3). 

\paragraph{Preuve de (5) et (6)}
L'assertion (5) est G-locale par nature, ce qui permet pour la montrer de supposer que $\Sigma$ est élémentaire, et elle
découle alors de la définition. L'assertion (6) en est une conséquence directe, compte-tenu du fait qu'une application G-localement
constante est localement constante.


\paragraph{Preuve de (7)}
Quitte à remplacer $Y$ par le lieu de dimension relative nulle de $f$
(qui en est un ouvert de Zariski), on peut supposer $f$ de dimension 
relative nulle. L'énoncé à prouver étant G-local sur $\Sigma$, 
on peut supposer ce dernier élémentaire. 
Il existe alors un un sous-espace $\Gamma$-strict 
$Z$ de $X$ purement de dimension $d$
pour un certain $d$, et un morphisme $g\colon Z\to \gma d$ tel que $\Sigma$ soit une partie 
$c$-linéaire par morceaux de $g^{-1}(S_d)$. Comme $f$ est de dimension relative nulle, $f^{-1}(Z)$
est de dimension $\leq d$, et puisque $d_k(x)=d$ pour tout $x\in \Sigma$, on a $d_k(y)=d$ pour tout 
$y\in f^{-1}(\Sigma)$ ; par conséquent $f^{-1}(\Sigma)$ est contenu dans la réunion $T$ des composantes
irréductibles de dimension $d$ de $f^{-1}(Z)$ (disons qu'on munit $T$ de sa structure réduite). 
En particulier, $f^{-1}(\Sigma)$ est contenu dans le $c$-squelette $(g\circ (f_{|T}))^{-1}(S_d)$, 
et comme
l'application $(g\circ (f_{|T}))^{-1}(S_d)\to g^{-1}(S_d)$ est $c$-linéaire par morceaux, 
$f^{-1}(\Sigma)$ est une partie $c$-linéaire par morceaux de  $(g\circ (f_{|T}))^{-1}(S_d)$, et est en particulier
un $c$-squelette élémentaire de $Y$. 
\end{proof}

Nous allons maintenant nous intéresser au comportement des $c$-squelettes
classiques sous image directe
par un morphisme topologiquement propre, en commençant par traiter deux cas particuliers.

\begin{prop}\label{prop-image-plat}
Soient $n$ et $m$ deux entiers
avec $n\geq m$,
et soient $X$ et $Y$ deux espaces $k$-analytiques
$\Gamma$-stricts, purement de dimension $m$ et $n$ respectivement. Soit $f\colon Y\to X$
un morphisme 
purement
de dimension relative $n-m$
et soit $\Sigma$ un $c$-squelette classique
de $Y$
tel que $d_k(y)=n$ pour tout $y\in \Sigma$.
On suppose que $f_{|\Sigma}$ est topologiquement propre. 
L'image $f(\Sigma)$ est alors
un $c$-squelette classique de $X$
et $d_k(x)=m$ pour tout $x\in f(\Sigma)$. 

\end{prop}

\begin{proof}
L'assertion est G-locale sur $X$, ce qui
permet de supposer $X$ compact. Comme
$\Sigma\to X$ est propre, 
$\Sigma$ est quasi-compact.  Le $c$-squelette
$\Sigma$ possède un G-recouvrement 
$(\Sigma_i)_{i\in I}$ où chaque $\Sigma_i$ est une partie
compacte de $\Sigma$ de la forme $Y_i\cap \Sigma$
pour un certain domaine analytique $\Gamma$-strict et compact $Y_i$ de $Y$. 
Par quasi-compacité
de $\Sigma$ on peut supposer $I$ fini et il suffit alors de démontrer que 
$f(\Sigma_i)$ est pour tout $i$ un $c$-squelette
classique de $X$ dont tout point
$x$ satisfait l'égalité $d_k(x)=m$. On peut donc
supposer que $Y$ et $\Sigma$
sont compacts.

L'assertion à prouver est insensible aux 
phénomènes de nilpotence, ce qui permet de supposer
$X$ et $Y$ réduits. 
Pour tout $y\in Y$ tel que $d_k(y)=n$
on a $d_k(f(y))=m$ d'après le lemme 
7.1 de \cite{ducros2021a}, si bien que 
$f$ est plat en $y$ par le théorème 10.3.7
de \cite{ducros2018} (on utilise ici le fait que $X$
est réduit) ; ceci vaut
en particulier pour tout $y\in \Sigma$. 
Le lieu de platitude
de $f$ est un ouvert de Zariski de $Y$ qui contient
par ce qui précède 
le $c$-squelette compact $\Sigma$. 
Quitte à remplacer $Y$ par un voisinage analytique compact et 
$\Gamma$-strict de $\Sigma$, on peut
donc supposer $f$ plat. 
Le $c$-squelette $\Sigma$ possède par compacité un recouvrement fini
par des $c$-squelettes élémentaires, et il suffit de prouver que l'image de chacun d'eux
sur $X$ est un $c$-squelette classique (on a déjà signalé
que $d_k(x)=m$ pour tout $x\in f(\Sigma)$). 
Ceci permet de supposer que $\Sigma$ est élémentaire. Il existe donc un domaine
analytique $\Gamma$-strict $V$ de $Y$, un sous-espace analytique fermé
$Z$ de $V$ (qu'on peut supposer réduit) purement de dimension $n$ et un morphisme
$g\colon Z\to \gma n$  tel que $\Sigma\subset
g^{-1}(S_d)$. Comme $Z$ est purement de dimension $n$, c'est une union de composantes irréductibles
de $V$. Puisque $d_k(y)=n$
pour tout $y\in \Sigma$, chaque point de $\Sigma$ n'est situé que sur une composante irréductible de $V$. 
Il existe donc un domaine analytique compact et $\Gamma$-strict $V'$ de $V$ contenu dans $Z$ et contenant
$\Sigma$ ; en remplaçant $Y$ par $V'$ on se ramène ainsi au cas où $Z=V=Y$, c'est-à-dire au cas où
il existe $g\colon Y \to \gma n$ tel que $\Sigma\subset
g^{-1}(S_n)$. 

Soit $y$ un point du $c$-squelette $\Gamma$-adique
$\Sigma\gad$ et soit $x$ son image sur $X\gad$ ; le point $y^\#$
appartient à $\Sigma$, et l'on a donc $d_k(x^\#)=m$. 
La proposition \ref{squelad-abhyankar} assure que la valuation 
$\Gamma$-colorée et bridée $y$ est d'Abhyankar de rang $n$. 
Celle-ci est induite par une valuation d'Abhyankar $\eta$ de rang $n$ sur le  corpoïde
$\hrt{y^\#}$, qui est lui-même de degré de transcendance $n$ sur $\widetilde k$. 
La valuation $x$ est induite par une valuation
$\xi$ sur le $\rpos$-corpoïde
$\hrt{x^\#}$, qui est lui-même de degré de transcendance $m$ sur $\widetilde k$. 
Le degré de transcendance résiduel gradué total de $\eta$ relativement à $\xi$ est dès lors majoré
par $n-m$, ce qui implique que le degré de transcendance résiduel de $\xi$ est lui-même minoré
par $m$, et il est finalement égal à $m$ par inégalité d'Abhyankar. Par conséquent, $\xi$ est d'Abhyankar de
rang $m$ et il en va dès lors de même de $x$. 
Il existe dès lors un voisinage affinoïde $\Gamma$-strict $X'$ de $x^\#$
dans $X$ et des fonctions analytiques
inversibles $h_1,\ldots, h_m$ sur $X'$, ainsi qu'un voisinage affinoïde $\Gamma$-strict $Y'$ de $y^\#$
et des fonctions analytiques inversibles $h_{m+1},\ldots, h_n$ sur $Y'$, tels que $h_1(x),\ldots, h_m(x)$
soit une base d'Abhyankar de $x$, et $h_1(y),\ldots, h_n(y)$ soit une base d'Abhyankar de $y$. Soient 
$\Tau$
l'image réciproque de $S_m$ sur $X'$ par $(h_1,\ldots, h_m)$ et $\Upsilon$ l'image réciproque 
de $S_n$ sur $Y'$ par $(h_1,\ldots, h_n)$. Posons $\Sigma'=\Sigma\cap Y'$. La réunion
$R=\Sigma'\cup \Upsilon$ 
est un $c$-squelette de $Y'$, et $(\Sigma')\gad$ et $\Upsilon\gad$ sont deux ouverts compacts de 
$R\gad$. Par conséquent $\Upsilon\gad \cap (\Sigma')\gad$ est un compact ouvert de $(\Sigma')\ad$,
et \textit{a fortiori} de $\Sigma\gad$. Par construction ce compact ouvert contient $y$, et son image sur $X\gad$
est contenue dans $\Tau\gad$. 

Par ce qui précède et par compacité de $\Sigma\gad$ il existe une famille finie $(\Sigma_i)$ de parties $c$-linéaires
par morceaux compactes de $\Sigma$ et, pour tout $i$, un $c$-squelette  élémentaire $\Tau_i$
de $X$ tels que les $\Sigma_i\gad$
recouvrent $\Sigma\gad$ et tels que l'image de $\Sigma_i\gad$ sur 
$X\gad$ soit contenue dans $\Tau_i\gad$ pour tout $i$. Mais alors les $\Sigma_i$ recouvrent $\Sigma$, 
et $f(\Sigma_i)\subset \Tau_i$ pour tout $i$, si bien que $f(\Sigma)$ est contenu dans la réunion des 
$\Tau_i$, qui est un $c$-squelette G-localement élémentaire de $X$ en vertu du lemme
\ref{squel-elem-gloc}. L'application de $\Sigma$ dans $\bigcup \Tau_i$ est automatiquement $c$-linéaire
par morceaux d'après \ref{fonctor-csquel}, et $f(\Sigma)$ est donc bien un $c$-squelette G-localement élémentaire
de $X$. 
\end{proof}

\begin{prop}\label{imsquel-qet}
Soit $f\colon Y\to X$ un morphisme quasi-étale
entre espaces $k$-analytiques $\Gamma$-stricts,
et soit $\Sigma$ un $c$-squelette classique de $Y$. Supposons que
$f_{|\Sigma}$ est topologiquement propre. L'image 
$f(\Sigma)$ est un $c$-squelette de $X$.
\end{prop}

\begin{proof}
La preuve comprend plusieurs étapes. 

\paragraph{Réduction au cas fini galoisien}
L'assertion 
est G-locale sur $X$, ce qui permet de le supposer affinoïde. 
Le $c$-squelette $\Sigma$ est alors quasi-compact, donc est réunion finie
de $c$-squelettes $\Sigma_i$ compacts, chaque
$\Sigma_i$ étant contenu dans un domaine analytique compact et $\Gamma$-strict
$Y_i$ de $Y$. Quitte à remplacer $Y$ par $\coprod Y_i$ et $\Sigma$ par $\coprod \Sigma_i$
(ce qui ne modifie pas $f(\Sigma)$), on peut supposer que $Y$ et $\Sigma$ sont compacts. 

Soit $x$ un point de $X$. 
Puisque $f$ est quasi-étale et puisque $Y$ est compact, il existe un voisinage affinoïde
connexe $\Gamma$-strict $V$ de $x$ dans $X$ et un recouvrement fini de $f^{-1}(V)$ par une famille
$V_i$ de domaines affinoïdes $\Gamma$-stricts de $Y$ tels que chacun des $V_i$ s'identifie, 
comme espace $V$-analytique, à un domaine affinoïde d'un $V$-espace fini étale connexe $W_i$. 
Soit $W\to V$ un revêtement fini galoisien connexe déployant tous les $W_i$. Chacun des 
$V_i\times_X W$ s'identifie alors comme
$W$-espace à une union disjointe $\coprod W_{ij}$ de domaines affinoïdes $\Gamma$-stricts 
de $W$. Pour tout $(i,j)$, l'image réciproque $\Sigma_{ij}$
de $\Sigma\cap V_i$ sur $W_{ij}$ en est un $c$-squelette
classique par la proposition \ref{recap-squelettes-classiques} (7), 
et il suffit de démontrer que l'image sur $V$ de la réunion des $\Sigma_{ij}$ 
(chacun étant vu comme un $c$-squelette de $W_i\subset W$) 
est un $c$-squelette de 
$V$. En remarquant que
la réunion des $\Sigma_{ij}$
est elle-même un $c$-squelette
classique de $W$ en vertu
de la proposition  la proposition \ref{recap-squelettes-classiques} 
(1),
on voit qu'on peut finalement
remplacer $X$ par $V$,
$Y$ par $W$ et $\Sigma$ par $\bigcup \Sigma_{ij}$,
et par là se ramener au
cas où $X$ et $Y$ sont connexes et
où $Y\to X$ est un revêtement fini galoisien.
Soit $G$ le groupe de Galois de $Y$ sur $X$. 
Pour tout $g$, l'image $g(\Sigma)$ est un $c$-squelette
classique fermé de $Y$, et
$\bigcup_{g\in G}g(\Sigma)$ est dès lors lui-même
un $c$-squelette classique de $Y$
d'après la proposition 
\ref{recap-squelettes-classiques} 
(1). 
On peut donc remplacer
$\Sigma$ par $\bigcup_{g\in G}g(\Sigma)$ (ça ne change pas son
image sur $X$) et ainsi le supposer stable sous l'action
de Galois. Le groupe $G$ opère alors
sur $\Sigma$ par automorphismes
$c$-linéaires par morceaux.

\paragraph{} Le théorème
\ref{quotient-fini-pl}
assure l'existence d'un recouvrement fini $(\Tau_\ell)$ de $\Sigma$
tel que $f_{|\Tau_\ell}$ soit injectif pour tout $\ell$. Quitte à raffiner
$(\Tau_\ell)$, on peut supposer que pour tout $\ell$ il existe un domaine
analytique compact et $\Gamma$-strict $Y_\ell$ de $Y$
contenant $\Tau_\ell$ et un morphisme
$h_\ell \colon Y_\ell \to \gma {n_\ell}$ pour un certain $n_\ell$ tel que 
$\abs{h_\ell}$ induise un isomorphisme $c$-linéaire par morceaux entre $\Tau_\ell$
et un $c$-polytope de $\rposn {n_\ell}$. Notons $h_{\ell 1},\ldots, h_{\ell n_\ell}$
les fonctions inversibles composantes du morphisme $h_\ell$.

Fixons $\ell$. Soit $y\in Y_\ell\gad$ et soit $x$ son image sur $X$. Le corps 
$\hr {y^\#}$ est une extension finie galoisienne de $\hr {x^\#}$ de degré 
divisant le cardinal $m$ de $G$  ;  par conséquent, $\abs{\hr y^\times}/\abs{\hr x^\times}$ divise
$m$. Il existe donc un domaine analytique compact et $\Gamma$-strict $U$
de $X$ tel que $x\in U\gad $ et $n_\ell$ fonctions analytiques inversibles 
$\lambda_1,\ldots, \lambda_{n_\ell}$ sur $U$ telles que 
$\abs{h_{\ell i}^m(y)}=\abs{\lambda_i(y)}$ pour tout $i$. Soit
$\Omega$ le lieu de validité simultanée 
sur $Y_\ell \times_XU$ des égalités $\abs{h_{\ell i}^m}
=\abs{\lambda_i}$. C'est un domaine analytique compact et $\Gamma$-strict 
de $Y_\ell$, et $\Omega\gad $ contient $y$ par construction. 
Par ce qui précède et par
compacité de $Y_\ell\gad$ il existe un recouvrement fini 
$(\Omega_s)_s$ de $Y_\ell$ par des domaines analytiques compacts et $\Gamma$-stricts
et, pour tout $s$, une famille $(\lambda_{s1},\ldots, \lambda_{sn_\ell})$ de fonctions analytiques
inversibles sur $f(\Omega_s)$ tels que $\abs{h_{\ell i}^m}=\abs{\lambda_{si}}$ sur
$\Omega_s$ pour tout $s$. Les $\abs{\lambda_{si}}$ induisent alors un isomorphisme
$c$-linéaire par morceaux entre $\Tau_\ell\cap \Omega_s$ et un $c$-polytope de $\rposn {n_\ell}$. 

Quitte à raffiner le recouvrement $(\Tau_\ell)$ on peut donc supposer que chacune des $h_{\ell i}$
provient d'une fonction analytique inversible sur $f(Y_\ell)$, encore notée $h_{\ell i}$.

Le compact $\bigcup_g g(\Tau_\ell)$ est une partie $c$-linéaire par
morceaux du $c$-squelette $\Sigma$, et est donc de la forme
$Y'\cap \Sigma$ pour un certain domaine analytique compact et $\Gamma$-strict
$Y'$ de $Y$. L'image $X'$ de $Y'$ sur $X$
est un domaine analytique compact et $\Gamma$-strict de
$X$. On a alors $\bigcup_gg( \Tau_\ell)=\Sigma\cap\bigcup_g g(Y')$, si bien que quitte
à remplacer $Y'$ par $\bigcup_{g\in G}g(Y')$ on peut supposer
que $Y'=f^{-1}(X')$. Comme on a par ailleurs 
$\bigcup_g g(\Tau_\ell)=f^{-1}(f(\Tau_\ell))$ et $\Sigma=f^{-1}(f(\Sigma))$, 
il vient
$f(\Tau_\ell)=X'\cap f(\Sigma)$. 
Autrement dit, lorsqu'on munit $f(\Sigma)$ de sa structure domaniale
induite par la structure analytique $\Gamma$-stricte de $X$, le compact $f(\Tau_\ell)$
est un domaine de $f(\Sigma)$.

\paragraph{}
Pour montrer que $f(\Sigma)$ est un $c$-squelette
de $X$, il suffit par ce qui précède 
de montrer que chacun des $f(\Tau_\ell)$ en est un. Fixons donc $\ell$
et écrivons $\Tau, n$ et $h$ au lieu de $\Tau_\ell, n_\ell$ et $h_\ell$. 
Et désignons désormais par $h_1,\ldots, h_n$ les composantes de $h$. 
Les $h_i$ sont des fonctions inversibles définies sur un domaine
analytique $\Gamma$-strict de $X$ contenant $f(\Tau)$. 

Il résulte alors du fait que $\Tau$ est un $c$-squelette de $Y$ 
et du choix des $h_i$ que les conditions du lemme
\ref{critere-squelette} sont satisfaites, avec $f(\Tau)$ dans le rôle de
$\Sigma$, avec $\mathscr O_X(f(\Tau))^\times$ dans celui de $E$, 
et $(h_1,\ldots, h_n)$ dans celui de $(f_1,\ldots, f_n)$. Par conséquent, 
$f(\Tau)$ est un $c$-squelette de $X$. 
\end{proof}

Nous allons maintenant en venir au théorème principal de cet article. 

\begin{theo}\label{image-squelette-glocalelem}
Soit $f\colon Y\to X$ un morphisme 
entre espaces $k$-analytiques $\Gamma$-stricts,
et soit $\Sigma$ un $c$-squelette classique 
de $Y$. Supposons que
$f_{|\Sigma}$ est topologiquement propre. 
\begin{enumerate}[1]
\item L'image 
$f(\Sigma)$ est un $c$-squelette de $X$.

\item Il existe un G-recouvrement $(\Sigma_i)$ de $f(\Sigma)$ par des parties 
$c$-linéaires par morceaux compactes et, pour tout $i$, un entier $\delta_i$
tel que $d_k(x)=\delta_i$ pour tout $x\in \Sigma_i$, et une famille décroissante $(V_{ij})_{j\geq 0}$ de 
domaines analytiques compacts et $\Gamma$-stricts de $X$, vides pour $j$ assez grand, 
tels que $\Sigma_i\subset V_{i0}$ et tel que $\Sigma_i\cap (V_{ij}\setminus V_{i,j+1})$
soit pour tout $j$ contenu dans un fermé de Zariski purement de dimension $\delta_i$
de $V_{ij}\setminus V_{i,j+1}$.

\item La fonction $x\mapsto d_k(x)$ de $f(\Sigma)$ dans $\N$ est localement constante. 

\item L'ensemble $\Sigma'$ des 
points $x$ de $f(\Sigma)$ tels que $d_k(x)=\dim_x X$ est un ouvert fermé
de $f(\Sigma)$ et c'est un $c$-squelette classique. 

\end{enumerate}
\end{theo}

\begin{proof}

On peut raisonner G-localement sur $X$, ce qui permet de le supposer compact. 
L'assertion à prouver étant insensible aux phénomènes de nilpotence, on peut supposer $X$ réduit. 
Dans ce cas $\Sigma$ est quasi-compact, si bien qu'il existe une famille finie
$(U_i)$ de domaines analytiques compacts et 
$\Gamma$-stricts de $Y$ et, pour tout $i$, un $c$-squelette élémentaire 
compact $\Sigma_i$ de $U_i$ tel que $\Sigma=\bigcup_i \Sigma_i$. 
Quitte à restreindre les $U_i$ on peut supposer qu'il existe pour tout $i$
un sous-espace analytique 
fermé $Y_i$ de $U_i$, purement de dimension $d_i$ pour un certain $i$, et un morphisme 
$\phi_i\colon Y_i\to \gma{d_i}$ tel que $\Sigma_i\subset \phi_i^{-1}(S_{d_i})$. 
Chacun des $\Sigma_i$ est réunion disjointe de ses intersections
$\Sigma_{ij}$ avec les différentes composantes irréductibles $Y_{ij}$ de $Y_i$. 
On peut alors pour démontrer le théorème remplacer 
$Y$ par $\coprod Y_{ij}$ (chaque $Y_{ij}$ étant muni 
de sa structure réduite) et $\Sigma$ par $\coprod \Sigma_{ij}$. 
On peut ainsi finalement (en renumérotant tout avec un seul indice)
supposer que $Y$ s'écrit comme une somme disjointe 
$\coprod Y_i$ où chaque $Y_i$ est un espace irréductible
et réduit dont on note $d_i$ la dimension et
qui est muni d'un morphisme $\phi_i\colon Y_i\to \gma{d_i}$ et d'une partie
$c$-linéaire par morceaux compacte $\Sigma_i$ de 
$\phi_i^{-1}(S_{d_i})$, et
où $\Sigma=\coprod \Sigma_i$.

Fixons $i$. Notons $e_i$ la dimension 
générique de $f|_{Y_i}$. 
Il résulte du théorème 4.3 de \cite{ducros2026} qu'il existe : 
\begin{itemize}[label=$\diamond$] 
\item un espace $k$-analytique $\Gamma$-strict et compact $Z_i$, muni d'un morphisme $Z_i\to X$ qui est
composé de morphismes quasi-étales et d'éclatements, et est en particulier quasi-étale
en dehors d'un diviseur de Cartier effectif $E_i$ ; 
\item un domaine analytique compact et $\Gamma$-strict $V_i$ de $Y_i\times_X Z_i$, dont on note
$V'_i$ la transformée stricte relative à $E_i$, c'est-à-dire l'adhérence réduite de $V_i \setminus Y_i \times_X E_i$ dans $V_i$, 
\end{itemize}
tels que les propriétés suivantes soient satisfaites : 
\begin{enumerate}[i]
\item $V'_i$ est purement  de dimension $d_i$ ; 
\item $V'_i\to Y_i$ est surjectif et génériquement quasi-étale ; 
\item $V'_i \to Z_i$ se factorise par un morphisme surjectif et plat 
$V'_i\to F_i$, dont les fibres sont purement de dimension $e_i$, où $F_i$ est un sous-espace
analytique fermé de $Z_i$ réduit et purement de dimension $d_i-e_i$, transverse à $E_i$. 
\end{enumerate}
Soit $\Tau$ l'image réciproque de $\Sigma_i$ sur $V_i$. Comme $d_k(y)=d_i$ pour tout
$y\in \Sigma_i$, on a $d_k(v)\geq d_i$ pour tout $v\in \Tau_i$. Combiné au fait
que $V'_i$ est de dimension $d_i$, ceci entraîne que $d_k(v)=d_i$ pour tout $v\in \Tau_i$, 
puis que $\Tau_i$ est contenu dans le lieu quasi-étale $V''_i$ de $V'_i\to Y_i$ (qui 
est Zariski-dense dans $V'_i$). Le morphisme $V''_i\to Y_i$ étant de dimension relative nulle, il résulte 
de la proposition \ref{recap-squelettes-classiques} (7) 
que $\Tau_i$ est un $c$-squelette élémentaire de $V''_i$, et partant de $V'_i$. 
La proposition \ref{prop-image-plat}
assure alors que l'image $\Upsilon_i$ de $\Tau_i$ sur $F_i$ est un $c$-squelette classique
de $F_i$, et que $d_k(t)=d_i-e_i$ pour tout 
$t\in \Upsilon_i$. Comme
$F_i$ est purement de dimension $d_i-e_i$ et est transverse à $E_i$, le $c$-squelette
$\Upsilon_i$ ne rencontre pas $E_i$ ; il existe donc un voisinage analytique compact et $\Gamma$-strict $Z'_i$
de $\Upsilon_i$ dans $Z_i$ qui ne rencontre pas $E_i$, et est dès lors quasi-étale sur $X$. 
L'image $f_i(\Sigma_i)$ est alors égale à l'image de $\Upsilon_i$ sur $X$. 

On déduit de ce qui précède que $\Sigma$ est égale à l'image de $\coprod \Upsilon_i$ par le morphisme 
quasi-étale $\coprod Z'_i\to X$. En vertu du théorème \ref{imsquel-qet}, cette image
est un $c$-squelette de $X$, 
ce qui montre (1). 

Fixons $i$. Par ce qui précède, on a $d_k(x)=d_i-e_i$ pour tout 
$x$ appartenant à l'image de $\Upsilon_i$ par $Z'_i\to X$. 
De plus,  il découle de \cite[théorème 5.3]{ducros2026}
qu'il existe une suite décroissante finie $(X=X_0\supset X_1\supset \ldots \supset X_m=\emptyset)$ 
de domaines analytiques compacts et $\Gamma$-stricts de $X$
telle que l'intersection de l'image de $F_i\cap Z'_i$
sur $X$
avec $X_j\setminus X_{j+1}$ soit pour tout $j$
un fermé de Zariski $G_j$ de $X_j\setminus X_{j+1}$
purement de dimension $d_i-e_i$ ; l'image de $\Upsilon_i$ sur $X$
est alors contenue dans la réunion des $G_j$, ce qui termine de montrer (2). 
L'assertion (3) s'en déduit  aussitôt compte-tenu du fait qu'une application G-localement constante est localement
constante.

Montrons (4). Soit $x\in f(\Sigma)$. Supposons tout d'abord que $x$ est situé sur une seule composante irréductible de $X$. 
Alors $x$ possède un voisinage ouvert $U$ purement de dimension $\dim_x X$, et $\Sigma'\cap U$ est donc l'ensemble
des points $\xi$ de $f(\Sigma)\cap U$ tels que $d_k(\xi)=\dim_x X$, qui d'après (3) est à la fois ouvert et fermé dans
$f(\Sigma)\cap U$. 
Supposons maintenant que $x$ soit situé sur au moins deux composantes irréductibles de $X$. Pour tout composante irréductible
$Z$ de $X$ contenant $x$ le point $x$ est alors situé sur un fermé de Zariski strict de $Z$ (puisqu'il appartient à au moins une autre 
composante irréductible de $X$), 
si bien que $d_k(x)<\dim Z$. Ceci valant pour tout 
$Z$ il résulte de (3) que $d_k(\xi)=d_k(x)<\dim_\xi X$ pour tout $\xi$ suffisamment proche de $x$ sur $f(\Sigma)$ ; le point $x$ possède donc un 
voisinage ouvert dans $f(\Sigma)$ qui ne rencontre pas $\Sigma'$. 

Il découle de ce qui précède que $\Sigma'$ est un ouvert fermé de $f(\Sigma)$ ; montrons maintenant que c'est un $c$-squelette
classique. Chaque point de $\Sigma'$ étant situé sur une unique composante irréductible de $X$, il existe un domaine analytique 
compact et $\Gamma$-strict $X'$ de $X$ tel que $X'\cap \Sigma=\Sigma'$ 
cet tel que $X=\coprod X'_a$ où les $X_a$ sont des domaines analytiques compacts, 
$\Gamma$-stricts et équidimensionnels
de$X$. Il suffit alors de démontrer que $\Sigma'\cap X'_a$ est classique pour tout $a$. 
En remplaçant $X$ par $X_a$ (et chacun des espaces $Y_i, Z_i$, etc. introduits ci-dessus par leur produit fibré avec $X$)
on se ramène au cas où $X$ est purement de dimension $n$ pour un certain $n$, et où $d_k(x)=n$ pour tout $x\in \Sigma$. 
Les espaces $Z_i$ sont alors tous eux aussi purement de dimension $n$. 

Soit $i$ un indice tel que $\Upsilon_i$ soit non vide. On a alors $d_i-e_i=n$, le fermé 
$F_i$ est nécessairement une réunion de composantes irréductibles de $Z_i$, et $d_k(t)=n$
pour tout $t\in \Upsilon_i$. 
Le morphisme de $F_i\cap Z'_i$ vers $X$ est alors un morphisme entre deux 
espaces compact et $\Gamma$-stricts purement de dimension $n$, et ses fibres sont de dimension nulle ; 
la proposition \ref{prop-image-plat} assure alors que $f_i(\Upsilon_i)$ est un $c$-squelette classique. 
En conséquence $\Sigma=\bigcup_i f_i(\Upsilon_i)$ est un $c$-squelette classique. 
\end{proof}

\begin{rema}
Sous les hypothèses du théorème ci-dessus, nous ne pensons pas que
le $c$-squelette $f(\Sigma)$ soit classique en général, même si exhiber
un contre-exemple semble délicat
(notre preuve montre que s'il existe des contre-exemples il en existe nécessairement avec 
$f$ quasi-étale) ; notons toutefois que $f(\Sigma)$ est classique 
d'après l'assertion (4)
lorsque $d_k(x)=\dim_x X$ pour tout $x\in f(\Sigma)$. Mais en général l'assertion (2)
(couplée à l'assertion (4)
que l'on applique aux fermés de Zariski de $V_{ij}\setminus V_{i,j+1}$ que mentionne (2)) assure simplement 
l'existence G-localement d'une \emph{stratification} finie, et non d'un G-recouvrement,
par des $c$-squelettes classiques. 

\end{rema}

\subsection{Les squelettes admissibles d'un espace de Berkovich}
Dans cette dernière partie nous nous proposons de dégager à l'aide des théorèmes qui précèdent une classe raisonnable de
squelettes, contenant les $S_n$ et possédant de bonnes propriétés de stabilité ; la définition suivante précise cette requête
un peu vague. 

\begin{defi}\label{def-classe-adm}
Une
\textit{classe admissible de $c$-squelettes} est la donnée, pour tout espace 
$k$-analytique $\Gamma$-strict $X$, d'un ensemble $\mathscr S(X)$ de $c$-squelettes de $X$, 
la collection des $\mathscr S(X)$ étant assujettie aux conditions suivantes. 

\begin{enumerate}[1]
\item Soit $X$ un espace $k$-analytique $\Gamma$-strict. 
\begin{enumerate}[b]
\item Toute partie de $X$ qui est G-recouverte par des éléments de
$\mathscr S(X)$ appartient à $\mathscr S(X)$. 
\item Si $Z$ est un
sous-espace $\Gamma$-strict de $X$, une partie
de $Z$ appartient à $\mathscr S(Z)$ si et seulement si elle appartient à 
$\mathscr S(X)$. 
\item Si $\Sigma$ est un élément de 
$\mathscr S(X)$, un sous-ensemble $\Tau$
de $\Sigma$ appartient à $\mathscr S(X)$ si et seulement
si c'est une partie $c$-linéaire par morceaux de
$\Sigma$. 
\item L'intersection d'une famille finie d'éléments de 
$\mathscr S(X)$ appartient à $\mathscr S(X)$. 

\end{enumerate}

\item Soit $f\colon Y\to X$ un morphisme d'espaces $k$-analytiques
$\Gamma$-stricts. 
\begin{enumerate}[b]
\item Pour tout élément $\Sigma$ de $\mathscr S(X)$
tel que $f^{-1}(x)$ soit vide ou de dimension nulle quel que soit
$x\in \Sigma$,
l'image réciproque $f^{-1}(\Sigma)$ appartient à $\mathscr S(Y)$. 

\item Pour tout élément $\Tau$ de $\mathscr S(Y)$ tel que 
$f_{|\Tau}$ soit topologiquement propre, l'image $f(\Tau)$
appartient à $\mathscr S(X)$. 
\end{enumerate}

\end{enumerate} 
\end{defi}

\begin{enonce}[remark]{Notation}\label{notation-sacc}
Soit $X$ un espace  $k$-analytique $\Gamma$-strict. Nous noterons 
$\sacc X$ l'ensemble des parties localement fermées $\Sigma$ de $X$
telles qu'il existe : 
\begin{itemize}[$\diamond$]
\item un G-recouvrement $(\Sigma_i)$ de $\Sigma$ par des parties
compactes et localement fermées dans $\Sigma$ (cette dernière condition est automatique
si $\Sigma$ est séparée) ; 
\item pour tout $i$, un espace $k$-analytique $\Gamma$-strict compact $Y_i$, un morphisme 
$f_i$ de $Y_i$ vers un domaine analytique $\Gamma$-strict compact 
$X_i$ de $X$
et un $c$-squelette
élémentaire
compact $\Tau_i$ de $Y_i$ tel 
que
$\Sigma_i=f_i(\Tau_i)$. 
\end{itemize}
\end{enonce}

\begin{theo}\label{theo-accessibles}
La collection des $\sacc X$ pour $X$ variable
possède les propriétés suivantes. 

\begin{enumerate}[a]
\item $\mathscr S^{\mathrm{acc}}$ est la plus petite classe admissible de $c$-squelettes
contenant $S_d$ pour tout $d$. 
\item $\mathscr S^{\mathrm{acc}}$ 
contient tous
les $c$-squelettes classiques. 
\item Soit $X$ un espace $k$-analytique $\Gamma$-strict
et soit $\Sigma$ un élément de $\sacc X$. 
II existe un G-recouvrement $(\Sigma_i)$ de $\Sigma$ par des parties 
$c$-linéaires par morceaux compactes et, pour tout $i$, un entier $\delta_i$
tel que $d_k(x)=\delta_i$ pour tout $x\in \Sigma_i$, et une famille décroissante $(V_{ij})_{j\geq 0}$ de 
domaines analytiques compacts et $\Gamma$-stricts de $X$, vides pour $j$ assez grand, 
tels que $\Sigma_i\subset V_{i0}$ et tel que $\Sigma_i\cap (V_{ij}\setminus V_{i,j+1})$
soit pour tout $j$ contenu dans un fermé de Zariski purement de dimension $\delta_i$
de $V_{ij}\setminus V_{i,j+1}$. 

\item Soit $X$ un espace $k$-analytique $\Gamma$-strict
et soit $\Sigma$ un élément de $\sacc X$. L'application $x\mapsto d_k(x)$ de $\Sigma$ vers $\N$
est localement constante. L'ensemble $\Sigma'$ des 
points $x$ de $\Sigma$ tels que $d_k(x)=\dim_x X$ est un ouvert fermé
de $\Sigma$ et c'est un $c$-squelette classique. 

\end{enumerate}
\end{theo}

\begin{enonce}[remark]{Commentaires}
Nous avons opté pour la notation $\mathscr S^{\mathrm{acc}}$
car nous proposons d'appeler les éléments de 
$\sacc X$ les $c$-squelettes \textit{accessibles}
de $X$. 
\end{enonce}

\begin{proof}
Dans toute la démonstration nous utiliserons les différents axiomes imposés
aux classes admissibles de $c$-squelettes en faisant simplement référence à leur numéro
dans la définition \ref{def-classe-adm}, de (1a) à (2b).

\paragraph{La classe $\mathscr S^{\mathrm{acc}}$ contient les
$S_d$}
Soit $d$ un entier. Pour tout $R>1$ dans 
$\Gamma\cdot \abs{k^\times}$, soit $V_R$ le domaine
affinoïde $\Gamma$-strict de $\gma d$ défini par les inégalités
$R^{-1}\leq \abs{T_i}\leq R$ pour $1\leq i\leq d$, et soit
$\Sigma_R$ l'intersection de $S_d$ avec $V_R$. Par définition, 
$\Sigma_R$ est pour tout $R$ un $c$-squelette élémentaire compact 
de $V_R$ ; le G-recouvrement de $S_d$ par les
$\Sigma_R$ atteste alors que $S_d\in \sacc{\gma d}$. 
\paragraph{La minimalité}
Soit 
$\mathscr S$ une classe admissible de $c$-squelettes contenant $S_d$ pour tout $d$. 
Soit $X$ un espace $k$-analytique, soit
$Y$ un sous-espace $\Gamma$-strict de $X$ purement de dimension $d$ et soit $f$ un morphisme de $Y$ vers
$\gma d$. Si $x$ appartient à $S_d$ alors $d_k(x)=d$, ce qui force
la fibre $f^{-1}(x)$ à être
vide ou de dimension nulle puisque
l'on a 
$d\geq d_k(y)=d_{\hr x}(y)+d$ pour tout $y\in f^{-1}(x)$. 
Au vu de l'axiome (2a), $f^{-1}(S_d)$ appartient à 
$\mathscr S(Y)$, puis à $\mathscr S(X)$ au vu de l'axiome
(1b). 
Il s'ensuit au vu de (1c) que
$\mathscr S(X)$ contient les 
$c$-squelettes élémentaires de $X$, et partant les
$c$-squelettes classiques
de $X$
en vertu de (1a). 

Si $f$ est un morphisme d'un espace $k$-analytique $\Gamma$-strict compact 
$Y$ vers un domaine analytique
$\Gamma$-strict compact $W$ de $X$ et si $\Sigma$ est un $c$-squelette élémentaire
compact de $Y$ alors
$\Sigma$ appartient à 
$\mathscr S(Y)$ par ce qui précède, 
et $f(\Sigma)$ appartient donc
à $\mathscr S(W)$ par l'axiome (2b),
puis à $\mathscr S(X)$ en vertu de (1b). Il
découle alors de (1a) que tout élément de
$\sacc X$ appartient à $\mathscr S(X)$. 

Il suffit donc maintenant de montrer que $\mathscr S^{\mathrm{acc}}$ est une classe admissible
de $c$-squelettes satisfaisant (c) et (d);
d'après ce qu'on vient de voir ce sera alors
la plus petite classe admissible
de $c$-squelettes contenant les $S_d$, 
et elle contiendra de surcroît tous les $c$-squelettes classiques.

\paragraph{Les éléments de $\sacc X$ sont des $c$-squelettes
satisfaisant (c) et (d)}
Soit $X$ un  espace $k$-analytique $\Gamma$-strict
et soit
$\Sigma$
que un élément de $\sacc X$. Montrons $\Sigma$ est un $c$-squelette
satisfaisant (c) et (d). 
On peut raisonner G-localement sur $X$, 
ce qui permet de supposer que $X$ et $\Sigma$ sont compacts. Il existe alors
une famille
finie $(Y_i)$ d'espaces analytiques compacts et 
$\Gamma$-stricts et, pour tout $i$, un $c$-squelette
élémentaire $\Tau_i$ de $Y_i$
et un morphisme $f_i\colon Y_i\to X$ tels que $\Sigma=\bigcup_i f_i(\Tau_i)$. 
Mais alors $\Sigma$ est l'image du $c$-squelette G-localement élémentaire 
compact 
$\coprod \Tau_i$ de $\coprod Y_i$ par le morphisme naturel 
de $\coprod Y_i$ vers $X$, et il résulte du théorème \ref{image-squelette-glocalelem}
que $\Sigma$ est un 
$c$-squelette 
satisfaisant (c) et (d).

\paragraph{Vérification de (1a)}
Soit $X$ un espace $k$-analytique $\Gamma$-strict et soit $\Sigma$ une 
partie localement fermée de $X$. Supposons que $\Sigma$ possède un G-recouvrement par 
des parties appartenant à $\sacc X$. Il découle alors tautologiquement de la définition de $\sacc X$
que $\Sigma$ appartient à $\sacc X$.

\paragraph{Vérification de (1b)}
Soit $X$ un espace $k$-analytique
et soit $Z$ un sous-espace
$\Gamma$-strict de $X$. Soit $\Sigma$
une partie de $Z$. Supposons
que $\Sigma$ appartient à $\sacc Z$
(resp. $\sacc X$). 
Il existe alors 
un G-recouvrement $(\Sigma_i)$ 
de $\Sigma$ par des parties 
compactes et localement fermées et, pour tout $i$, 
un espace $k$-analytique compact et $\Gamma$-strict 
$Y_i$, un $c$-squelette élémentaire compact $\Tau_i$
de $Y_i$ et un morphisme $f_i$ de $Y_i$ vers un domaine
analytique compact et $\Gamma$-strict $Z_i$
(resp. $X_i$) de $Z$ (resp. $X$) tel que $\Sigma_i
=f_i(\Tau_i)$. 

Dans le premier cas, on choisit
pour tout $i$ un recouvrement
fini $(Y_{ij})$ de chacun des $Y_i$ par des domaines analytiques
compacts et $\Gamma$-stricts tels que
$f_i(Y_{ij})$ soit contenu pour tout $(i,j)$ dans un domaine
analytique compact et $\Gamma$-strict
$X_{ij}$ de $X$. Le G-recouvrement
$(f_i(\Tau_i\cap Y_{ij}))_{i,j}$ de $\Sigma$
atteste alors que 
ce dernier appartient à 
$\sacc X$. 

Dans le second cas, on pose 
$Y'_i=f_i^{-1}(Z)$
pour tout indice $i$. 
Le sous-ensemble
$\Tau_i$ de $Y_i$ est alors un $c$-squelette
élémentaire de $Y'_i$, et $f_i(\Tau_i)$ est contenu
dans le domaine analytique compact et $\Gamma$-strict
$Z\times_X X_i$ de $Z$. Le G-recouvrement de 
$\Sigma$ par les $\Sigma_i$ atteste alors que 
$\Sigma\in \sacc Z$.

\paragraph{Vérification de (1c)}
Soit $X$ un espace
$k$-analytique $\Gamma$-strict, soit $\Sigma$
appartenant à $\sacc X$ et soit $\Tau$
une partie de $\Sigma$. Si $\Tau$ est un élément
de $\sacc X$, c'est un $c$-squelette de $X$ d'après 
ce qu'on vient de voir, et donc une partie $c$-linéaire par morceaux de $\Sigma$. 
Réciproquement, supposons que $\Tau$ soit une partie
$c$-linéaire par morceaux de $\Sigma$. Pour montrer que
$\Tau$ appartient à $\sacc X$, on peut par (1a) 
(déjà établi) raisonner
G-localement sur $\Tau$, ce qui permet
de supposer que $X,\Tau$ et $\Sigma$ sont compacts, 
et que $\Sigma$ est égal à $f(\Upsilon)$ où $f$ est un morphisme
d'un espace analytique compact et $\Gamma$-strict
$Y$ vers $X$ et où 
$\Upsilon$ est un $c$-squelette
élémentaire compact de $Y$. Mais alors 
$\Tau=f(f^{-1}(\Tau))$ ; et
comme $f^{-1}(\Tau)$ est une partie
$c$-linéaire par morceaux fermée de $\Upsilon$, c'est un
$c$-squelette élémentaire de $Y$, si bien 
que $\Tau$ est un élément de $\sacc X$.

\paragraph{Vérification de (1d)}
Soit $X$ un espace $k$-analytique compact
et soit $(\Sigma_i)$ une famille finie d'éléments de
$\sacc X$. Par définition tous les $\Sigma_i$
sont localement fermés ; choisissons pour tout $i$
un ouvert $U_i$ de $X$ dont $\Sigma_i$ est un fermé. 
Soit $U$ l'intersection des $U_i$. 
Chacun des $\Sigma_i\cap U$
est une partie $c$-linéaire par morceaux
de $\Sigma_i$, donc appartient à $\sacc X$ d'après (1c), 
puis à $\sacc U$ d'après (1b).
En vertu de (1a) la réunion $\Tau$ des $\Sigma_i\cap U$
appartient encore à $\sacc U$, puis à $\sacc X$
d'après (1b). 
Chacun des $\Sigma_i\cap U$ est
en vertu de (1c)
une partie $c$-linéaire par morceau de $\Tau$ et il en va donc
de même de leur intersection, qui n'est autre que
$\bigcap_i \Sigma_i$. En utilisant à nouveau (1c)
on en déduit que $\bigcap_i \Sigma_i$ est un
élément de $\sacc X$.

\paragraph{Vérification de (2a)}
Soit $f\colon Y
\to X$ un morphisme 
entre deux espaces 
$k$-analytiques $\Gamma$-stricts et soit $\Sigma$ un élément de $\sacc X$
tel que $f^{-1}(x)$ soit vide ou de dimension nulle pour tout $x\in \Sigma$. Nous
allons montrer que $f^{-1}(\Sigma)$ 
appartient à $\sacc Y$. 
D'après (1a) l'assertion est G-locale sur $f^{-1}(\Sigma)$, ce qui permet de supposer
$Y$ compact ; elle est \textit{a fortiori} G-locale sur 
$\Sigma$, 
ce qui permet de supposer que $X$ est compact 
et qu'il existe un espace $k$-analytique $\Gamma$-strict compact 
$Z$, un morphisme 
$g\colon Z\to X$ et un $c$-squelette élémentaire
$\Tau$ de $Z$ tel que $\Sigma=g(\Tau)$. 
Soit $\Upsilon$ l'image réciproque de
$\Tau$ sur $Y\times_XZ$. C'est
d'après la proposition \ref{recap-squelettes-classiques}
(7)
un $c$-squelette 
élémentaire de $Y\times_X Z$, 
et son image sur $Y$ est égale à $f^{-1}(\Sigma)$. 
Comme $Y\times_XZ$ est compact, $f^{-1}(\Sigma)$ appartient
à $\sacc Y$.

\paragraph{Vérification de (2b)}
Soit $f\colon Y
\to X$ un morphisme entre espaces $k$-analytiques 
$\Gamma$-stricts, 
et soit $\Tau$ un élément de $\sacc Y$ 
tel que $f_{|\Tau}$ soit topologiquement propre. 
Nous allons montrer que 
$f(\Tau)$ appartient à $\sacc X$. D'après (1a), 
l'assertion est 
G-locale sur $X$, ce qui permet de supposer que
$X$ est compact. Dans ce cas $\Tau$ est quasi-compact, et il existe alors un
recouvrement fini $(\Tau_i)$ de $\Tau$
par des parties $c$-linéaires compactes et, pour tout $i$, 
un domaine $k$-analytique et $\Gamma$-strict $Y_i$ de 
$Y$ contenant $\Tau_i$. On a $f(\Tau)=\bigcup f(\Tau_i)$, ce qui permet
de remplacer $Y$ par $\coprod Y_i$ et $\Tau$ par $\coprod \Tau_i$ 
(que cette opération préserve l'hypothèse que $\Tau$ appartient à 
$\sacc Y$ résulte de (1a), (1b) et (1c)), et donc de supposer
que $Y$ et $\Tau$ sont compacts. 

Par définition de $\sacc Y$ et par compacité
de $\Tau$ il existe alors une famille finie 
$(Z_j)$ d'espaces analytiques compacts et $\Gamma$-stricts et pour tout $j$ 
un $c$-squelette élémentaire compact $\Upsilon_j$ de 
$Z_j$ et un morphisme
$f_j\colon Z_j\to X$ 
 tel que $\Tau=\bigcup f_j(\Upsilon_j)$. 
Mais alors le compact $\Sigma=f(\Tau)$ est la réunion 
finie des $(f\circ f_j) (\Upsilon_j)$, et est
dès lors un élément de $\sacc X$. 
\end{proof}

\bibliographystyle{smfalpha}
\bibliography{aducros}

\end{document}